\documentclass[a4paper,11pt,reqno]{amsart}

\usepackage[utf8]{inputenc}
\usepackage[T1]{fontenc}
\usepackage{lmodern}
\usepackage[english]{babel}
\usepackage{microtype}

\usepackage{amsmath,amssymb,amsfonts,amsthm,esint}
\usepackage{mathtools,accents}
\usepackage{mathrsfs}
\usepackage{aliascnt}
\usepackage{braket}
\usepackage{bm}

\usepackage[a4paper,margin=3cm]{geometry}
\usepackage[citecolor=blue,colorlinks]{hyperref}

\usepackage{enumerate}
\usepackage{xcolor}
\makeatletter
\g@addto@macro\@floatboxreset\centering
\makeatother

\makeatletter
\def\newaliasedtheorem#1[#2]#3{
	\newaliascnt{#1@alt}{#2}
	\newtheorem{#1}[#1@alt]{#3}
	\expandafter\newcommand\csname #1@altname\endcsname{#3}
}
\makeatother

\numberwithin{equation}{section}

\newtheoremstyle{slanted}{\topsep}{\topsep}{\slshape}{}{\bfseries}{.}{.5em}{}

\theoremstyle{plain}
\newtheorem{theorem}{Theorem}[section]
\newaliasedtheorem{proposition}[theorem]{Proposition}
\newaliasedtheorem{lemma}[theorem]{Lemma}
\newaliasedtheorem{corollary}[theorem]{Corollary}
\newaliasedtheorem{counterexample}[theorem]{Counterexample}

\theoremstyle{definition}
\newaliasedtheorem{definition}[theorem]{Definition}
\newaliasedtheorem{question}[theorem]{Question}
\newaliasedtheorem{openquestion}[theorem]{Open Question}
\newaliasedtheorem{conjecture}[theorem]{Conjecture}

\theoremstyle{remark}
\newaliasedtheorem{remark}[theorem]{Remark}
\newaliasedtheorem{example}[theorem]{Example}

\newcommand{\setN}{\mathbb{N}}
\newcommand{\setQ}{\mathbb{Q}}
\newcommand{\setR}{\mathbb{R}}

\newcommand{\cC}{\mathcal{C}}

\newcommand{\cL}{\mathcal{L}}

\newcommand{\eps}{\varepsilon}

\let\altphi\phi
\let\phi\varphi
\let\varphi\altphi
\let\altphi\undefined

\newcommand{\abs}[1]{\left\lvert#1\right\rvert}
\newcommand{\norm}[1]{\left\lVert#1\right\rVert}

\newcommand{\Id}{\mathrm{Id}}

\let\div\undefined
\DeclareMathOperator{\div}{div}
\DeclareMathOperator{\Hess}{Hess}

\newcommand{\di}{\mathop{}\!\mathrm{d}}

\newcommand{\loc}{{\rm loc}}

\newcommand{\res}{\mathop{\hbox{\vrule height 7pt width .5pt depth 0pt
			\vrule height .5pt width 6pt depth 0pt}}\nolimits}

\newcommand{\restr}{\raisebox{-.1618ex}{$\bigr\rvert$}}
\DeclareMathOperator{\supp}{supp}

\newcommand{\Ch}{{\sf Ch}}

\DeclareMathOperator{\Lip}{Lip}
\DeclareMathOperator{\Lipb}{Lip_b}

\DeclareMathOperator{\lip}{lip} % %for the slope

\DeclareMathOperator{\Per}{Per}
\DeclareMathOperator{\Ric}{Ric}

\newcommand{\haus}{\mathscr{H}}
\newcommand{\Leb}{\mathscr{L}}

\newcommand{\Prob}{\mathscr{P}}

\newcommand{\XX}{{\boldsymbol{X}}}

\newcommand{\dist}{\mathsf{d}}

\newcommand{\meas}{\mathfrak{m}}

\newcommand{\Test}{{\rm Test}}

\DeclareMathOperator{\CD}{CD}
\DeclareMathOperator{\RCD}{RCD}

\newfont{\tmpf}{cmsy10 scaled 2500}

\DeclareMathOperator{\Tan}{Tan}

\def\Xint#1{\mathchoice
	{\XXint\displaystyle\textstyle{#1}}%
	{\XXint\textstyle\scriptstyle{#1}}%
	{\XXint\scriptstyle\scriptscriptstyle{#1}}%
	{\XXint\scriptscriptstyle\scriptscriptstyle{#1}}%
	\!\int}
\def\XXint#1#2#3{{\setbox0=\hbox{$#1{#2#3}{\int}$ }
		\vcenter{\hbox{$#2#3$ }}\kern-.6\wd0}}

\def\dashint{\Xint-}

\begin{document}
	
\title{Boundary regularity and stability for spaces with Ricci bounded below}
\author{Elia Bru\`{e}, Aaron Naber, and Daniele Semola}

\address{School of Mathematics, Institute for Advanced Study\\
	1 Einstein Dr.\\
	Princeton NJ 05840\\
	U.S.A.}
\email{elia.brue@math.ias.edu}

\address{Northwestern University\\
633 Clark Street\\
Evanston, IL 60208\\
U.S.A.}
\email{anaber@math.northwestern.edu}

\address{Mathematical Institute, University of Oxford\\
         Oxford OX2 6GG \\
	United Kingdom}
\email{Daniele.Semola@maths.ox.ac.uk}
	
\maketitle

\begin{abstract}
		This paper studies the structure and stability of boundaries in noncollapsed $\RCD(K,N)$ spaces, that is, metric-measure spaces $(X,\dist,\haus^N)$ with lower Ricci curvature bounded below.  Our main structural result is that the boundary $\partial X$ is homeomorphic to a manifold away from a set of codimension 2, and is $N-1$ rectifiable.  Along the way we show effective measure bounds on the boundary and its tubular neighborhoods.  These results are new even for Gromov-Hausdorff limits $(M_i^N,\dist_{g_i},p_i) \rightarrow  (X,\dist,p)$ of smooth manifolds with boundary, and require new techniques beyond those needed to prove the analogous statements for the regular set, in particular when it comes to the manifold structure of the boundary $\partial X$.
\smallskip
				
		The key local result is an $\eps$-regularity theorem, which tells us that if a ball $B_{2}(p)\subset X$ is sufficiently close to a half space $B_{2}(0)\subset \setR^N_+$ in the Gromov-Hausdorff sense, then $B_1(p)$ is biH\"older to an open set of $\setR^N_+$.  In particular, $\partial X$ is itself homeomorphic to $B_1(0^{N-1})$ near $B_1(p)$.  Further, the boundary $\partial X$ is $N-1$ rectifiable and the boundary measure  $\haus^{N-1}\res \partial X$ is Ahlfors regular on $B_1(p)$ with volume close to the Euclidean volume.  		
			
		Our second collection of results involve the stability of the boundary with respect to noncollapsed mGH convergence $X_i\to X$.  Specifically, we show a boundary volume convergence which tells us that the $N-1$ Hausdorff measures on the boundaries converge $\haus^{N-1}\res \partial X_i\to \haus^{N-1}\res \partial X$ to the limit Hausdorff measure on $\partial X$.  We will see that a consequence of this is that if the $X_i$ are boundary free then so is $X$.  
	
\end{abstract}

\tableofcontents

\section{Introduction}

This paper studies structural and stability properties for noncollapsed $\RCD(K,N)$ spaces with boundary. In particular, we give affirmative answers to some of the recent conjectures presented in \cite{DePhilippisGigli18,KapovitchMondino19}.

Most of the statements are new and of interest even for noncollapsed limits of smooth Riemannian manifolds with convex boundary and interior lower Ricci curvature bounds.
\smallskip

Our main results can be grouped into 
	\begin{itemize}
		\item structure results for boundaries and spaces with boundary;
	    \item stability/gap theorems about the absence/presence of boundary.
	  \end{itemize}  
	  
In particular, we obtain the rectifiable structure of the boundary together with measure estimates. Moreover we prove that noncollapsed $\RCD$ spaces are homeomorphic to topological manifolds (possibly with boundary) up to sets of codimension two.

On the side of stability/gap results we are going to prove that the absence of boundary is preserved under noncollapsed (pointed) Gromov-Hausdorff convergence and that the boundary volume measures converge in full generality. We also show that the presence of boundary is stable, under an additional assumption which is satisfied for sequences of smooth manifolds with boundary. 

\smallskip

Below, after briefly introducing the relevant terminology and background, we outline the main achievements of the paper.
\smallskip

The Riemannian Curvature Dimension condition $\RCD(K,\infty)$ was introduced in \cite{AmbrosioGigliSavare14} (see also \cite{AmbrosioGigliMondinoRajala15}) coupling the Curvature Dimension condition $\CD(K,\infty)$, previously proposed in \cite{Sturm06a,Sturm06b} and independently in \cite{LottVillani}, with the infinitesimally Hilbertian assumption, corresponding to the Sobolev space $H^{1,2}$ being Hilbert.\\
The natural finite dimensional refinements subsequently led to the notions of $\RCD(K,N)$ and $\RCD^*(K, N)$ spaces, corresponding to $\CD(K, N)$ (resp. $\CD^*(K, N)$, see \cite{BacherSturm10}) coupled with linear heat flow. The class $\RCD(K,N)$ was proposed in \cite{Gigli15}, motivated by the validity of the sharp Laplacian comparison and of the Cheeger-Gromoll splitting theorem, proved in \cite{Gigli13}. The (a priori more general) $\RCD^*(K,N)$ condition was thoroughly analysed in \cite{ErbarKuwadaSturm15} and (subsequently and independently) in \cite{AmbrosioMondinoSavare15} (see also \cite{CavallettiMilman16} for the equivalence betweeen $\RCD^*$ and $\RCD$ in the case of finite reference measure).\\
Several geometric and analytic properties have been proved for $\RCD(K,N)$ spaces in the last years, often inspired by the theory of (weighted) Riemannian manifolds with lower Ricci bounds and of Ricci limits. Without the aim of being complete, let us mention the heat kernel estimates \cite{JangLiZhang}, the rectifiability \cite{MondinoNaber19}, the constancy of the dimension in the almost everywhere sense \cite{BrueSemola18} (cf. with \cite{ColdingNaber13} dealing with Ricci limit spaces) and the existence of a second order differential calculus \cite{Gigli18}.

In the theory of Ricci limit spaces, further regularity properties are satisfied under the noncollapsing assumption. If the approximating sequence of smooth Riemannian manifolds, besides the lower Ricci bound
\begin{equation}
\Ric_{M_i}\ge -(N-1)\, ,
\end{equation}
verifies also the lower volume bound
\begin{equation}
\haus^{N}(B_1(p_i))\ge v>0\, ,
\end{equation}
then by volume convergence \cite{Colding97,CheegerColding97} the volume measures converge to the $\haus^N$-measure on the limit metric space.  Noncollapsed Ricci limit spaces are much more regular than general Ricci limits, see \cite{CheegerColding97,CheegerNaber13,CheegerNaber15,CheegerJiangNaber18}.

Motivated by this refinement in the theory of Ricci limits, a notion of \emph{noncollapsed} $\RCD(K,N)$ metric measure space $(X,\dist,\meas)$ has been proposed in \cite{DePhilippisGigli18} by asking that $\meas=\haus^N$ (a weaker definition had been previously suggested in \cite{Kitabeppu19}). In the same work some properties valid for noncollapsed Ricci limits have been generalized to the synthetic framework, such as the volume convergence and the stratification of the singular set. More recent contributions dealt with topological regularity \cite{KapovitchMondino19}, volume bounds for the singular strata \cite{AntonelliBrueSemola19} and differential characterizations \cite{Honda19}.

\subsection{Singular strata and boundaries}  

On a noncollapsed $\RCD(K,N)$ metric measure space $(X,\dist,\haus^N)$ any tangent cone is a metric cone (see \cite{CheegerColding96,CheegerColding97,DePhilippisGigli16,DePhilippisGigli18}). Moreover, there is a natural \textit{stratification of the singular set}
\begin{equation}
\mathcal{S}^0\subset \mathcal{S}^1\subset\dots\subset\mathcal{S}^{N-1}=\mathcal{S}:=X\setminus\mathcal{R}\, ,
\end{equation}	
where
\begin{equation}
\mathcal{R}:=\left\lbrace x\in X: \Tan_x(X,\dist)=\{(\setR^N,\dist_{\mathrm{eucl}})\}\right\rbrace 
\end{equation}	 
is the set of regular points of $(X,\dist,\haus^N)$ and, for any $0\le k\le N-1$,
\begin{equation}
\mathcal{S}^k:=\left\lbrace x\in X: \text{no tangent cone at $x$ splits off $\setR^{k+1}$}\right\rbrace\, . 
\end{equation} 
This stratification was first introduced in \cite{CheegerColding97} for noncollapsed Ricci limits. Therein it was proven that
\begin{equation}\label{w1}
\mathcal{S}^{N-1}\setminus \mathcal{S}^{N-2} = \emptyset
\end{equation}
and that the following Hausdorff dimension estimate holds: 
\begin{equation}\label{eq:hausest}
\dim_H\mathcal{S}^{k}\le k\, ,
\quad \text{ $1\le k\le N-2$}\, .
\end{equation}	
A more quantitative analysis of singular strata was initiated in \cite{CheegerNaber13}, based on quantitative differentiation arguments and yielding to Minkowski-type estimates for the \textit{quantitative singular strata}
\[
\mathcal{S}^k_{\eps,r}:= \left\lbrace x\in X:\, \text{for no $r\le s <1$ $B_s(x)$ is a $(k+1,\eps)$-symmetric ball}\right\rbrace
\]
and 
\begin{equation}\label{eq:quantstrat}
\mathcal{S}^k_{\eps}:=\bigcap_{r>0}\mathcal{S}^k_{\eps,r}\, .
\end{equation}
We recall that $B_s(x)$ is said to be a $(k,\eps)$-symmetric ball provided 
\[
\dist_{GH}(B_s(x),B_s(z)) \le s \eps\, ,
\]
where $z\in C(Z)\times \setR^k$ is a tip of the metric cone $C(Z)\times \setR^k$. We refer to \autoref{sec:preliminaries} for the precise introduction of metric cones and of the Gromov-Hausdorff distance $\dist_{GH}$.

Later on, in \cite{CheegerJiangNaber18} the estimates for the quantitative singular strata have been sharpened, and the $k$-rectifiable structure of $\mathcal{S}^k$ has been shown for any $0\le k\le N-2$.

\medskip

In the framework of $\RCD$ spaces the top dimensional singular stratum $\mathcal{S}^{N-1}\setminus\mathcal{S}^{N-2}$ is not empty in general, since Riemannian manifolds with convex boundary and lower Ricci curvature bounds in the interior belong to this class. Still, the Hausdorff dimension estimate \eqref{eq:hausest} holds for any $0\le k\le N-1$ (see \cite{DePhilippisGigli18}). 
The same phenomenon happens in the theory of Alexandrov spaces, where the top dimensional singular stratum is strictly linked to the boundary of the space \cite{Perelman91,BuragoBuragoIvanov01}. Elementary examples suggest that this is the case also for noncollapsed $\RCD$ spaces.

\smallskip

In \cite{DePhilippisGigli18} and \cite{KapovitchMondino19} two different notions of boundary for an $\RCD(K,N)$ space $(X,\dist,\haus^N)$ have been proposed. For the sake of this introduction we are going to deal with the one introduced in \cite{DePhilippisGigli18}, where the authors define 
\begin{equation}\label{eq:bddpg}
\partial X:=\overline{\mathcal{S}^{N-1}\setminus\mathcal{S}^{N-2}}\, .
\end{equation}
Above we denoted by $\overline{\mathcal{S}^{N-1}\setminus\mathcal{S}^{N-2}}$ the topological closure of $\mathcal{S}^{N-1}\setminus\mathcal{S}^{N-2}$.\\ 
Let us point out that, since the density of $\haus^N$ at any point in $\mathcal{S}^{N-1}\setminus\mathcal{S}^{N-2}$ equals $1/2$, by lower semicontinuity of the density it holds
\begin{equation}
\partial X\setminus\left(\mathcal{S}^{N-1}\setminus\mathcal{S}^{N-2}\right)\subset\mathcal{S}^{N-2}\, ,
\end{equation}
in particular 
\begin{equation}
\dim_H\left(\partial X\setminus\left(\mathcal{S}^{N-1}\setminus\mathcal{S}^{N-2}\right)\right)\le N-2\, .
\end{equation}

A comparison with the notion of boundary introduced in \cite{KapovitchMondino19} will be investigated subsequently in the paper (cf. \autoref{thm:lipbdry} (i)).

 Given the above definition of boundary it sounds natural to introduce the following.

\begin{definition}\label{def:withoutboundary}
	We say that an $\RCD(K,N)$ space $(X,\dist,\haus^N)$ has boundary in $B_1(p)$ if
	\[
	(\mathcal{S}^{N-1}\setminus \mathcal{S}^{N-2})\cap B_1(p) \neq  \emptyset\, ,
	\]	
	otherwise we say that $(X,\dist,\haus^N)$ has no boundary in $B_1(p)$.
\end{definition}

\vspace{.2cm}

\subsection{An $\eps$-regularity theorem for top dimensional singularities}

For all the subsequent developments of the paper, the building block is an $\eps$-regularity theorem, dealing with the structure of balls sufficiently close in the GH sense to a ball centered on the boundary of a half-space.

 Let us preliminarily recall that a set $E\subset X$ is said to be $(N-1)$-rectifiable provided
\[
E\subset M\cup \bigcup_{i\in \setN} E_i\, , 
\]
where $\haus^{N-1}(M)=0$ and the $E_i$ are biLipschitz to a Borel subset of $\setR^{N-1}$ for any $i\in \setN$.\\

\begin{theorem}[$\eps$-regularity]\label{thm:epsregintro}
	
Let $1\le N<\infty$ be a fixed natural number and let $\eps>0$. If $\delta\le \delta(N,\eps)$, then for any $\RCD(-\delta(N-1),N)$ m.m.s. $(X,\dist,\haus^N)$ with $p\in X$ such that
		\begin{equation}\label{eq:GHclose}
		\dist_{GH}(B_{16}(p),B_{16}^{\setR_+^N}(0)) < \delta\, ,
		\end{equation}
		it holds that $\partial X\cap B_1(p)\neq\emptyset$. Moreover
		\begin{itemize}
			\item[(i)] (Ahlfors regularity) for any $x\in\partial X\cap B_1(p)$ and for any $0<r<1$
			\begin{equation}
			(1-\eps)\omega_{N-1}r^{N-1}\le \haus^{N-1}(\partial X\cap B_r(x))\le (1+\eps)\omega_{N-1}r^{N-1}\, ;
			\end{equation}
			\item[(ii)] (Rectifiable structure) $\partial X\cap B_1(p)$ is $(N-1)$-rectifiable;
			\item[(iii)] (Topological structure) 
			there exists a map $F:B_1(p)\to \setR_+^N$ satisfying
                \begin{itemize}		
				\item[(a)] $(1-\eps) \dist(x,y)^{1+\eps} \le |F(x) - F(y)| \le C(N) \dist(x,y)$ for any $x,y\in B_1(p)$;
				\item[(b)] $F(p)=0$ and $\partial \setR_+^N \cap B_{1-2\eps}(0) \subset F(\partial X \cap B_1(p)) = \partial \setR^N_+\cap F(B_1(p))$;
				\item[(c)] $F$ is open and a homeomorphism with its image;
				\item[(d)] $ B_{1-2\eps}^{\setR_+^N}(0) \subset F(B_1(p))$.
				\end{itemize}
				
		\end{itemize}		
\end{theorem}

\begin{remark}		
In view of the volume $\eps$-regularity for the boundary \autoref{cor:volrigidityb} the conclusions of \autoref{thm:epsregintro} hold by assuming $p\in \partial X$ and the volume pinching condition
\begin{equation}
	\haus^N(B_{32}(p)) \ge \frac{1}{2} \omega_N (32)^N - \delta 
\end{equation}
in place of \eqref{eq:GHclose}.
\end{remark}

The proof of \autoref{thm:epsregintro} requires most of the tools developed in the paper and will be split into several intermediate results.
	\smallskip

One of the building blocks to prove the boundary measure estimates in \autoref{thm:epsregintro} is a weaker $\eps$-regularity theorem, \autoref{thm:stabilityS}. There we prove that there exist constants $c(N)>1$ and $\eta(N)>0$ such that if 

	\[
\dist_{GH}(B_{1}(p),B_{1}^{\setR_+^N}(0)) < \eta(N)\, ,
\]
then 
\[c(N)^{-1}\le \haus^{N-1}(\partial X\cap B_1(p))\le c(N)\, .
\]

 	Stability is a key feature of the top dimensional singular stratum. It is well known that codimension two singularities might appear even for limits of smooth manifolds.  The easiest example of this being that of a two dimensional singular cone, which can be obtained as a limit of smooth manifolds with uniform lower Ricci bounds by rounding off the tip.

  Among the other things, \autoref{thm:epsregintro} (and even its weaker version \autoref{thm:stabilityS}) implies that the top dimensional singular stratum $\mathcal{S}^{N-1}\setminus\mathcal{S}^{N-2}$ is empty for noncollapsed Ricci limits of manifolds without boundary, as known from the seminal paper \cite{CheegerColding97}. It is worth stressing that our strategy is completely different from the original one, which is based on a topological argument and seems not suitable to handle the general case of $\RCD$ spaces. 
A previous attempt in this direction has been made in \cite{KapovitchMondino19}, where the authors extended Cheeger-Colding's result to the setting of noncollapsed $\RCD$ spaces verifying an additional topological regularity assumption. In contrast, our proof is quantitative in nature and does not require any topological argument. Moreover, the statement we achieve is stronger, and new even in the smooth framework. Indeed we prove that closeness to the model boundary ball implies the presence of a definite amount of boundary points.

The Ahlfors regularity for the boundary measure in sharp form, \autoref{thm:epsregintro} (i), will be established through several steps. The key step is the improved structure theorem for boundary balls \autoref{thm:improvedneckregion}, which when combined with \autoref{thm:stabilityS} yields to Ahlfors regularity in weaker form, with a constant $c(N)>1$ and $1/c(N)$ in place of $1+\eps$ and $1-\eps$, respectively. The sharp version of the bound will be obtained later in \autoref{cor:bdryvolconvergencereg} by combining the stability of \autoref{thm:improvedneckregion} and the rectifiable and biH\"older structure of \autoref{thm:goodpar} (ii), (iii) and (iv).

	The topological regularity part of \autoref{thm:epsregintro} is new and of interest even in the case of limits of smooth Riemannian manifolds.  At its heart, the proof is based on two key points, (cf. with the proof of \autoref{thm:improvedneckregion}).  The first is the stability of \autoref{lemma:distancetodistance} which tells us that if a ball $B_r(x)$ is Gromov-Hausdorff close to a half space, then the boundary singularities $\partial X\cap B_r(x)$ must be $\eps$-close to a ball in $\setR^{n-1}\cap B_r(0^{n-1})$.  The second is a boundary volume $\eps$-regularity \autoref{cor:volrigidityb}, based in turn on \autoref{lemma:conerigidity}, which roughly tells us that if there are two balls $B_r(x)\subseteq B_R(x)$, both close to half spaces and centered at a boundary point, then the smaller ball $B_r(x)$ must be at least as close to a half-space as the larger ball $B_R(x)$.  The effect of these two results is that once boundary singularities start to appear, they cannot stop appearing and we can eventually put them together into a topological structure.

\subsection{Structure of boundaries and of spaces with boundary}

The $\eps$-regularity \autoref{thm:epsregintro}, when combined with a covering argument, yields a structural result for noncollapsed $\RCD$ spaces with boundary.

Here and throughout the paper we shall adopt the notation
\[
B_r(A):=\bigcup_{a\in A} B_r(a)
\]
to indicate the tubular neighbourhood of a set on a metric space.

\begin{theorem}[Boundary Structure]\label{thm:boundatystructure}
	Let $(X,\dist,\haus^N)$ be an $\RCD(-(N-1),N)$ space with $p\in X$ such that $\haus^{N}(B_1(p))>v>0$. Then, if $(\mathcal{S}^{N-1}\setminus \mathcal{S}^{N-2})\cap B_2(p)\neq \emptyset$ the following hold
	\begin{itemize}
		\item[(i)] (Rectifiability and volume estimates) $\partial X$ is $(N-1)$-rectifiable and
		\[
		\haus^{N-1}(B_r(x)\cap \partial X)\le C(N,v) r^{N-1}\quad \text{for any $x\in \partial X\cap B_1(p)$ and $r\in (0,1)$}\, ;
		\]
		\item[(ii)] (Volume estimate for the tubular neighbourhood) 
		\begin{equation}\label{eq:tubularestboundary}
		\haus^N(B_r(\partial X)\cap B_1(p)) \le C(N,v) r
		\quad \text{for any $r\in (0,1)$, $p\in X$},
		\end{equation}
		\item[(iii)] (Uniqueness of tangents) for any $x\in\mathcal{S}^{N-1}\setminus\mathcal{S}^{N-2}$ the tangent cone at $x$ is unique and isomorphic to $\setR^{N}_+$;
		\item[(iv)] (Topological regularity) for any $0<\alpha<1$ there exists a closed set $C_\alpha \subset\mathcal{S}^{N-2}(X)$ such that 
		\begin{itemize}
			\item[(a)] $\dim_H(X\setminus C_{\alpha})\le N-2$;
			 \item[(b)] $X\setminus C_\alpha$ is a topological manifold with boundary and $C^{\alpha}$-charts; 
			\item[(c)] the topological boundary of $X\setminus C_\alpha$ coincides with $\partial X\cap (X\setminus C_\alpha)$.
	\end{itemize}
\end{itemize}
\end{theorem}

The rectifiability of the top dimensional singular stratum was conjectured both in \cite[Conjecture 4.10]{KapovitchMondino19} and in \cite{DePhilippisGigli18}, together with the local finiteness of the $\haus^{N-1}$-measure. 
Moreover, with \eqref{eq:tubularestboundary} we sharpen the volume bound for the tubular neighbourhood of the top dimensional singular set obtained in \cite[Corollary 2.7]{AntonelliBrueSemola19} by adapting the techniques developed in \cite{CheegerNaber13} to the synthetic framework. The topological regularity part of \autoref{thm:boundatystructure} improves upon \cite[Theorem 4.11]{KapovitchMondino19}, including the boundary in the statements.

The regularity results above are mostly peculiar of codimension one singularities:
\begin{itemize}
\item Volume estimates for the tubular neighbourhood and the measure estimate for the full singular stratum, and not only for the quantitative one, fail in codimension higher than one. Indeed there are examples of two dimensional Alexandrov spaces where the singular set $\mathcal{S}^0$ has not locally finite $\haus^0$-measure, see for instance \cite[Section 3.4]{CheegerJiangNaber18}.
\item  In \cite[Theorem 1.2]{ColdingNaber13b} a noncollapsed Ricci limit space $(X,\dist,\haus^N)$ with a point $x\in\mathcal{S}^{N-2}\setminus\mathcal{S}^{N-3}$ with non unique tangent cone is constructed (actually tangents with maximal splitting $\setR^k$ for any $0\le k\le N-2$ appear at that point). 
\item  As pointed out in \cite[Remark 1.11, Example 3.2]{CheegerJiangNaber18} based on \cite{LiNaber19}, there is an example of $N$-dimensional Alexandrov space such that the singular set $\mathcal{S}^{N-2}$ is a Cantor set, and in particular no point has a neighbourhood in which $\mathcal{S}^{N-2}$ is topologically a manifold.
\end{itemize}

In the case of Ricci limits, \autoref{thm:boundatystructure} (iv) can be sharpened to a finite $\haus^{N-2}$-measure estimate for the topologically singular set, relying on \cite{CheegerJiangNaber18}:
	
	\begin{theorem}\label{intro:top2}
		Let $(X,\dist,\haus^N)$ be an $\RCD$ m.m.s. arising as noncollapsed limit of a sequence of smooth Riemannian manifolds with convex boundaries and Ricci curvature bounded from below in the interior by $-(N-1)$. Then, for any $0<\alpha<1$, there exist a constant $C=C(N,\alpha,\haus^N(B_1(p)))$ and a closed set of codimension two $C_\alpha \subset\mathcal{S}^{N-2}(X)$ such that
		\begin{equation}
		\haus^{N-2}(C_\alpha\cap B_1(p)) \le C(N,\alpha,\haus^N(B_1(p))),\quad \text{for any $p\in X$}
		\end{equation}
		and $X\setminus C_\alpha$ is a topological manifold with boundary and $C^{\alpha}$-charts. Moreover the topological boundary of $X\setminus C_\alpha$ coincides with $\partial X\cap (X\setminus C_\alpha)$.
\end{theorem}

  \vspace{.1cm} 
	
  \subsection{Stability and gap theorems for boundaries}

The following stability theorem gives an affirmative answer to \cite[Conjecture 5.11]{KapovitchMondino19}. Its proof follows directly from (a weak form of) the $\eps$-regularity theorem for boundary balls, \autoref{thm:epsregintro} (i).

  \begin{theorem}[Stability]\label{thm:stabilityboundary}
	    Let $N\in\setN^+$ and $K\in\setR$ be fixed. Let $(X_n,\dist_n,\haus^N,x_n)$ be a sequence of pointed $\RCD(K,N)$ spaces with no boundary on $B_2(x_n)$ converging in the pmGH topology to $(Y,\dist_Y,\haus^N,y)$. Then $Y$ has no boundary on $B_1(y)$.
   \end{theorem}

While the above tells that spaces without boundary converge to spaces without boundary under non collapsing pGH convergence, stability of boundary points (i.e whether boundary points converge to boundary points) remains an open question in the general case.
 
  The analysis of the Laplacian of the distance from the boundary performed in \autoref{sec:distance from the boundary} allows us to prove the local Ahlfors regularity of the boundary volume measure, together with stability of boundary points in the case of Ricci limits with boundary. 

 	\begin{theorem}\label{thm:riccilimit}
 		Let $(X,\dist,\haus^N,p)$ be the noncollapsed pGH limit of a sequence of smooth $N$-dimensional Riemannian manifolds $(X_n,\dist_n, p_n)$ with convex boundary and Ricci curvature bounded from below by $K$ in the interior. 
 	    Then:
 		\begin{itemize}
 			\item[(i)] if $B_1(p_n) \cap \partial X_n \neq \emptyset $ for every $n$ then $\partial X \neq \emptyset$. Moreover if points $x_n\in \partial X_n$ converge to $x\in X$, then $x\in\partial X$;
 			
 			\item[(ii)] for any $x\in \partial X$ one has
 			\begin{equation}\label{eq:gap}
 			\haus^{N-1}(B_2(x)\cap\partial X) > C(K) \haus^N(B_1(x))\, ;
 			\end{equation}
 			
 			\item[(iii)] $\haus^{N-1}\res\partial X$ is locally Ahlfors regular and for any $x\in\partial X$ any tangent cone at $x$ has boundary.
 		\end{itemize}
 	\end{theorem}	

 We conjecture that the gap estimate \eqref{eq:gap} holds for general noncollapsed $\RCD$ spaces without further assumptions, which would also prove stability of boundary points in full generality.

Our last result is a version of Colding's volume convergence theorem  (cf. \cite{Colding97,CheegerColding97}) for boundary measures:

  \begin{theorem}[Boundary Volume Convergence]\label{intro:stabilityboundary}
	Let $1\le N<\infty$ be a fixed natural number. Assume that $(X_n,\dist_n,\haus^N,p_n)$ are $\RCD(-(N-1),N)$ spaces converging in the pGH topology to $(X,\dist,\haus^N,p)$. 
    Then
	\begin{equation}
	\haus^{N-1}\res \partial X_n \to \haus^{N-1}\res \partial X\quad \text{ weakly}\, .
	\end{equation}
	In particular
	\[
	\lim_{n\to \infty} \haus^{N-1}(\partial X_n\cap B_r(x_n)) = \haus^{N-1}(\partial X\cap B_r(x))
	\]
	whenever $X_n\ni x_n\to x\in X$ and $\haus^{N-1}(\partial X\cap \partial B_r(x))=0$.	
 \end{theorem}

\vspace{.1cm}
 
\subsection{The remainder of the paper}
The rest of the paper is divided in eight sections. 

\medskip

The first two aim at presenting preliminary results which will be used throughout the paper. In \autoref{sec:preliminaries} we recall the main definitions and basic results of the theory of $\RCD$ spaces. In \autoref{sec:almostsplitting} we prove a local version of the almost splitting theorem, originally due to Cheeger-Colding (see \cite{CheegerColding96} and \cite{CheegerNaber15} for the present form on Ricci limit spaces) and previously proved on $\RCD$ spaces only in a weaker form (see \cite{BruePasqualettoSemola19,BruePasqualettoSemola20}). Moreover we adapt the proof of the transformation theorem \cite{CheegerJiangNaber18} (see also \cite{CheegerNaber15}) by Cheeger-Jiang-Naber to the $\RCD$ framework.
\medskip

In \autoref{sec:neckregion} we introduce and study \textit{neck regions} tailored for the analysis of boundaries on noncollapsed $\RCD$ spaces. This study is the key ingredient for the all the developments of the paper: rectifiable regularity, topological regularity and stability.\\ 
The role of this tool has been prominent in the recent literature about spaces with lower Ricci curvature bounds and bounded Ricci curvature, see \cite{JiangNaber16,CheegerJiangNaber18}, and also in several other frameworks, see for instance \cite{NaberValtorta17,NaberValtorta19}.\\
The analysis of neck regions is made in two steps. After their introduction in \autoref{def:Neckregion}, we first describe their structure in \autoref{thm:neckstructure}.
In the second step we prove existence of neck regions in \autoref{thm:existenceofneck} under geometric assumptions, guaranteeing in particular the non triviality of the previous structural result.\\
 In the analysis of the structure of neck regions there are several simplifications with respect to the study in \cite{CheegerJiangNaber18,JiangNaber16}. Instead non trivial new ideas are needed to deal with the stability of codimension one singularities and the existence of neck regions.
\medskip

In \autoref{sec:neckdecomposition}, following closely the neck decomposition theorems in \cite{JiangNaber16,CheegerJiangNaber18}, we prove that any noncollapsed $\RCD(K,N)$ space can be decomposed into neck-regions, $(N,\eps)$-symmetric balls and a set of codimension at least $2$, with quantitative summability control over the radii of balls appearing in the covering.
\medskip

In \autoref{sec:Boundary rectifiability and stability} we combine the previously obtained existence and structure of neck regions with the neck decomposition theorem to prove the weak $\eps$-regularity \autoref{thm:stabilityS}.  In particular, we show the stability \autoref{thm:stabilityboundary} for spaces without boundary, and the $(N-1)$-rectifiable structure of the boundary together with local finiteness estimates for the boundary measure (cf. \autoref{thm:boundatystructure} (i)). 
\medskip

We dedicate \autoref{sec:distance from the boundary} to the study of the distance function from the boundary. We show how upper bounds on the absolutely continuous part of its Laplacian imply noncollapsing estimates on the boundary measure, see \autoref{thm:consequencesconj}. We present an open question concerning the case of general noncollapsed $\RCD(K,N)$ spaces that we are able to verify for smooth manifolds with boundary and their noncollapsed pGH limits, as well as Alexandrov spaces with curvature bounded below.  As a consequence we prove \autoref{thm:riccilimit}.
\medskip

In \autoref{sec:improved neck structure} we improve the structure of neck regions by a bootstrap argument based on the stability of the boundary. In \autoref{thm:improvedneckregion} we prove that, on a ball sufficiently close in the GH sense to the model ball of the half-space, balls centered at boundary points are close to the model ball in the half-space and balls centered at interior points are close to the model ball in the Euclidean space at any sufficiently small scale.\\ 
The improved neck structure \autoref{thm:improvedneckregion} has a number of consequences: the topological regularity of the boundary up to sets of ambient codimension two (see \autoref{thm:goodpar}), the improved volume estimate \autoref{cor:bdryvolconvergencereg} and the boundary volume convergence \autoref{intro:stabilityboundary}. In \autoref{sec:topologicalstructureuptotheboundary} we deal with the topological regularity up to the boundary of noncollapsed $\RCD$ spaces proving \autoref{thm:epsregintro} (iv), \autoref{thm:boundatystructure} (iv) and \autoref{intro:top2}.

\subsection*{Acknowledgements}

The first author is supported by the Giorgio and Elena Petronio Fellowship at the Institute for Advanced Study.\\
The second author was partially supported by the National Science Foundation Grant No. DMS-1809011.\\
The third author is supported by the European Research Council (ERC), under the European’s Union Horizon 2020 research and innovation programme, via the ERC Starting Grant “CURVATURE”, grant agreement No. 802689. 

Most of this work was developed while the first and third authors were PhD students at Scuola Normale Superiore. They wish to express their gratitude to this institution for the excellent working conditions and the stimulating atmosphere.

\section{Preliminaries}\label{sec:preliminaries}

A metric measure space will be a triple $(X,\dist,\meas)$ where $(X,\dist)$ is a complete and separable metric space and $\meas$ is a locally finite Borel measure.\\ 
We will denote by $B_r(x)=\{\dist(\cdot,x)<r\}$ and $\bar{B}_r(x)=\{\dist(\cdot,x)\leq r\}$ the open and closed balls respectively.  By $\Lip(X)$ (resp. $\Lipb(X)$) we denote the space of Lipschitz (resp. bounded) functions and for any $f\in\Lip(X)$ we shall denote its slope by
\begin{equation}
\lip f(x) := \limsup_{y\to x}\frac{\abs{f(x)-f(y)}}{\dist(x,y)}\, .
\end{equation}
We will use the standard notation $L^p(X,\meas)$, for the $L^p$ spaces and $\Leb^n,\haus^n$ for the $n$-dimensional Lebesgue measure on $\setR^n$ and the $n$-dimensional Hausdorff measure on a metric space, respectively. The Hausdorff measure is always normalised in such a way that it coincides with the Lebesgue measure on Euclidean spaces. We shall also denote by $\haus^{n}_{\infty}$ the pre-Hausdorff measure in dimension $n$ (obtained with no upper bounds on the radii of the covering sets).
We shall denote by $\omega_n$ the Lebesgue measure of the unit ball in $\setR^n$.

We will also deal with pointed metric measure spaces $(X,\dist,\meas,x)$ in case a reference point $x\in X$ has been fixed. We will say that a pointed metric measure space is normalised whenever 
\begin{equation}
\int_{B_1(x)}\left(1-\dist(x,y)\right)\di\meas(y)=1 \, .
\end{equation}
\medskip

We will deal with the Gromov-Hausdorff (GH), measured Gromov-Hausdorff (mGH) and pointed measured Gromov-Hausdorff (pmGH) convergence of (pointed) metric measure spaces. We refer to \cite{GigliMondinoSavare15} for the relevant background about these notions. The associated distances will be denoted by $\dist_{GH}$, $\dist_{mGH}$ and $\dist_{pmGH}$.

A basic reference about analysis on metric space is the book \cite{BuragoBuragoIvanov01}. Given a proper metric space $(X,\dist)$ and two bounded subsets $F,E\subset X$ we denote by
\[
\dist_H(E,F):=\inf\set{r>0: E\subset B_r(F)\, \text{and}\, F\subset B_r(E)}
\]
their Hausdorff distance in $(X,\dist)$.\footnote{We remark that to obtain a distance one should restrict to bounded and closed sets, but this will cause no troubles for our aims.}

We recall a simple connection between convergence in the Hausdorff distance and behaviour of pre-Hausdorff measures $\haus_{\infty}^{\alpha}$, for any $\alpha\ge 0$. If $\dist_H(A_n,A)\to 0$ and $A$ is compact, then
	\begin{equation}\label{eq:limsuphaus}
	\haus^{\alpha}_{\infty}(A)\ge\limsup_{n\to\infty}\haus^{\alpha}_{\infty}(A_n)\, .
	\end{equation}
When the sets are subsets of metric spaces converging in the pGH topology we will understand the convergence as realized in a common background proper metric space and the Hausdorff convergence of compact sets has to be understood as Hausdorff convergence in the ambient space.

\begin{remark}\label{rm:compvarconv}
We recall that in a proper metric space $(Z,\dist_Z)$ with a sequence of uniformly bounded compact sets $K_n\subset Z$ and $K\subset Z$, the following conditions are equivalent: 
\begin{itemize}
\item[i)] $K_n$ converge to $K$ in the Hausdorff distance;
\item[ii)] $K_n$ converge to $K$ in the Kuratowski sense, i.e. any limit point $x$ of a subsequence $x_n\in K_n$ belongs to $K$ and for any $y\in K$ there exists a sequence $y_n\in K_n$ such that, up to subsequence, $y_n\to y$;
\item[iii)] setting $\dist_C:Z\to[0,\infty)$ to be the distance function from any closed set $C\subset Z$, it holds that $\dist_{K_n}\to\dist_K$ uniformly as $n\to\infty$.
\end{itemize}
We refer to \cite{Beer} for a treatment of these equivalences and we remark that they hold also for subsets of a pGH converging sequence of metric spaces (once the convergence is realized in a common proper metric space).
\end{remark}

\subsection{Calculus tools}
The Cheeger energy $\Ch:L^2(X,\meas)\to[0,+\infty]$ associated to a m.m.s. $(X,\dist,\meas)$ is the convex and lower semicontinuous functional defined through 
\begin{equation}\label{eq:cheeger}
\Ch(f):= \inf\left\lbrace\liminf_{n\to\infty}\int_X|\lip f_n|^2\di\meas:\quad f_n\in\Lipb(X)\cap L^2(X,\meas),\ \norm{f_n-f}_2\to 0 \right\rbrace
\end{equation}
and its finiteness domain will be denoted by $H^{1,2}(X,\dist,\meas)$. Looking at the optimal approximating sequence in \eqref{eq:cheeger}, it is possible to identify a canonical object $\abs{\nabla f}$, called minimal relaxed slope, providing the integral representation
\begin{equation}
\Ch(f) = \int_X\abs{\nabla f}^2\di\meas\qquad\forall f\in H^{1,2}(X,\dist,\meas) \, .
\end{equation}

\begin{definition}
	Any metric measure space such that $\Ch$ is a quadratic form is said to be \textit{infinitesimally Hilbertian}.
\end{definition}

Let us recall from \cite{AmbrosioGigliSavare14,Gigli15} that, under the infinitesimally Hilbertian assumption, the function 
\begin{equation}
\nabla f_1 \cdot \nabla f_2 := \lim_{\eps\to 0}\frac{\abs{\nabla(f_1+\eps f_2)}^2-\abs{\nabla f_1}^2}{2\eps}
\end{equation}
defines a symmetric bilinear form on $H^{1,2}(X,\dist,\meas)\times H^{1,2}(X,\dist,\meas)$ with values into $L^1(X,\meas)$.

It is possible to define a Laplacian operator $\Delta:\mathcal{D}(\Delta)\subset L^{2}(X,\meas)\to L^2(X,\meas)$ in the following way. We let $\mathcal{D}(\Delta)$ be the set of those $f\in H^{1,2}(X,\dist,\meas)$ such that, for some $h\in L^2(X,\meas)$, one has
\begin{equation}\label{eq:amb1}
\int_X \nabla f\cdot\nabla g\di\meas=-\int_X hg\di\meas\qquad\forall g\in H^{1,2}(X,\dist,\meas) 
\end{equation} 
and, in that case, we put $\Delta f=h$. 
It is easy to check that the definition is well-posed and that the Laplacian is linear (because $\Ch$ is a quadratic form).

	\begin{definition}[Perimeter and sets of finite perimeter]\label{def:setoffiniteperimeter}
		Given a Borel set $E\subset X$ and an open set $A\subset X$ the perimeter $\Per(E,A)$ is defined as
		\begin{equation}
		\Per(E,A):=\inf\left\lbrace \liminf_{n\to\infty}\int_A\lip(u_n) \di\meas: u_n\in\Lip_{\loc}(A),\quad u_n\to\chi_E\quad \text{in } L^1_{\loc}(A,\meas)\right\rbrace \, .
		\end{equation}
		We say that $E$ has locally finite perimeter if $\Per(E,K)<\infty$ for any compact set $K$. In that case it can be proved that the set function $A\mapsto\Per(E,A)$ is the restriction to open sets of a locally finite Borel measure $\Per(E,\cdot)$ defined by
		\begin{equation}
		\Per(E,B):=\inf\left\lbrace \Per(E,A): B\subset A,\text{ } A \text{ open}\right\rbrace\, .
		\end{equation}
	\end{definition}

	The following coarea formula is taken from \cite[Proposition 4.2]{MirandaJr}.

\begin{theorem}[Coarea formula]\label{thm:coarea}
	Let $(X,\dist,\meas)$ be a locally compact metric measure space and $v\in\Lip(X)$.
	Then, $\{v>r\}$ has locally finite perimeter for $\Leb^1$-a.e. $r\in\setR$ and, for any Borel function $f:X\to[0,\infty]$, it holds 
	\begin{equation}\label{eq:coarea}
	\int_X f |\nabla v| \di\meas = \int_{-\infty}^{\infty}\left(\int  f\di\Per(\{v>r\},\cdot)\right)\di r \, .
	\end{equation}
\end{theorem}

\subsection{$\RCD$ spaces}

Let us start by recalling the so-called curvature dimension condition $\CD(K,N)$. Its introduction dates back to the seminal and independent works \cite{LottVillani} and \cite{Sturm06a,Sturm06b}, while in this presentation we closely follow \cite{BacherSturm10}. 

Below $\Prob(X)$ denotes the set of probability measure over $X$ while 
\[
\text{Geo}(X):=\set{\gamma:[0,1]\to X:\, \dist(\gamma(t),\gamma(t))=|t-s|\dist(\gamma(1),\gamma(0))\quad s,t\in [0,1]}.
\]

\begin{definition}[Curvature dimension bounds]\label{def:CD}
	Let $K\in\setR$ and $1\le N<\infty$. We say that a m.m.s. $(X,\dist,\meas)$ is a $\CD(K,N)$ space if, for any $\mu_0,\mu_1\in\Prob(X)$ absolutely continuous w.r.t. $\meas$ with bounded support, there exists an optimal geodesic plan $\Pi\in\Prob(\text{Geo}(X))$ such that for any $t\in[0,1]$ and for any $N'\ge N$ we have 
	\begin{align*}
	-\int & \rho_t^{1-\frac{1}{N'}}\di\meas\\
	\le &-\int\left\lbrace\tau_{K,N'}^{(1-t)}(\dist(\gamma(0),\gamma(1)))\rho_0^{-\frac{1}{N'}}(\gamma(0))+\tau_{K,N'}^{(t)}(\dist(\gamma(0),\gamma(1)))\rho_1^{-\frac{1}{N'}}(\gamma(1)) \right\rbrace\di\Pi(\gamma)\, , 
	\end{align*} 
	where $(e_t)_{\sharp}\Pi=\rho_t\meas$, $\mu_0=\rho_0\meas$, $\mu_1=\rho_1\meas$ and the distortion coefficients $\tau_{K,N}^{t}(\cdot)$ are defined as follows. First we define the coefficients $[0,1]\times[0,\infty)\ni(t,\theta)\mapsto\sigma_{K,N}^{(t)}(\theta)$ by
	\begin{equation}
	\sigma_{K,N}^{(t)}(\theta) :=
	\begin{cases*}
	+\infty &\text{if $K\theta^2\ge N\pi^2$,}\\
	\frac{\sin(t\theta\sqrt{K/N})}{\sin(\theta\sqrt{K/N})}&\text{if $0<\theta<N\pi^2$,}\\
	t&\text{if $K\theta^2=0$,}\\
	\frac{\sinh(t\theta\sqrt{K/N})}{\sinh(\theta\sqrt{K/N})}&\text{if $K\theta^2<0$,}
	\end{cases*}
	\end{equation}
	then we set $\tau_{K,N}^{(t)}(\theta) := t^{1/N}\sigma_{K,N-1}^{(t)}(\theta)^{1-1/N}$.
\end{definition}

\begin{definition}\label{defRCDKN}
	We say that a metric measure space $(X,\dist,\meas)$ satisfies the \textit{Riemannian curvature-dimension} condition for some $K\in\setR$ and $1\le N<\infty$ (it is an $\RCD(K,N)$ m.m.s. for short) if it is a $\CD(K,N)$ infinitesimally hilbertian metric measure space.
\end{definition}
Note that, if $(X,\dist,\meas)$ is an $\RCD(K,N)$ m.m.s., then so is $(\supp\meas,\dist,\meas)$, hence in the following we will always tacitly assume $\supp\meas = X$.

\begin{remark}[Compatibility with the smooth case]
The $\RCD(K,N)$ notion is compatible with the smooth case of weighted Riemannian manifolds with (weighted-)Ricci curvature bounded from below \cite{GigliMondinoSavare15,Sturm06a,Sturm06b,LottVillani,Villani09}.	
It means that a Riemannian manifold meets the $\RCD(K,N)$ condition if and only if it has dimension smaller than $N$ and the $N$-dimensional Bakry-Ricci tensor is bounded below by $K$.
\end{remark}

\begin{remark}[Stability]
	\label{remark:stability}
	A fundamental property of $\RCD(K,N)$ spaces, that will be used several times in this paper, is the stability w.r.t. pmGH convergence, meaning that a pmGH limit of a sequence of (pointed) $\RCD(K,N)$ spaces is still an $\RCD(K,N)$ m.m.s.. 
\end{remark}

The basic references for the theory of convergence and stability of Sobolev functions on converging sequences of $\RCD(K,N)$ metric measure spaces are \cite{GigliMondinoSavare15} and \cite{AmbrosioHonda,AmbrosioHonda18}.
\medskip

We recall that any $\RCD(K,N)$ m.m.s. $(X,\dist,\meas)$ satisfies the Bishop-Gromov inequality:
\begin{equation}\label{eq:BishopGromovInequality}
\frac{\meas(B_R(x))}{v_{K,N}(R)}\le\frac{\meas(B_r(x))}{v_{K,N}(r)}\, ,
\end{equation}
for any $0<r<R$ and for any $x\in X$, where $v_{K,N}(r):= N\omega_N\int_0^r \left(s_{K,N}(s)\right)^{N-1}\di s$ and
\begin{equation}\label{DefinitionSkn}
s_{K,N}(r) :=
\begin{cases}
\sqrt{\frac{N-1}{K}}\sin\left(\sqrt{\frac{K}{N-1}}r\right) & \mbox{if} \quad K>0\, , \\
r & \mbox{if} \quad K=0\, , \\
\sqrt{\frac{N-1}{-K}}\sinh\left(\sqrt{\frac{-K}{N-1}}r\right) & \mbox{if} \quad K<0\, . \\
\end{cases}
\end{equation}
In particular $(X,\dist,\meas)$ is locally uniformly doubling, that is to say, for any $R>0$ there exists $C(K,N,R)>0$ such that
\begin{equation}\label{eq:locdoubling}
\meas(B_{2r}(x))\le C(K,N,R)\meas(B_r(x))\quad \text{for any $x\in X$ and for any $0<r<R$\, .}
\end{equation}
Moreover, in \cite{VonRenesse08} has been proven that  $\RCD(K,N)$ spaces verify a local Poincar\'e inequality, therefore they fit in the general framework of PI spaces considered for instance in \cite{Cheeger99}.\\

\subsubsection{Structure theory}
From the point of view of geometric measure theory a notion of $k$-regular point for an $\RCD(K,N)$ metric measure space $(X,\dist,\meas)$ can be introduced in the following terms. 

\begin{definition}[Regular points]
	We say that $x\in\mathcal{R}_k$ whenever
	\[
	\left(X,\frac{\dist}{r},\frac{\meas}{\meas(B_r(x))},x\right) \to 
	(\setR^k,\dist_{eucl},\omega_k^{-1}\haus^k,0^k)
	\quad \text{in the pmGH topology}
	\] 
	as $r\downarrow 0$, where $\omega_k:=\haus^k(B_1(0^k))$. 
\end{definition}

In \cite{BrueSemola18} (see also the very recent \cite{Deng20}), generalizing a previous result obtained for Ricci limits in \cite{ColdingNaber13}, it has been proved that for any $\RCD(K,N)$ metric measure space $(X,\dist,\meas)$ there exists an integer $1\le n\le N$, that we shall call \emph{essential dimension} of $(X,\dist,\meas)$ from now on, such that
\begin{equation}
\meas(X\setminus\mathcal{R}_n)=0 \, .
\end{equation}
In this generality we also know after \cite{MondinoNaber19} that $X$ is $(\meas,n)$-rectifiable as metric space. Moreover, the representation formula $\meas=\theta\haus^n$, for some locally integrable nonnegative density $\theta$, has been obtained in the independent works \cite{KellMondino18,DePhlippisMarcheseRindler17,GigliPasqualetto16a}.

\subsubsection{Calculus on $\RCD$ spaces }
We refer to \cite{AmbrosioGigliSavare14,Gigli15,Gigli18} for the basic background about first and second order differential calculus on $\RCD$ spaces.

Here and in the following we denote by $\Hess u$ the Hessian of a function $u\in H^{2,2}(X,\dist,\meas)$, referring to \cite{Gigli18} for its introduction and the study of its main properties in this framework. Thanks to locality, we will be dealing also with functions that are defined only locally. The space of test functions as introduced in \cite{Gigli18} will be denoted by
\begin{equation}
\Test(X,\dist,\meas)\coloneqq\{f\in D(\Delta)\cap L^{\infty}(X,\meas): \abs{\nabla f}\in L^{\infty}(X)\quad\text{and}\quad\Delta f\in H^{1,2}(X,\dist,\meas) \}\, .
\end{equation}
The existence of many test functions within this framework is one of the outcomes of \cite{Savare14}.\\ 
We will also rely repeatedly on the following existence result for \emph{good} cut-off functions.

\begin{lemma}[Good cut-off functions \cite{AmbrosioMondinoSavare14}]
	\label{lem:good_cut-off}
	Let $(X,\dist,\meas)$ be an $\RCD(K,N)$ space. Let $p\in X$ be fixed. Then there exists
	$\eta\in\Test(X,\dist,\meas)$ such that \(0\leq\eta\leq 1\) on \(X\), the support
	of \(\eta\) is compactly contained in $B_5(p)$, and $\eta=1$ on $B_4(p)$.
\end{lemma}

In \autoref{sec:almostsplitting} we will rely on some tools from optimal transportation on $\RCD(K,N)$ spaces. Mainly we will be concerned with first and second derivatives of (sufficiently regular) potentials along Wasserstein geodesics. We will denote by $W_2$ the Wasserstein distance induced by optimal transport with cost equal to the distance squared on the space $\Prob_2(X)$ of probabilities with finite second moment. We refer to \cite{Villani09,AmbrosioGigliSavare14} for the basic terminology about this topic and to \cite{GigliRajalaSturm16} for a more detailed account about optimal transportation on $\RCD(K,N)$ spaces.

The next result follows by combining Proposition 5.15 and Corollary 5.7 in \cite{Gigli13}.
\begin{proposition}\label{prop:derivativealonggeodesic1}
	Let $(X,\dist,\meas)$ be an $\RCD(K, N)$ space for some $K\in \setR$ and $1\le N<\infty$. Consider a $W_2$-geodesic $(\eta_s)_{s\in [0,1]}\in \Prob_2(X)$, satisfying $\eta_s\le C\meas$ and $\supp \eta_s \subset B_R(p)$ for any $s\in [0,1]$, for some $C>0$, $R>0$ and $p\in X$. Then, for any $u\in \Lip(X,\dist)$, the function $s\mapsto \int u\di\eta_s$ is $C^1$ and one has
	\begin{equation}
	\frac{\di}{\di s} \int u \di \eta_s=\frac{1}{s}\int \nabla u\cdot \nabla \phi_s \di \eta_s
	\quad \text{for every $s\in (0,1]$}\, ,
	\end{equation}
	where $\phi_s$ is any Lipschitz Kantorovich potential from $\eta_s$ to $\eta_0$.
\end{proposition}

\subsection{Continuity equation and flow maps}

Let us recall that given any function $u\in\Test(X,\dist,\meas)$ a solution to the continuity equation induced by $\nabla u$ is any absolutely continuous curve $(\rho_t)_{t\in[0,1]}\subset \Prob_2(X)$ such that the following holds: for any $f\in \Lip_c(X,\dist)$ the function $t\mapsto\int f\di \rho_t$ is absolutely continuous and it satisfies 
\begin{equation}
\frac{\di}{\di t}\int f \di \rho_t=\int \nabla u\cdot \nabla f \di \rho_t
\quad \text{for a.e. $t\in [0,1]$ }.
\end{equation} 
We refer to \cite{GigliHan15a} and \cite{AmbrosioTrevisan14} for the treatment of the continuity equation on ($\RCD$) metric measure spaces.

The next result is a particular case of \cite[Proposition 3.11]{GigliHan15a}.

\begin{lemma}\label{lemma:derivativealongflow}
	Let $(X,\dist,\meas)$ be an $\RCD(K,N)$ metric measure space for some $K\in\setR$ and $1\le N<\infty$. Let $u\in\Test(X,\dist,\meas)$ and let $(\rho_t)_{t\in[0,1]}\subset \Prob_2(X)$ be a solution of the continuity equation associated to $\nabla u$.
	If we further assume that $\rho_t\le C\meas$ for any $t\in[0,T]$ for some $C>0$, then for any $\nu\in\Prob_2(X)$ it holds
	\begin{equation}\label{eq:derivativeW_2}
	\frac{\di}{\di t}\frac{1}{2}W_2^2(\rho_t,\nu)=\int \nabla u\cdot\nabla\phi_t\di \rho_t\quad\text{for a.e. $t\in (0,1)$}\, ,
	\end{equation}
	where $\phi_t$ is any optimal Kantorovich potential for the transport problem between $\rho_t$ and $\nu$.
\end{lemma}

  The theorem below is taken from \cite{GigliTamanini19} where the second order differentiation formula along $W_2$-geodesics has been proved on $\RCD(K,N)$ metric measure spaces.

\begin{theorem}\label{thm:diffforumlaforvectorfields}
	Let $(X,\dist,\meas)$ be an $\RCD(K,N)$ m.m.s. for some $1\le N<\infty$. Let $(\eta_s)_{s\in[0,1]}$ be a $W_2$-geodesic connecting probability measures $\eta_0$ and $\eta_1$ absolutely continuous w.r.t. $\meas$ and with bounded densities and assume that $u\in\Test(X,\dist,\meas)$. Then, the curve
	\begin{equation}
	s\mapsto \int \nabla u\cdot\nabla\phi_s\di\eta_s
	\end{equation}
	is $C^1$ on $[0,1]$, where $\phi_s$ is any function such that for some $r\in[0,1]$ with $s\neq r$ it holds that $-(r-s)\phi_s$ is an optimal Kantorovich potential from $\eta_s$ to $\eta_r$. Moreover
	\begin{equation}
	\frac{\di}{\di s}\int \nabla u\cdot\nabla\phi_s\di\eta_s=\int\Hess u(\nabla\phi_s,\nabla\phi_s)\di\eta_s\, ,
	\quad\text{for any $s\in[0,1]$.}
	\end{equation}	
\end{theorem}

In the context of $\RCD(K,\infty)$ spaces a general theory of flows for Sobolev vector fields has been developed in \cite{AmbrosioTrevisan14}. Here we only collect some simplified statements relevant for our purposes.
 			
\begin{definition}\label{def:Regularlagrangianflow}
	Let $u\in\Test(X,\dist,\meas)$. We say that a Borel map $\XX:[0,\infty)\times X\rightarrow X$ is a Regular Lagrangian flow (RLF for short) associated to $\nabla u$ if the following conditions hold true:
	\begin{itemize}
		\item [1)] $\XX(0,x)=x$ and $\XX(\cdot,x)\in C([0,\infty);X)$ for every $x\in X$;
		\item [2)] there exists $L\ge0$, called compressibility constant, such that
		\begin{equation}
		\XX(t,\cdot)_{\sharp} \meas\leq L\meas,\qquad\text{for every $t\geq 0$}\, ;
		\end{equation}
		\item [3)] for every $f\in \Test(X,\dist,\meas)$ the map $t\mapsto f(\XX(t,x))$ is locally absolutely continuous
		in $[0,\infty)$ for $\meas$-a.e. $x\in X$ and
		\begin{equation}\label{eq: RLF condition 3}
		\frac{\di}{\di t} f(\XX(t,x))= \nabla u\cdot \nabla f(\XX(t,x)) \quad \quad \text{for a.e.}\ t\in (0,\infty)\, .	
		\end{equation}
	\end{itemize}
\end{definition}

In the next theorem we state some general results concerning Regular Lagrangian flows that will be used in the sequel.
\begin{theorem}\label{th: Lagrangianflows}
	Let $(X,\dist,\meas)$ be an $\RCD(K,\infty)$ space for some $K\in\setR$. Let us fix a function $u\in\Test(X,\dist,\meas)$ with $\Delta u\in L^{\infty}(X,\meas)$. Then
	\begin{itemize}
		\item [(i)] there exists a unique regular Lagrangian flow $\XX:\setR\times X\to X$ associated to $\nabla u$\footnote{To be more precise, there exist unique Regular Lagrangian flows $\XX^+,\XX^-:[0,+\infty)\times X\to X$ associated to $\nabla u$ and $-\nabla u$ respectively and we let $\XX_t=\XX^+_t$ for $t\ge0$ and $\XX_t=\XX_{-t}^-$ for $t\le0$.} (uniqueness is understood in the following sense: if $\XX$ and $\bar\XX$ are Regular Lagrangian flows associated to $\nabla u$, then for $\meas$-a.e. $x\in X$ one has $\XX_t(x)=\bar \XX_t(x)$ for any $t\in \setR$);
		\item [(ii)] $\XX$ satisfies the semigroup property: for any $s\in\setR$ it holds that, for $\meas$-a.e. $x\in X$,
		\begin{equation}\label{eq:semigroup property}
		\XX(t,\XX(s,x))=\XX(t+s,x)
		\qquad
		\forall t\in \setR \, ,
		\end{equation}
		and the following bound is verified:
		\begin{equation}\label{w}
		e^{-t\norm{\Delta u}_{L^{\infty}}}\meas \le \left(\XX_t\right)_{\sharp}\meas \le e^{t\norm{\Delta u}_{L^{\infty}}}\meas\, ;
		\end{equation}
		\item[(iii)] if $\Delta u=0$, $\Hess u=0$ and $|\nabla u|\le 1$ in $B_4(p)$ for some $p\in X$ then, for any $x\in B_1(p)$ and $t\in (-1,1)$ the map $\XX$ admits a  pointwise representative satisfying
		\begin{equation}\label{eq:isometry}
		\dist(\XX_t(x),\XX_t(y))=\dist(x,y)
		\quad \text{for any $x,y\in B_1(p)$, and $t\in (-1,1)$}\, .
		\end{equation}
	\end{itemize}
\end{theorem}
\begin{proof}
	We refer to  \cite[Theorem 1.12]{AmbrosioBrueSemola19} for the proof of (i) and (ii), while (iii) follows just localising the argument in \cite[Theorem 2.7]{BrueSemola18b}.
\end{proof}

\begin{remark}[Continuity equations and flow maps]
	\label{remark:Continuity equations and flow maps}
	Solutions to the continuity equations and flow maps are strictly related. Indeed, given a function $u\in\Test(X,\dist,\meas)$ with $\Delta u\in L^{\infty}(X,\meas)$ and $\XX$ a RLF associated to $\nabla u$ one has that
	\[
	\rho_t \meas :=(\XX_t)_{\#} \rho_0 \meas\, ,
	\quad t\in [0,1]\, ,
	\]
	is the unique solution to the continuity equation with initial datum $\rho_0\in L^{\infty}(X)$ and velocity field $\nabla u$.
\end{remark}

\subsection{Noncollapsed spaces}

	In \cite{DePhilippisGigli18} the notion of \emph{noncollapsed} $\RCD(K,N)$ metric measure space has been proposed motivated by the theory of noncollapsed Ricci limit spaces, studied since \cite{CheegerColding97} (see also a similar, though a priori weaker, notion suggested in \cite{Kitabeppu19}). We say that $(X,\dist,\meas)$ is a noncollapsed $\RCD(K,N)$ space if $N$ is an integer and $\meas=\haus^N$.
	
We point out that another relevant class to consider would be that of $\RCD(K,N)$ metric measure spaces for which the essential dimension equals $N$ (called \emph{weakly noncollapsed} in \cite{DePhilippisGigli18}). In the compact case it is known that these spaces are noncollapsed in the above sense \cite{Honda19} and it is conjectured that this should be true also in the general case.

On top of the usual structure of $\RCD(K,N)$ spaces, noncollapsed spaces have additional regularity properties.

Let us begin by pointing out the following powerful result \cite[Theorem 1.2]{DePhilippisGigli18}, generalizing a previous statement due to Cheeger-Colding \cite{CheegerColding97} (see also \cite{Colding97}).

\begin{theorem}[Volume convergence]\label{thm:volumeconvergence}
	Let $(X_n,\dist_n,\haus^N,x_n)$ be pointed noncollapsed $\RCD(K,N)$ metric measure spaces and assume that they converge in the pGH topology to $(X,\dist,x)$. Then, if 
	\begin{equation}
	\limsup_{n\to\infty}\haus^N(B_1(x_n))>0\, ,
	\end{equation}
	they also converge in the pmGH topology to $(X,\dist,\haus^N,x)$.
\end{theorem}

We refer to \cite[Definition 3.6.16]{BuragoBuragoIvanov01} for the definition of metric cone $(C(Y),\dist_C)$ over a metric space $(Y,\dist_Y)$. Here we just recall that for points $(r_1,x_1)$ and $(r_2,x_2)$ such that $r_1,r_2\ge 0$ and $\dist_Y(x_1,x_2)\le \pi$ the cone distance is given by the law of cosines:
\begin{equation}
\dist_C^2\left((r_1,x_1),(r_2,x_2)\right)=r_1^2+r_2^2-2r_1r_2\cos(\dist_Y(x_1,x_2))\, .
\end{equation}
In \cite{DePhilippisGigli18}, generalizing \cite{CheegerColding97} and relying on \cite{DePhilippisGigli16} (which extends in turn one of the key results in \cite{CheegerColding96}) it has been proven that for a noncollapsed $\RCD(K,N)$ metric measure space $(X,\dist,\haus^N)$ any tangent cone is a metric cone over a noncollapsed $\RCD(N-2,N-1)$ metric measure space $(Y,\dist_Y,\haus^{N-1})$. This amounts to say that any pGH limit of $(X,r_i^{-1}\dist,x)$, for some sequence of radii $r_i\downarrow 0$, is a metric cone in the sense above.

Given any noncollapsed $\RCD(K,N)$ metric measure space $(X,\dist,\haus^N)$ and any $x\in X$
let us denote by 
\[
\Theta_X (x):=\lim_{r\to 0}\frac{\haus^{N}(B_r(x))}{\omega_N r^N}
\]
the density of $\haus^N$ at $x$ (when there is no risk of confusion we will drop the dependence on the ambient space $X$). The existence of the limit above follows from the Bishop-Gromov inequality. Moreover, the lower semicontinuity of the density (cf. \cite[Lemma 2.2 (i)]{DePhilippisGigli18}) together with a standard result about differentiation of measures allow to infer that $\Theta(x)\le 1$ for every $x\in X$ and $\Theta(x)=1$ for $\haus^N$-a.e. $x\in X$.\\
By volume rigidity (see \cite[Theorem 1.6]{DePhilippisGigli18} after \cite{Colding97}) we recognize that $\Theta(x)=1$ if and only if the tangent cone is unique and isometric to $(\setR^N,\dist_{eucl})$.\\ 
More generally, Colding's volume convergence theorem \cite{Colding97,CheegerColding97} (see also \cite[Theorem 1.3]{DePhilippisGigli18}) yields that for any $x\in X$ any cross section $(Y,\dist_Y)$ of a tangent cone $C(Y)$ at $x$ satisfies 
\begin{equation}\label{eq:volcross}
\haus^{N-1}(Y)=N\omega_N\Theta(x)\, .
\end{equation}

\subsection{Cone splitting via content}\label{sec:conesplitting}
 Let us start by restating a quantitative version of the cone splitting lemma \cite[Lemma 4.1]{CheegerNaber13} tailored for $\RCD(K,N)$ spaces (see \cite{AntonelliBrueSemola19} for the present version).

\begin{definition}\label{def:ConicalSet}
	We define the \emph{ $\varepsilon-(t,r)$ conical set} in $B_{\frac{1}{2}}(x_0)$ as
	\begin{equation}
	C^{\varepsilon}_{t,r} := \left\lbrace x\in B_{\frac{1}{2}}(x_0) : \dist_{GH}\left(\bar{B}_{\frac{tr}{2}}(x),\bar{B}_{\frac{tr}{2}}(z)\right) \leq \frac{\varepsilon r }{2}\, \text{for some $\RCD(0,N)$ cone $Z$ with tip $z$}
	\right\rbrace.
	\end{equation}
\end{definition}

\begin{theorem}\label{thm:conesplittingquant}
	For all $K\in \mathbb{R}$, $N\in [2,\infty)$, $0<\gamma<1$, $\delta<\gamma^{-1}$, and for all $\tau,\psi>0$ there exist $0<\varepsilon=\varepsilon(N,K,\gamma,\delta,\tau,\psi)<\psi$ and $0<\theta=\theta(N,K,\gamma,\delta,\tau,\psi)$ such that the following holds. Let $(X,\dist,\meas)$ be an $\RCD(K,N)$ m.m.s., $x\in X$ and $r\le\theta$ be such that there exists an $\varepsilon r$-GH equivalence
	\begin{equation}
	F:B_{\gamma^{-1}r}\left((0,z^*)\right)\to B_{\gamma^{-1}r}(x)
	\end{equation}
	for some cone $\setR^l\times C(Z)$, with $(Z,\dist_Z,\meas_Z)$ an $\RCD(N-l-2,N-l-1)$ m.m.s..
	If there exists 
	\begin{equation}
	x'\in B_{\delta r}(x)\cap C^{\varepsilon}_{\gamma^{-N},\delta r}
	\end{equation}
	with 
	\begin{equation}
	x'\notin B_{\tau r}\left(F\left(\setR^l\times\left\lbrace z^*\right\rbrace \cap B_{\gamma^{-1}r}((0,z^*))\right)\right)\cap B_r(x),
	\end{equation}
	then for some cone $\setR^{l+1}\times C(\tilde{Z})$, where  $(\tilde{Z},\dist_{\tilde{Z}},\meas_{\tilde{Z}})$ is an $\RCD(N-l-3,N-l-2)$ m.m.s.,
	\begin{equation}
	\dist_{GH}\left(B_r(x), B_r((0,\tilde{z}^*))\right)<\psi r\, .
	\end{equation}  
\end{theorem}

\autoref{thm:conesplittingquant} is a quantitative version of the following statement: if a metric cone with vertex $z$ is a metric cone also with respect to $z'\neq z$, then it contains a line.

Let us now present a quantitative version of the cone splitting theorem via content, taken from \cite[Theorem 4.9]{CheegerJiangNaber18}.
We begin by defining the notion of the pinching set.
\begin{definition}\label{def:smallpinching}
	Let $(X,\dist,\haus^N)$ be an $\RCD(-\xi(N-1),N)$ and $p\in X$, we set
	\begin{equation}
	\bar{V}:=\inf_{y\in B_4(p)}\mathcal{V}_{\xi^{-1}}(y)
	\end{equation}
	where $\mathcal{V}_{r}(x):= \frac{\haus^N(B_r(x))}{v_{K,N}(r)}$ is the volume ratio appearing in \eqref{eq:BishopGromovInequality}.
	
	We define the set with small volume pinching accordingly to be
	\begin{equation}\label{eq:setsmallvolpinch}
	\mathcal{P}_{r,\xi}(x):=\left\lbrace y\in B_{4r}(x): \mathcal{V}_{\xi r}(y)\le \bar{V}+\xi \right\rbrace. 
	\end{equation}
\end{definition}

\begin{theorem}\label{thm:almostconesplittingcontent}
	Let $(X,\dist,\haus^N)$ be an $\RCD(-\xi(N-1),N)$ space with $\haus^N(B_1(p))>v>0$.
	If for some $r\in (0,1)$ it holds that $\mathcal{H}^{N}(B_{\gamma r}(\mathcal{P}_{r,\xi}))\ge \eps\gamma r^{N}$, with $0<\delta$, $\eps<\delta(N,v)$, $\gamma\le\gamma(N,v,\eps)$ and $\xi\le\xi(\delta,\eps,\gamma,N,v)$, then there exists $q\in B_{4r}(p)$ such that $B_{\delta^{-1}r}(q)$ is $(N-1,\delta^2)$-symmetric.
\end{theorem}

The proof of \autoref{thm:almostconesplittingcontent} easily follows from \autoref{thm:conesplittingquant}, see for instance \cite[Theorem 7.6]{JiangNaber16}.

\section{Splitting maps on $\RCD$ spaces}\label{sec:almostsplitting}

In the development of the structure theory of Ricci limit spaces a prominent role has been played by the $\delta$-splitting maps \cite{CheegerColding96,CheegerColding97,CheegerNaber15,CheegerJiangNaber18}. After the construction of a second order differential calculus on $\RCD$ spaces in \cite{Gigli18} this tool, which provides a way to turn analytic information into geometric information, has also begun to play a role in the synthetic framework \cite{BruePasqualettoSemola19}.

All the works mentioned above rely on the equivalence between the existence of an $\setR^k$-valued $\delta$-splitting map and the (pointed measured)GH-closeness to a product with factor $\setR^k$. Below we state the definition of a $\delta$-splitting map relevant for the sake of this paper.

\begin{definition}\label{def:deltasplitting}
	Let $(X,\dist,\meas)$ be an $\RCD(-(N-1),N)$ m.m.s., $p\in X$ and $\delta>0$ be fixed.
	We say that $u:=(u_1,\ldots,u_k):B_r(p)\to \setR^k$ is a $\delta$-splitting map provided it is harmonic and it satisfies:
	\begin{itemize}
		\item[(i)] $|\nabla u_a|< C(N)$;
		\item[(ii)] $r^2\fint_{B_r(p)} |\Hess u_a|^2\di \meas <\delta$;
		\item[(iii)] $\fint_{B_r(p)} |\nabla u_a\cdot \nabla u_b-\delta_{a,b}|\di \meas <\delta$;
	\end{itemize}
	for any $a,b=1,\ldots,k$.
\end{definition}

\begin{remark}[About the scale invariant smallness of the Hessian]\label{rm:hessiansuper}
	If we make the stronger assumption that the ambient space is $\RCD(-\delta(N-1),N)$, then condition (ii) is unnecessary once we strengthen the harmonicity assumption to harmonicity on $B_{2r}(p)$, since it follows from conditions (i) and (iii) integrating the Bochner inequality against a good cut-off function, see for instance \cite{BruePasqualettoSemola20}.
\end{remark}

\begin{remark}[Sharper gradient bounds]\label{rm:sharpgradientbound}
 If we assume that $(X,\dist,\meas)$ is an $\RCD(-\delta(N-1),N)$ metric measure space then the gradient bound in (i) can be sharpened to the conclusion
 \begin{equation}\label{eq:sharpgradbound}
 \sup_{B_{r/2}(p)}\abs{\nabla u_a}\le 1+C(N)\delta^{1/2} \quad \text{for any $a=1,\ldots,k$}\, .
 \end{equation}
 In particular, if $u:B_r(x)\to\setR^{k}$ is a $\delta$-splitting map   according to \autoref{def:deltasplitting} and the ambient space is $\RCD(-\delta(N-1),N)$, then $u:B_{r/2}(x)\to\setR^{k}$ is a $C(N)\delta$-splitting map and we can replace condition (i) in the definition with the sharper gradient bound \eqref{eq:sharpgradbound}.
 
 Moreover the following Lipschitz estimate holds
 \begin{equation}\label{eq:lipschitzdeltasplitt}
 	|u(x) - u(y)| \le (1 + C(N)\delta^{1/2}) \dist(x,y)\quad \text{for any $x,y\in B_{r/4}(p)$}\, .
 \end{equation}
 The validity of \eqref{eq:sharpgradbound} has been pointed out for the first time in the framework of smooth Riemannian manifolds in \cite[equations (3.42)--(3.46)]{CheegerNaber15}, we report here a slightly modified argument tailored for the $\RCD$ framework.

 Let us fix $a\in\{1,\dots,k\}$ and drop the dependence on the chosen component, writing just $u$.\\
 Observe that, by volume doubling and (iii), for any $y\in B_{r/2}(p)$ it holds that 
 \begin{equation}\label{eq:changecentre}
 \fint_{B_{r/2}(y)}\abs{\abs{\nabla u}^2-1}\di\meas\le C(N)\delta\, .
 \end{equation}  

Now we consider a regular cut-off function $\phi:X\to[0,1]$ such that $\phi=1$ on $B_{3/4r}(x)$ and $\phi=0$ outside of $B_r(x)$, $r^2\abs{\Delta\phi}\le C(N)$ and $r\abs{\nabla\phi}\le C(N)$ (see \autoref{lem:good_cut-off}) and the one parameter family
\begin{equation}
f_t(y):=\int(\abs{\nabla u}^2(z)-1)\phi(z)p_t(y,z)\di\meas(z)\, .
\end{equation}
Differentiating with respect to time, taking into account the heat kernel estimate
\begin{equation}\label{eq:doubheat}
p_t(y,z)\le \frac{C(N) e^{-c(N)\frac{r^2}{t} }}{\meas(B_{\sqrt{t}}(x))}\le \frac{C(N)}{\meas(B_r(x))},\;\;\text{$\forall \, y\in B_{\frac12 r}(x)$, $z\in B_r(x)\setminus B_{\frac34r}(x)$ and $t\in[0,r^2]$,}
\end{equation} 
which follows from \cite{JangLiZhang} and volume doubling, for any $y\in B_{r/2}(x)$ and $t\in[0,r^2]$ we can estimate
\begin{align}
\nonumber\frac{\di}{\di t}f_t(y)=&\int\left(\Delta\abs{\nabla u}^2\phi
+2\nabla\abs{\nabla u}^2\cdot\nabla\phi
+(\abs{\nabla u}^2-1)\Delta\phi\right)(z)p_t(y,z)\di\meas(z)\\
\nonumber\ge&-\delta\int\abs{\nabla u}^2(z)p_t(y,z)\di\meas(z)-\frac{C(N)}{r}\int_{B_r(x)\setminus B_{\frac34r}(x)}\abs{\Hess u}(z)p_t(y,z)\di\meas(z)\\
\nonumber&-\frac{C(N)}{r^2}\int_{B_r(x)\setminus B_{\frac34r}(x)}\abs{\abs{\nabla u}^2-1}(z)p_t(y,z)\di\meas(z)\\
\nonumber\overset{\eqref{eq:doubheat}}{\ge}&-C(N)\delta-C(N)\frac{\delta^{1/2}}{r^2}-C(N)\frac{\delta}{r^2}\\
\ge&-C(N)\frac{\delta^{1/2}}{r^2}\, .\label{eq:boundsharp}
\end{align}
Above the first inequality follows from the bounds for the cut-off function and from Bochner's inequality.

Given \eqref{eq:boundsharp}, observing that, for $\meas$-a.e. $y\in B_{r/2}(x)$ it holds
\begin{equation}
f_t(y)\to\abs{\abs{\nabla u}^2-1}(y),\;\;\;\text{as $t\downarrow 0$}\, ,
\end{equation}
we can integrate between $0$ and $r^2$ to obtain
\begin{align*}
\abs{\nabla u}^2(y)-1\le& C(N)\delta^{1/2}+\int\abs{\abs{\nabla u}^2(z)-1}\phi(z)p_t(y,z)\di\meas(z)\\
\le &C(N)\delta^{1/2}+C(N)\fint_{B_r(x)}\abs{\abs{\nabla u}^2-1}\di\meas\\
\le &C(N)\delta^{1/2}\, .
\end{align*}
From this we easily infer that 
\begin{equation}
\sup_{B_{r/2(x)}}\abs{\nabla u}\le 1+C(N)\delta^{1/2}\, .
\end{equation}
\medskip 

In order to show \eqref{eq:lipschitzdeltasplitt} it is enough to check that
for any $v\in \setR^k$ with $|v|=1$ it holds
\begin{equation}\label{z101}
	\sup_{B_{r/2(x)}}\abs{\nabla (v\cdot u)}\le 1+C(N)\delta^{1/2}\, ,
\end{equation}
Indeed \eqref{z101} yields
\begin{equation}\label{z100}
	|v\cdot (u(x) - u(y))| \le (1 + \delta^{1/2})\dist(x,y) \quad \text{ for any $x,y\in B_{r/4}(p)$}\, ,
\end{equation}
which implies \eqref{eq:lipschitzdeltasplitt} by taking the supremum w.r.t. $v\in \mathbb{S}^{k-1}$.

Now the key observation to prove \eqref{z101} is that $v\cdot u$ verifies (up to a constant) the same bounds of the components of the original $\delta$-splitting map. In particular it is harmonic and it satisfies 
\begin{equation}
\dashint_{B_r(p)}||\nabla (v\cdot u)|^2-1|\di \meas \le C(N) \delta\, ,
\end{equation}
therefore applying the argument already described for $u_a$ we get \eqref{z100}.

\end{remark}

 The first main result of this section will be \autoref{splitting vs isometry} below, where we prove the equivalence between the existence of an $\setR^k$-valued $\delta$-splitting map on a ball and the measured GH closeness of the ball with same center and comparable radius to the ball of a product with $\setR^k$. This statement will be proved arguing by compactness, starting from its rigid version \autoref{thm:functionalsplittinglocal}.

The second key result is \autoref{prop:transformation}, a version of the transformation theorem \cite[Proposition 7.7]{CheegerJiangNaber18} (see also \cite{CheegerNaber15} for a previous version with different assumptions) tailored for our purposes.

\subsection{Functional splitting theorem, local version}

In the rest of the note we will rely on the following functional version of the (iterated) splitting theorem in local form. With respect to the present literature the main novelty is the locality of the statement, which requires some cut-off arguments and the use of \autoref{thm:diffforumlaforvectorfields}, which relies in turn on \cite{GigliTamanini19}. The proof combines techniques from \cite{CheegerColding96} and \cite{Gigli13}.

\begin{theorem}\label{thm:functionalsplittinglocal}
	Let $(X,\dist,\meas)$ be an $\RCD(0,N)$ m.m.s. for some $N\ge 1$  and let  $p\in X$ be fixed. Assume that for some positive $k\in\setN$ there exists $u=(u_1,\ldots, u_k):B_6(p)\to \setR^k$  satisfying
	\begin{itemize}
	\item[(i)] $u(p)=0$;
	\item[(ii)] $|\nabla u_a|=1$ and $\Delta u_a=0$, $\meas$-a.e. in $B_5(p)$ for any $a=1,\ldots, k$;
	\item[(iii)] $\nabla u_a\cdot \nabla u_b=0$, $\meas$-a.e. in $B_5(p)$, for any $a\neq b$.
	\end{itemize}
    Then there exist a m.m.s. $(Z,\dist_Z,\meas_Z)$ and a function $f:B_1(p)\to Z$ such that
    \begin{equation}
    	(u,f): B_{1/k}(p)\to \setR^k\times Z
    \end{equation}
    is an isomorphism of metric measure spaces with its image.
    \end{theorem}
\begin{proof}
Let $\eta\in \Test(X,\dist,\meas)$ be a good cut-off function (see \autoref{lem:good_cut-off}) satisfying $\eta=1$ on $B_4(p)$ and $\eta=0$ on $X\setminus B_5(p)$. Let us define the vector fields $b_{a}:=\nabla (\eta u_a)$ and denote by $X^{a}$ their Regular Lagrangian flows, for $a=1,\ldots,k$. 
Notice that by the improved Bochner inequality with Hessian term \cite[Theorem 3.3.8]{Gigli18}, $\Hess (\eta u_a)=0$ in $B_4(p)$. Therefore thanks to \autoref{th: Lagrangianflows} (iii) we have a pointwise defined representative of $X^a_t(x)$ for $t\in (-1,1)$ and $x\in B_1(p)$ satisfying \eqref{eq:isometry}. This easily implies
\begin{equation}\label{z4}
	u_a(X_t^a(x))=t
	\quad\text{for any $x\in B_1(p)$ and $t\in (-1,1)$}\, .
\end{equation}

Let us now set $Z:=\set{u=0}$, $\dist_Z(x,y):=\dist(x,y)$ for $x,y\in Z$, and 
\begin{equation}
	\Phi: \setR^k\times Z\to X
	\quad \text{s.t.}\quad \Phi(t_1,\ldots,t_k,x):=X^1_{t_1}\circ X^2_{t_2}\circ \ldots \circ X^k_{t_k}(x)\, .
\end{equation} 
In order to conclude the proof it is enough to show that there exists a pointwise representative 
\begin{equation}\label{eq:claim1}
	\Phi: (-1/k,1/k)^k\times (B_1(p)\cap Z)\to X
	\quad \text{which is an isometry with its image.}
\end{equation}
Observe indeed that $B_{1/k}(p)\subset \Phi((-1/k,1/k)^k\times (B_1(p)\cap Z))$, since for any $y\in B_{1/k}(p)$ we can set 
\[
\pi(y):=\Phi(-u_1(y),\ldots,-u_k(y),y)\in Z\cap B_1(p)
\]	
and $t_a:=u_a(\pi(y))\in (-1/k,1/k)$ and check, by means of \eqref{z4}, that $\Phi(t_1,\ldots,t_k, \pi(y))=y$. Finally we notice that $\Phi^{-1}: B_{1/k}(p)\to \setR^k\times Z$ is the sought map, since it is an isometry and can be written as $\Phi^{-1}=(u,f)$ for some $f:B_{1/k}(p)\to Z$, thanks to \eqref{z4}. Moreover setting $\meas_Z:=\pi_{\#}( \meas\res B_{1/k}(p))$, one can easily check that
\[
\Phi_{\#} (\Leb^k \times \meas_Z) = \meas\quad \text{on $B_{1/k}(p)$}\, .
\]

The proof of \eqref{eq:claim1} is divided in three steps.
\\
\textbf{Step 1.} There exists a pointwise representative of $\Phi$ on $(-1/k,1/k)^{k}\times (B_1(p)\cap Z)$ such that, for any $x,y\in B_{1}(p)\cap Z$ and $(t_1,\ldots,t_k)\in (-1/k,1/k)^k$ it holds
\begin{equation}\label{z5}
	\dist(\Phi(t_1,\ldots,t_k,x),\Phi(t_1,\ldots,t_k,y))=\dist(x,y)
	\quad\text{and}\quad 
	\dist(X^a_t(x),x)=t\, ,
\end{equation}
for $a=1,\ldots,k$.

As we have already remarked, there exists a pointwise defined representative of $X^a_t(x)$ for $t\in (-1,1)$ and $x\in B_1(p)$  satisfying \eqref{eq:isometry}, therefore the first identity in \eqref{z5} follows.\\ 
Concerning the second equality, observe that, since $|\nabla(\eta u_a)|=1$ in $B_4(p)$ we have that $\dist(X^a_t(x),x)\le t$ for $t\in (-1,1)$ and $x\in B_1(p)$. Moreover \eqref{z4} and the fact that $u_a$ is $1$-Lipschitz in $B_1(p)$ give
\begin{equation}
	t=|u_a(X^a_t(x))-u_a(x)|\le \dist(X^a_t(x),x)\, ,
	\quad \text{for $x\in Z\cap B_1(p)$ and $t\in (-1,1)$}\, .
\end{equation}
Combining the two inequalities also the second equality in \eqref{z5} follows.

\textbf{Step 2.} In this step we are going to prove that for any $a\in \set{1,\ldots,k}$, any $x,y\in B_1(p)$ and any $t\in (0,1)$ it holds
\begin{equation}\label{z6}
	\frac{1}{2}\dist^2(X^a_t(x),y)-\frac{1}{2}\dist^2(x,y)=\int_0^t \left( u_a(X_s^a(x))-u_a(y) \right) \di s\, .
\end{equation}
To this aim let us fix $a\in \set{1,\ldots,k}$, $x,y\in B_1(p)$ and $r>0$ with the property that $B_r(x)\cup B_r(y)\subset B_1(p)$. Then let us define
\begin{equation}
	\mu^r:=\frac{1}{\meas(B_r(x))} \meas\res B_r(x)
	\quad \text{and}\quad
	\nu^r:=\frac{1}{\meas(B_r(y))} \meas \res B_r(y)\, .
\end{equation}

Let us set $\rho_t^r:=\left({X^a_t}\right)_{\#}\mu^r$ and observe that for any function $f \in \Lip(X,\dist)$ we have
\begin{equation}
	\frac{\di}{\di t} \int f\di \rho_t^r=\frac{\di}{\di t} \int f(X_t^a)\di \mu_r=\int \nabla f\cdot \nabla u \di \rho_t^r\, ,
\end{equation}
namely $\rho_t^r=\left({X^a_t}\right)_{\#}\mu^r$ solves the continuity equation associated to $\nabla u$ (cf \autoref{remark:Continuity equations and flow maps}).
Therefore \autoref{lemma:derivativealongflow} guarantees that $t\mapsto  \frac{1}{2}W_2^2(\rho_t^r,\nu_r)$ is absolutely continuous and 
\begin{equation}\label{z9}
\frac{\di }{\di t}\frac{1}{2}W_2^2(\rho_t^r,\nu^r)
=\int \nabla u\cdot \nabla \phi_t \di \rho_t^r
\quad\text{for a.e. $t\in (0,1)$}\, ,
\end{equation}
where $\phi_t$ is any optimal Kantorovich potential from $\rho_t^r$ to $\nu^r$.\\
Let us now fix $t\in (0,1)$ such that \eqref{z9} holds true. Denote by $(\eta^{r,t}_s)_{s\in[0,t]}$ the $W_2$-geodesic connecting $\nu_r$ to $\rho_t^r$.
From \autoref{prop:derivativealonggeodesic1} 
we get
\begin{equation}\label{z10}
	\frac{\di }{\di s}\restr_{s=1} \int u \di \eta^{r,t}_s= \frac{1}{t}\int \nabla u\cdot \nabla \phi_t \di \rho_t^r\, ,
\end{equation}
where $\phi_t$ is any optimal Kantorovich potential from $\rho_t^r$ to $\nu^r$.\\
Combining \eqref{z9} and \eqref{z10} we deduce
\begin{equation}\label{z11}
	\frac{\di }{\di t}\frac{1}{2}W_2^2(\left({X^a_t}\right)_{\#}\mu^r,\nu^r)
	=t\frac{\di }{\di s}\restr_{s=1} \int u \di \eta^{r,t}_s
	\quad \text{for a.e. $t\in (0,1)$}\, .
\end{equation}
Moreover, since $\Hess u=0$ in $\supp \eta_s^{r,t}$ for any $0\le s\le t\le 1$, \autoref{thm:diffforumlaforvectorfields} implies that $s\mapsto \int u \di \eta^{r,t}_s$ is affine. Therefore
\begin{equation}
	t\int u \di \eta^{r,t}_s=(t-s) \int u\di \nu^r+ s\int u\circ X^a_t \di \mu^r\, ,
	\quad \text{for any $0\le s\le t\le 1$,}
\end{equation}
that, along with \eqref{z11}, yields
\begin{equation}\label{z13}
	\frac{1}{2}W_2^2(\left({X^a_t}\right)_{\#}\mu^r,\nu^r)-\frac{1}{2}W_2^2(\mu^r,\nu^r)
    =\int_0^t \left( \int u\circ X^a_s\di \mu^r-\int u\di \nu^r\right)\di s\, .
\end{equation}
Finally, \eqref{z6} follows from \eqref{z13} by continuity letting $r\to 0$.

\textbf{Step 3.} We conclude the proof of \eqref{eq:claim1} by showing that
\begin{equation}\label{z14}
	\dist^2(\Phi(t_1,..,t_k,x),\Phi(s_1,\ldots,s_k,y))
	=\dist^2(x,y)+|t_1-s_1|^2+\ldots+|t_k-s_k|^2\, ,
\end{equation}
for any $x,y\in Z\cap B_1(p)$ and any $s_a,t_a\in (-1/k,1/k)$, for $a=1,\dots,k$.\\
In order to do so let us assume without loss of generality that $t_1\ge s_1$ and set $\bar x:=X^2_{t_2}\circ\ldots\circ X^k_{t_k}(x)$ and $\bar y:=X^2_{s_2}\circ \ldots\circ X^k_{s_k}$ if $k\ge 2$ and $\bar x:=x$, $\bar y:=y$ otherwise.\\ 
By exploiting the semigroup property (ii) in \autoref{th: Lagrangianflows}, \textbf{Step 1} and \textbf{Step 2}, we get
\begin{align*}
    	\dist^2(\Phi(t_1,..,t_k,x)&,\Phi(s_1,\ldots,s_k,y))-\dist^2(\bar x,\bar y)
    	=\dist^2(X_{t_1}(\bar x),X_{s_1}(\bar y))-\dist^2(\bar x,\bar y)\\
    	&=\dist^2(X_{t_1-s_1}(\bar z),\bar y)-\dist^2(\bar x,\bar y)\\
    	&=2\int_0^{t_1-s_1}\left( u_1(X^1_s(\bar x))-u_1(\bar y)\right) \di s\\
    	&\overset{\eqref{z4}}{=} |t_1-s_1|^2+(t_1-s_1)(u_1(\bar x)-u_1(\bar y))\, .
\end{align*}
Observe now that $u_1(\bar x)=u_1(\bar y)=0$ since the function $t\mapsto u_1(X^a_t(z))$ is constant for $t\in (-1,1)$, $z\in B_2(p)$ and $a\neq 1$. Indeed, taking the derivative w.r.t. $t\in (-1,1)$ and using (iii) in our assumptions we have 
\begin{equation}
	\frac{\di}{\di t} u_1(X^a_t(z))=\nabla u_1\cdot \nabla u_a (X_t^a(z))=0\, ,
\end{equation}
for a.e. $t\in(-1,1)$ and for a.e. $z\in B_2(p)$. The statement can then be proved for any time and starting point by a continuity argument.
It follows that
\begin{equation}
	\dist^2(\Phi(t_1,..,t_k,x),\Phi(s_1,\ldots,s_k,y))=\dist^2(\bar x,\bar y)+|t_1-s_1|^2
\end{equation}
and a simple induction argument gives \eqref{z14}. 
\end{proof}

Below we specialize \autoref{thm:functionalsplittinglocal} to the case in which $(X,\dist,\haus^N)$ is a noncollapsed $\RCD(K,N)$ space and the splitting map has $N-1$ components. In this case we are going to prove that, as expected, the factor $Z$ is one dimensional.

\begin{theorem}\label{thm:functionalsplittingcodimension1}
Let $(X,\dist,\haus^N)$ be an $\RCD(0,N)$ m.m.s. for some natural $2\le N<\infty$ and let $p\in X$ be fixed. Assume that there exists a $0$-splitting map 
\[
u=(u_1,\ldots, u_{N-1}):B_6(p)\to \setR^{N-1}\, .
\] 
Then there exist a m.m.s. $(Z,\dist_Z,\haus^1)$, with $(Z,\dist_Z)$ isometric to the ball of a one dimensional Riemannian manifold (possibly with boundary), and a map $f:B_{1/(N-1)}(p)\to Z$ such that 
\[
(u,f):B_{1/(N-1)}(p)\to B_{1/(N-1)}^{\setR^{N-1}\times Z}((0,z_0))
\] 
is an isomorphism of metric measure spaces.

Moreover, up to an additive constant, $f$ coincides with the signed distance function from the level set $\{f=z_0\}$.
\end{theorem}

\begin{remark}
In particular, $Z$ is isometric, in the sense of Riemannian manifolds, to either a circle or a connected closed interval $I\subseteq \setR$, possibly with infinite length.
\end{remark}

\begin{proof}

Let us start by applying the local functional splitting \autoref{thm:functionalsplittinglocal} to get the metric measure space $(Z,\dist_Z,\meas_Z)$ and the map $f:B_{1/(N-1)}(p)\to Z$.
 
By a slight modification of the proof of \cite[Corollary 5.30]{Gigli13}, we can prove a weaker version of the $\CD(0,N)$ condition for the space $(Z,\dist_Z,\meas_Z)$. More precisely we can check that, for any $\mu_0\in \Prob_2(Z)$ satisfying $\mu_0\ll \meas_Z$ and $\supp \mu_0 \subset B_{1/(N-1)}(z_0)$ there exists $r>0$ such that for any $\mu_1\in \Prob_2(X)$ absolutely continuous w.r.t. $\meas_Z$ and supported on $B_r(\supp \mu_0)$ one has a unique $W_2$-geodesic connecting $\mu_0$ and $\mu_1$ which satisfies the defining inequality for the $\CD(0,N)$ condition.

Next we observe that, as a consequence of the discussion above, of the isometry between $B_{1 /(N-1)}(p)$ and the split ball and of the noncollapsing assumption, all the metric measured tangents to $(Z,\dist_Z,\meas_Z)$ are either lines or half lines as metric spaces. By the structure theory of $\RCD$ spaces, the tangent is unique and a line for $\meas_Z$-a.e. $z\in B_1^Z(z_0)$. Moreover, by the noncollapsing assumption, at points where there is a line in the tangent the tangent is unique, since they correspond to points on the starting space where the tangent is $\setR^N$.

Adapting the arguments of \cite{KitabeppuLakzian16} (see also \cite{Schultz19} for a recent generalization with simplified arguments relying on optimal transport tools), it is possible to prove that at points of $Z$ where there is a line in the tangent there is a small ball isometric to the Euclidean one. Moreover, at the other points the tangent is still unique and isometric to a half line pointed at the extreme (otherwise there would be a full line in the tangent and we would be in the previous case). Arguing as in the proof of \cite[Theorem 3.1]{Schultz19} we conclude that each point in $Z$ has a neighborhood isometric either to $(-\eps,\eps)$ or to $[0,\eps)$. Hence the metric conclusion follows from the characterization of one dimensional Riemannian manifolds.

The conclusion about the measure can be achieved relying on the fact that 
\[
(u,f):B_{1/(N-1)}(p)\to B_{1/(N-1)}^{\setR^{N-1}\times Z}((0,z_0))
\]
is an isomorphism of metric measure spaces and the measure on $(X,\dist)$ is $\haus^N$.

The last conclusion in the statement can be easily proved given the previous ones.
\end{proof}

\begin{remark}
The converse of \autoref{thm:functionalsplittinglocal} is trivially verified. Indeed, if the space is locally isomorphic to a product with Euclidean factor then the coordinates of the Euclidean factor are easily seen to verify properties (i)--(iii).
\end{remark}

\subsection{$\delta$-splitting maps and $\eps$-GH isometries}

Arguing by compactness we now obtain an approximated version of \autoref{thm:functionalsplittinglocal}. As in the rigid case the novelty with respect to the literature of $\RCD$ spaces is the ease of producing locality of the statement, cf. with \cite{BruePasqualettoSemola19}. We refer to \cite[Lemma 1.21]{CheegerNaber15} and \cite[Theorem 4.11]{CheegerJiangNaber18} for similar statements for Ricci limits. 

\begin{theorem}[$\delta$-splitting vs $\eps$-GH isometry]
	\label{splitting vs isometry}
	Let $1\le N<\infty$ be fixed. 
	
	\begin{itemize}
		\item[(i)] For every $0<\delta<1/2$ and $\eps\le \eps(N,\delta)$ the following holds. If $(X,\dist,\meas)$ is an $\RCD(-\eps(N-1),N)$ m.m.s.  satisfying
		\begin{equation}\label{eq:zzz5}
		\dist_{mGH}(B_2(p),B_2^{\setR^k\times Z}(0,z))\le \eps
		\end{equation}
		for some integer $k$, some $p\in X$ and some pointed m.s. $(Z,\dist_Z)$, then
		there exists a $\delta$-splitting map $u=(u_1,\ldots,u_k):B_1(p)\to \setR^k$.

		\item[(ii)] For every $\eps>0$ and $\delta < \delta(N,\eps)$ the following holds. If  $(X,\dist,\meas)$ is a normalised $\RCD(-\delta(N-1),N)$ m.m.s. and there exists a $\delta$-splitting map $u:B_6(p)\to \setR^k$ for a given $p\in X$, then
	    \begin{equation}
	     \dist_{GH}(B_{1/k}(p),B^{\setR^k\times Z}_{1/k}(0,z))< \eps
	    \end{equation}
    	for some pointed metric space $(Z,\dist_Z,z)$. Moreover, there exists $f:B_{1}(p)\to Z$ such that
    	\begin{equation}
    	(u-u(p),f):B_{1/k}(p)\to B_{1/k}^{\setR^k\times Z}(0,z)
    	\quad \text{is an $\eps$-GH isometry.}
	    \end{equation} 
	 
	    \item[(iii)] If we additionally assume that $(X,\dist,\haus^N)$ is $\RCD(-\delta(N-1),N)$ noncollapsed with $\haus^N(B_1(p))>v>0$, $k=N-1$, and $\delta<\delta(N,v,\eps)$, then $(Z,\dist_Z,\haus^1)$ in (ii) can be chosen to be the ball of a one dimensional Riemannian manifold, possibly with boundary.
	
    \end{itemize}
\end{theorem}

\begin{proof}
The first part of the statement can be proved arguing as in the proof of \cite[Proposition 3.9]{BruePasqualettoSemola19}, relying on the local convergence and stability results obtained in \cite{AmbrosioHonda18}.
	
	Let us now prove the second conclusion. Arguing by contradiction, for any $n\in\setN$, we can find a normalised pointed $\RCD(-1/n,N)$ m.m.s. $(X_{n},\dist_n,\meas_n,p_n)$ and a $1/n$-splitting map $u_n:B_6(p_n)\to \setR^k$ such that $u_n(p_n)=0$ and the following property holds: for any pointed metric space $(Z,\dist_Z,z)$ and any function $f:B_{1/k}(p_n) \to B_{1/k}^Z(z)$, the map
	\begin{equation}\label{z16}
		(u_n,f):B_{1/k}(p_n)\to B_{1/k}^{\setR^k\times Z}((0,z))\quad
		\text{is not an $\eps$-GH equivalence.}
	\end{equation}
	Thanks to the stability and compactness of the $\RCD$ condition we can find a pointed $\RCD(0,N)$ m.m.s. $(X_{\infty},\dist_{\infty},\meas_{\infty},p_{\infty})$ such that, up to extract a subsequence (that we do not relabel), it holds
	\begin{equation}\label{eq:comp}
		(X_n,\dist_n,\meas_n,p_n)\to (X_{\infty},\dist_{\infty},\meas_{\infty},p_{\infty})
		\quad\text{in the pmGH topology}.
	\end{equation}
	Arguing as in \cite[Proposition 3.7]{BruePasqualettoSemola19} we can assume that $u_n\to u$ uniformly in $B_6(p_{\infty})$, where $u$ is a $C(N)$-Lipschitz and harmonic function in $B_6(p_{\infty})$ satisfying $\nabla u_a\cdot \nabla u_b=\delta_{a,b}$, $\meas_{\infty}$-a.e. in $B_2(p_{\infty})$ for $a,b=1,\ldots,k$.
	Thanks to \autoref{thm:functionalsplittinglocal} we can find a m.s. $(Z,\dist_Z)$ and a function $f:B_{1/k}(p_{\infty})\to Z$ such that
	\begin{equation}\label{z15}
		(u,f):B_{1/k}(p_{\infty})\to \setR^k\times Z
		\quad\text{is an isometry with its image.}
	\end{equation}
	Let us conclude the proof by showing that \eqref{z15} contradicts \eqref{z16}. Let us consider a sequence of $1/n$-isometries $\Psi_n:B_{1/k}(p_n)\to B_{1/k}(p_{\infty})$. By \cite[Lemma 27.4]{Villani09} we can suppose that $\Psi_n$ converge to an isometry from $B_{1/k}(p_{\infty})$ into itself. Up to composing with the inverse of this isometry we assume that the maps $\Psi_n$ converge to the identity map of $B_{1/k}(p_{\infty})$. Set $f_n:= f\circ \Psi_n$.
	Next we claim that
	\begin{equation}\label{eq:apiso}
		(u_n,f_n):B_{1/k}(p_n)\to B_{1/k}^{\setR^k\times Z}((0,z))
		\quad \text{is a $\eps$-GH isometry for $n\in \setN$ big enough,}
	\end{equation}
which will contradict \eqref{z16} yielding the sought conclusion.

Being $f$ continuous (actually $1$-Lipschitz since $(u,f)$ is an isometry with its image), one can easily prove that $(u_n,f_n)\to (u,f)$, therefore the image of $(u_n,f_n)$ is $\eps$-dense in $B_{1/k}^{\setR^k\times Z}((0,z))$ for any $n$ big enough. It remains just to check that
\begin{equation}
\abs{\dist^2(x,y)-\abs{u_n(x)-u_n(y)}^2-\abs{f_n(x)-f_n(y)}^2}\le \eps\quad \text{for any $x,y\in B_{1/k}(p_n)$}
\end{equation}
when $n$ is big enough.
We argue by contradiction. If the conclusion were false we could find sequences $(x_n)$ and $(y_n)$ in $B_{1/k}(p_n)$ such that the defining condition of $\eps$-isometries does not hold for these points, i.e.
\begin{equation}\label{eq:notepsiso}
\abs{\dist^2(x_n,y_n)-\abs{u_n(x_n)-u_n(y_n)}^2-\abs{f_n(x_n)-f_n(y_n)}^2}>\eps\, .
\end{equation}
By compactness, up to extracting a subsequence that we do not relabel, we can assume that $x_n$ converge to $x\in B_{1/k}(p_{\infty})$ and $y_n$ converge to $y\in B_{1/k}(p_{\infty})$. It is easily verified that $x\neq y$, thanks to \eqref{eq:notepsiso} and to the Lipschitz regularity of $u_n$ and $f$. Passing to the limit \eqref{eq:notepsiso}, taking into account the uniform convergence of $u_n$ to $u$ and the convergence of $\Psi_n$ to the identity map together with the continuity of $f$, we get 
\begin{equation}
\abs{\dist^2(x,y)-\abs{u(x)-u(y)}^2-\abs{f(x)-f(y)}^2}\ge\eps\, ,
\end{equation}
that contradicts \eqref{z15}.

The additional conclusion under the noncollapsing assumption can be obtained relying on \autoref{thm:functionalsplittingcodimension1}. Taking into account the lower bound on the volume, the pmGH convergence in the contradiction argument above improves to noncollapsed convergence. Therefore the limit space is $\RCD(0,N)$ noncollapsed.
\end{proof}

	\begin{remark}
		When $(X,\dist, \haus^N)$ is a noncollapsed $\RCD(-\delta(N-1),N)$ space satisfying $\haus^N(B_1(p))>v$, then in \autoref{splitting vs isometry} we can relax \eqref{eq:zzz5} to
		\begin{equation}
		\dist_{GH}(B_2(p),B_2^{\setR^k\times Z}(0,z))\le \eps,
		\end{equation}
	    provided $\delta\le \delta(N,v,\eps)$.
	\end{remark}

\begin{remark}\label{rm:epssplitimplepsisoimpr}
In the case of maximal dimension we can slightly improve upon the implication between $\delta$-splitting and $\eps$-isometry. In particular the following holds:
for any $\eps>0$ there exists $\delta=\delta(\eps,N)>0$ such that if $(X,\dist,\meas)$ is an $\RCD(-\delta(N-1),N)$ space, $B_{3/2}(p)\subset X$,
\begin{equation}
\dist_{GH}(B_{3/2}(p),B_{3/2}^{\setR^N}(0))<\delta
\end{equation}	
and $u:B_1(p)\to\setR^N$ is a $\delta$-splitting map, then $u:B_1(p)\to\setR^N$ is an $\eps$-isometry.\\
The same statement holds for splitting maps with $N-1$ components in case we put $\setR^N_+$ in place of $\setR^N$. 

This statement can be proved relying on the local convergence and stability results of \cite{AmbrosioHonda18}, taking into account the fact that local spectral convergence holds for all radii when the limit space is the Euclidean space (or, more in general, a metric measure cone).

Notice that the main improvement is that we do not need to worsen the radius to pass from the $\delta$-splitting condition to the $\eps$-isometry. Moreover we can allow not only for harmonic $\delta$-splitting functions but also for functions with small Laplacian in $L^2(B_1(p))$, cf. with \autoref{def:almostsplittingmap} below.
\end{remark}	

	For the study of the topological structure of $\RCD$ spaces with boundary in \autoref{sec:topologicalstructureuptotheboundary} we will need a slightly less restrictive notion of $\delta$-splitting map.
	
	\begin{definition}\label{def:almostsplittingmap}
		Fix $\delta>0$. Let $(X,\dist,\meas)$ be an $\RCD(-\delta(N-1),N)$ m.m.s. and $p\in X$.
		We say that $u:=(u_1,\ldots,u_k) :B_r(p)\to \setR^k$ is a $\delta$-\textit{almost splitting map} provided it satisfies:
		\begin{itemize}
			\item[(i)] $|\nabla u_a|< C(N)$;
			\item[(ii)] 
			\begin{equation}
			\sum_{a,b=1}^k \fint_{B_r(p)} |\nabla u_a\cdot \nabla u_b-\delta_{a,b}|\di \meas + \sum_{a=1}^k r^2 \fint_{B_r(p)} (\Delta u_a)^2 \di \meas <\delta\, .
		\end{equation}
		\end{itemize}
	\end{definition}
    
    Arguing as in \autoref{rm:hessiansuper} one can easily check through Bochner's inequality that a $\delta$-almost splitting map $u: B_{2r}(p)\to \setR^k$ satisfies
    \[
     r^2 \fint_{B_r(p)} |\Hess u|^2\di \meas \le C(N)\delta\, .
    \]
    Therefore, the only meaningful difference between the notion of $\delta$-splitting map and $\delta$-almost splitting map is that the latter is not harmonic but enjoys a scale invariant $L^2$-smallness of the Laplacian.
    
    \begin{remark}\label{rmk:almostdeltasplitting}
    	It is immediately seen that \autoref{splitting vs isometry} and \autoref{rm:epssplitimplepsisoimpr} still hold when relaxing the assumptions by considering $\delta$-almost splitting maps in place of $\delta$-splitting maps.
    \end{remark}

\subsection{Transformation theorem}\label{subsec:transformation}

In \cite{CheegerNaber15} a key result in order to prove the codimension 4 conjecture for noncollapsed limits of manifolds with bounded Ricci curvature was the so-called transformation theorem. Given an $(N-2,\delta(\eps))$-splitting map $u:B_1(p)\to\setR^{N-2}$, \cite[Theorem 1.32]{CheegerNaber15} provides conditions guaranteeing the existence of a lower triangular matrix with positive entries $T$ such that $Tu:B_r(x)\to\setR^{N-2}$ is an $(N-2,\eps)$-splitting map for $0<r<1$.

In \cite{CheegerJiangNaber18} (see in particular Proposition 7.8) a geometric version of the transformation theorem was proved, in order to study singular strata of any codimension on Ricci limits.  In particular, the weak version of the estimate proven in \cite{CheegerJiangNaber18} was that given a $(k,\delta)$-splitting map on $B_1(p)$, there is a lower triangular matrix with positive entries $T$ such that $Tu:B_r(x)\to\setR^k$ remains $(k,\eps)$-splitting as long as $B_s(p)$ is $k$-symmetric and far from being $(k+1)$-symmetric, for any $r\le s \le 2$.

Here we provide a version of the geometric transformation theorem tailored for the purpose of studying the structure of noncollapsed $\RCD$ spaces with boundary. We focus the attention only $\delta$-boundary balls (see \autoref{def:boundaryball}) and $(N,\delta)$-symmetric balls (corresponding to $k=n-1,n$ in \cite{CheegerJiangNaber18}) and, for technical reasons, we work with possibly non harmonic $\delta$-splitting maps (cf. with \autoref{def:almostsplittingmap}) rather than harmonic $\delta$-splitting maps. Up to these small variants the argument presented here is the one from \cite{CheegerJiangNaber18}.

	\begin{proposition}[Transformation]\label{prop:transformation}
		Let $1\le N<\infty$ be a fixed natural number. For any $\eps>0$ there exists $\delta(N,\eps)>0$ such that for any $\delta<\delta(N,\eps)$, for any $\RCD(-\delta^2(N-1),N)$ space $(X,\dist,\haus^N)$ and for any $x\in X$ the following hold. 
		\begin{itemize}
			\item If $B_s(x)$ is a $\delta^2$-boundary ball for any $r_0\le s\le 1$ and $u:B_2(x)\to\setR^{N-1}$ is a $\delta$-almost splitting map, then for each scale $r_0\le s\le 1$ there exists an $(N-1)\times (N-1)$ lower triangular matrix $T_{s}$ such that
		\begin{itemize}
			\item[i)] $T_su:B_s(x)\to\setR^{N-1}$ is an $\eps$-almost splitting map on $B_s(x)$;
			\item[ii)] $\fint_{B_s(x)}\nabla (T_su)^a\cdot\nabla (T_su)^b\di\haus^N=\delta_{ab}$;
			\item[iii)] $\abs{T_s\circ T_{2s}^{-1}-\Id}\le \eps$.
		\end{itemize}
		\item If $B_s(x)$ is an $(N,\delta^2)$-symmetric ball for any $r_0\le s\le 1$ and $u:B_2(x)\to\setR^{N}$ is a $\delta$-almost splitting map, then for each scale $r_0\le s\le 1$ there exists an $N\times N$ lower triangular matrix $T_{s}$ such that
		\begin{itemize}
			\item[i)] $T_su:B_s(x)\to\setR^{N}$ is an $\eps$-almost splitting map on $B_s(x)$;
			\item[ii)] $\fint_{B_s(x)}\nabla (T_su)^a\cdot\nabla (T_su)^b\di\haus^N=\delta_{ab}$;
			\item[iii)] $\abs{T_s\circ T_{2s}^{-1}-\Id}\le \eps$.
		\end{itemize}
			\end{itemize}
	\end{proposition}

 We postpone the proof of the transformation \autoref{prop:transformation} after some technical lemmas.
The first one is about the very rigid form of harmonic functions with almost linear growth on the Euclidean space and half-space. It can be easily proved thanks to the explicit knowledge of entire harmonic functions (cf. with \cite[Lemma 7.8]{CheegerJiangNaber18}, dealing with a much more general case) and we omit the details.

\begin{lemma}\label{lemma:harmalmostlinear}
	Let $1\le N<\infty$ be a fixed natural number, then there exists $\eps=\eps(N)>0$ such that the following holds. Let $(X,\dist,\haus^N)$ be isomorphic either to the Eucildean space $\setR^N$ or to the half-space $\setR^N_+$. Then any harmonic function $u:X\to\setR$ with almost linear growth, $\abs{u(x)}\le C\abs{x}^{1+\eps}+C$ for any $x\in X$, is linear and induced by an $\setR$ factor.
\end{lemma}

 The second lemma is about estimates for the transformation matrixes, given their existence. We refer to \cite[Lemma 7.9]{CheegerJiangNaber18} for its proof, which is a simple inductive argument relying on the uniqueness of Cholesky decompositions \cite{GolubVanLoan13}.\\ 
Below we shall denote by $\abs{\cdot}_{\infty}$ the $L^{\infty}$-norm on matrixes.

\begin{lemma}\label{lemma:growth}
Under the assumptions of \autoref{prop:transformation}, there exists a constant $C=C(N)>0$ such that, if $T_{s}$ and $T_{2s}$ are matrixes verifying (i) and (ii) at scale $s$ and $2s$ respectively, then automatically
\begin{equation}\label{eq:matrixest}
\abs{T_{s}\circ T^{-1}_{2s}-\Id}_{\infty}\le C(N)\eps.
\end{equation}
	\end{lemma}

Given \autoref{lemma:growth} and arguing inductively as in \cite{CheegerJiangNaber18} it is then possible to prove a growth estimate for the transformation matrixes, once we assume that they exist.

\begin{corollary}\label{cor:growth}
	Under the assumptions of \autoref{prop:transformation}, there exists a constant $C=C(N)>0$ such that, if $T_{\bar{r}}$ and $T_{r}$ are matrixes verifying (i) and (ii) at scales $0<\bar{r}<r$ respectively, then
	\begin{equation}\label{eq:matrixestind}
	\abs{T_r^{-1}\circ T_{\bar{r}}}_{\infty}\le \left(\frac{r}{\bar{r}}\right)^{C\eps}.
	\end{equation}
	\end{corollary}

\vspace{.1cm}

\begin{proof}[Proof of \autoref{prop:transformation}]
	Let us treat only the second case of $(N,\delta^2)$-symmetric balls. The case of boundary balls can be handled with the same argument. Observe also that we only need to prove (i) and (ii), since (iii) will follow from \autoref{lemma:growth}.
		
    We wish to get the sought conclusion arguing by contradiction. We suppose that there exists $0<\eps_0\ll 1$ such that the following hold:
\begin{itemize}
	\item[a)] there exist pointed $\RCD(-\delta_i,N)$ spaces $(X_n,\dist_n,\haus^N,x_n)$ such that the balls $B_r(x_n)$ are $(N,\delta_i^2)$-symmetric for any $r_n\le r\le 1$, and $\delta_n$-almost splitting maps $u_n:B_2(x_n)\to\setR^N$, for a sequence $\delta_n\downarrow 0$;
	\item[b)] there exist $s_n>r_n$ such that for any $s_n<r\le 1$ there exist lower triangular matrixes $T_{x_n,r}$ such that $T_{x_n,r}u:B_r(x_n)\to\setR^{N}$ is an $\eps_0$-splitting map on $B_r(x_n)$ and 
	\begin{equation}\label{eq:normb}
	\fint_{B_r(x_n)}\nabla (T_{x_n,r}u)^a\cdot\nabla (T_{x_n,r}u)^b\di\haus^N=\delta_{ab}\, ;
	\end{equation}
	\item[c)] no such mapping $T_{x_n,s_n/10}$ exists on $B_{s_n/10}(x_i)$.
\end{itemize}

Let us start by noticing that it must hold $s_{n}\downarrow 0$ as $n\to\infty$, otherwise we would easily reach a contradiction.\\
Then let us consider the scaled pointed spaces $\tilde{X}_n:=(X_n,s_n^{-1}\dist_n,\haus^N,x_n)$. Observe that, since $B^{X_n}_r(x_n)$ is $(N,\delta_n)$-symmetric for any $r_i\le r\le1$, on the scaled space it holds that $B_{r}^{\tilde{X}_n}(x_n)$ is $(N,\delta_n)$-symmetric for any $r_n/s_n\le 1\le r\le s_n^{-1}$. Since $s_n\to 0$ as $n\to\infty$, we infer that $\tilde{X}_n$ converge to $\setR^N$ in the pGH (and a posteriori pmGH) topology.    

Let us now set 
\begin{equation}
v_n:=s_n^{-1}T_{x_n,s_n}(u_n-u_n(x_n))\, . 
\end{equation}
Observe that 
\begin{equation}\label{eq:scalesmall}
\fint_{B^{\tilde{X}_n}_r(x_n)}(\Delta v_n)^2\di\haus^N\le C(r) \delta_n\,\;\;\;\text{for any $1\le r\le s_n^{-1}$ ,}
\end{equation}
thanks to \autoref{cor:growth} and the fact that $u_n$ is a $\delta_n$-almost splitting map ( cf. \autoref{def:almostsplittingmap} (ii)).
Moreover $v_n$ has almost linear growth, $\abs{v_n(x)}\le C\dist(x_n,x)^{1+\eps}+C$ for any $x$ such that $\dist(x_n,x)\le s_n^{-1}$, thanks to \autoref{cor:growth}, and it verifies
\begin{equation}\label{eq:norm}
\fint_{B_1^{\tilde{X}_n}(x_n)}\nabla v_n^a\cdot\nabla v_n^b\di\haus^N=\delta_{ab}\, ,
\end{equation}	
by \eqref{eq:normb}. 

By \cite{AmbrosioHonda18} and \eqref{eq:scalesmall} we obtain that $v_n$ converge locally in $W^{1,2}$ and locally uniformly to a harmonic function $v:\setR^N\to\setR^N$ with almost linear growth. Passing to the limit \eqref{eq:norm} and taking into account \autoref{lemma:harmalmostlinear}, we get that $v$ is an orthogonal transformation of $\setR^N$.

Localizing the $W^{1,2}$-convergence (see \cite[Theorem 1.5.7, Proposition 1.3.3.]{AmbrosioHonda}), we obtain
\begin{equation}
\lim_{n\to\infty}\fint_{B_1^{\tilde{X}_n}(x_n)}\abs{\nabla v_n^a\cdot\nabla v_n^b-\delta_{ab}}\di\haus^N=0\, .
\end{equation}
Therefore, taking into account also \eqref{eq:scalesmall}, we infer that $v_i:B_1^{\tilde{X}_n}(x_n)\to\setR^N$ becomes an $\eps_n$-almost splitting map where $\eps_n\to 0$ as $n\to\infty$.\\
Hence for each $1/10\le r\le 1$ and any sufficiently large $n$ there exists a lower triangular $N\times N$ matrix $A_{n,r}$ with $\abs{A_{n,r}-\Id}\le C(N)\eps_n$ and 
\begin{equation}
\fint_{B_r^{\tilde{X}_n}(x_n)}\nabla (A_{n,r}v_n)^a\cdot\nabla (A_{n,r}v_n)^b\di\haus^N=\delta_{ab}\, .
\end{equation}
In particular, for any sufficiently large $n$, $A_{n,r}v_n:B_r^{\tilde{X}_n}(x_n)\to\setR^N$ is an $\eps_0$-splitting map for any $1/10\le r\le 1$ satisfying the orthogonality condition (ii) in the statement. This contradicts the minimality of $s_i$ (cf. with condition (c)), scaling back to the starting spaces $X_n$.
This finishes the proof of the existence of transformation matrixes, the growth estimate (iii) can be obtained by \autoref{lemma:growth}, as we already argued.
\end{proof}

\vspace{.4cm}

\section{Neck regions}\label{sec:neckregion}

This section is dedicated to the introduction and the analysis of \textit{neck regions}. We first provide the relevant definition tailored for the study of singularities of codimension one for noncollapsed $\RCD$ spaces. Then in \autoref{subsec:structneck} and \autoref{subsec:existenceneck} we provide structural results for neck regions and an existence result, respectively.

The notion of neck region has been introduced in \cite{JiangNaber16} and \cite{NaberValtorta19} to study $L^2$-curvature bound for spaces with bounded Ricci curvature and the energy identity for Yang-Mills connections. Its use has been crucial also in \cite{NaberValtorta17} and, more recently, in \cite{CheegerJiangNaber18}, for the rectifiability of singular sets in arbitrary codimension on noncollapsed Ricci limits.

In the following we shall denote by $\setR_+^N:=\set{x\in \setR^N:\ x_N\ge 0}$ the Euclidean half space of dimension $N\ge 1$.  

\begin{definition}[Boundary ball]\label{def:boundaryball}
	Let $1\le N<\infty$ and $(X,\dist,\haus^N)$ be an $\RCD(-(N-1),N)$ metric measure space. Given $x\in X$ and $r>0$ we say that $B_r(x)$ is a $\delta$-\textit{boundary ball} if it is $\delta r$-GH close to $B_r^{\setR_+^N}(0)$.  
	
	Given a $\delta$-boundary ball $B_1(x)$ and a $\delta$-isometry $\Psi: B_1^{\setR_+^N}(0)\to B_1(x)$ we set 
\begin{align}
	\mathcal{L}_{x,1}:=\Psi(\set{x_N=0})\, .
\end{align}
When $B_{r}(x)$ is a $\delta$-boundary ball we will consider the approximate singular set $\mathcal{L}_{x,r}$ that can be introduced in the analogous way.
\end{definition}

\begin{remark}\label{remark:subballsofboundaryballs}
	 The following property is an easy consequence of definitions. Given a $\delta$-boundary ball $B_1(x)$, a $\delta$-isometry 
\[\Psi: B_1^{\setR_+^N}(0)\to B_1(x)\] 
and $y\in \mathcal{L}_{x,1}$, any ball $B_s(y)\subset B_r(x)$ is a $\delta s^{-1}$-boundary ball.
\end{remark}

\vspace{.1cm}

We now introduce the relevant notion of a neck region for this paper:

\begin{definition}[Neck region]\label{def:Neckregion}
	Fix $\eps,\delta\in (0,1/2)$, an integer $N\ge 1$ and $\tau:=10^{-10N}$.
	Let $(X,\dist,\mathcal{H}^N,p)$ be a pointed noncollapsed $\RCD(-\eps(N-1),N)$ metric measure space. We say that $\mathcal{N}\subset B_2(p)$ is an $(\eps,\delta)$-neck region if there exist a closed set $\mathcal{C}\subset B_1(p)$ and a function $r:\mathcal{C}\to [0, 1/8]$ such that $\mathcal{N}:=B_2(p)\setminus \cup_{x\in \mathcal{C}}\bar{B}_{r_x}(x)$ and, setting $\mathcal{C}_0:=\set{x\in \mathcal{C}:\ r_x=0}$ and $\mathcal{C}_+:=\mathcal{C}\setminus\mathcal{C}_0$, the following hold:
	\begin{itemize}
		\item[(i)] the family $\set{\bar{B}_{\tau^2 r_x}(x)}_{x\in \mathcal{C}}\subset B_2(p)$ is disjoint;
		\item[(ii)] for any $x\in \mathcal{C}$ and $r_x\le r\le \tau^{-3}$, $B_r(x)$ is an $\eps^2$-boundary ball, i.e. there exists an $\eps^2 r$-GH isometry
		\begin{equation}
		\Psi_{x,r}: B_r^{\setR_+^N}(0)\to B_r(x)\, ;
		\end{equation}
		\item[(iii)] setting $\mathcal{L}_{x,r}:=\Psi_{x,r}(\set{x_N=0})$ 
		\begin{equation}
		\mathcal{C}\cap B_r(x)\subset B_{2\eps r}(\mathcal{L}_{x,r})
		\quad\text{and}\quad
		\mathcal{L}_{x,r}\cap B_r(x)\subset B_{10^3 \tau r}(\mathcal{C})
		\end{equation}
		for any $x\in \mathcal{C}$ and $r_x< r < \tau^{-3}$;
		\item[(iv)] there exists a $\delta^4$-splitting map $u:B_{\tau^{-4}}(p)\to \setR^{N-1}$ such that, for any $x\in \mathcal{C}$ and $r_x< r <\tau^{-3}$ it holds that
		\begin{equation}
		u:B_r(x)\to \setR^{N-1}
		\quad\text{is a $\delta$-splitting map }
		\end{equation}
		and 
		\begin{equation}\label{eq:hessian in C0}
			r^2 \fint_{B_r(x)} |\Hess u|^2 \di \haus^{N} \le r
			 \delta^2\, ,
		\end{equation}
	\end{itemize}	
\end{definition}

\begin{remark}
In fact, it will follow from the construction that $\Lip r_x\le \tau^2$.	
\end{remark}

As in \cite{JiangNaber16} and \cite{CheegerJiangNaber18} we introduce the \emph{packing measure} as an approximation of the Hausdorff measure restricted to the top dimensional singular stratum.

\begin{definition}[Packing measure]\label{def:packingmeasure}
Given any neck region $\mathcal{N}=B_2(p)\setminus \cup_{x\in \mathcal{C}}\bar{B}_{r_x}(x)$ we shall denote by $\mu$ the associated \emph{packing measure} defined by
\begin{equation}\label{eq:packingmeasure}
\mu:=\mathcal{H}^{N-1}\res \mathcal{C}_0+\sum_{x\in \mathcal{C}_+} r_x^{N-1}\delta_x\, .
\end{equation}	
\end{definition}

\begin{remark}\label{rm:volumeonboundary balls}
It follows from Colding's volume convergence theorem \cite[Theorem 1.3]{DePhilippisGigli18} (see also \cite{Colding97,CheegerColding97}) that there exists a function $\Psi:=\Psi(\eps,N)$ depending only on $N$ and going to $0$ as $\eps\to 0$ such that, if $B_r(x)$ is an $\eps$-boundary ball, then
\begin{equation}\label{eq:volumeboundary}
\abs{\frac{\mathcal{H}^N(B_r(x))}{\omega_N r^N}-\frac{1}{2}}\le\Psi(\eps,N).
\end{equation}
In particular, if $\mathcal{N}=B_2(p)\setminus\cup_{x\in\mathcal{C}}\bar{B}_{r_x}(x)$ is an $(\eps,\delta)$-neck region, then \eqref{eq:volumeboundary} holds for any $x\in\mathcal{C}$ and for any $r_x\le r\le \tau^{-3}$.
\end{remark}

\begin{remark}
Condition (i) in the definition of neck region above guarantees that we do not overly cover. In particular, the set $\mathcal{C}$ turns to approximate the singular set with the additional property that it allows to approximate the relevant Hausdorff measure.

Condition (iii) plays a role of a Reifenberg condition on the singular set, still it is not sufficient alone to prove rectifiability, which requires the combination with (ii) and (iv).
\end{remark}

\begin{remark}
	With respect to the notions of neck region adopted in \cite{JiangNaber16} and \cite{CheegerJiangNaber18} here we chose to put the harmonic $\delta$-splitting map directly into the definition. Building $\delta$-splitting maps that control the geometry of neck regions requires a great amount of efforts and several ideas in codimension greater or equal than two, as in \cite{JiangNaber16,CheegerJiangNaber18}. Here instead we heavily rely on the fact that we are working in codimension one and the $L^2$-Hessian bounds for harmonic $\delta$-splitting maps propagate well thanks to a weighted maximal function argument, as pointed out in \cite{CheegerNaber15}. 	
\end{remark}	

\subsection{Structure of neck regions}\label{subsec:structneck}

The aim of this subsection is to prove a structure theorem for neck regions, its relevance will be clear after \autoref{subsec:existenceneck} where we are going to prove that neck regions can be built on any ball sufficiently close to a ball of the model half-space.

Let us recall that the main goal of the present paper is to prove rectifiability and measure bounds for the singular stratum of codimension one on noncollapsed $\RCD$ spaces together with stability under noncollapsing convergence. In this regard \autoref{thm:neckstructure} is a key building block since, together with \autoref{thm:existenceofneck}, it tells that our desired properties hold, up to a controlled error, on balls close to the model ball.

\begin{theorem}[Neck structure theorem]\label{thm:neckstructure}
Let $N\in \setN$, $v>0$ and $0<\eps<1$ be fixed with $\eta<\eta(N,\eps)$ and $\delta<\delta(N,v,\eps)$.  Then for any $(\eta,\delta)$-neck region $\mathcal{N}=B_2(p)\setminus \overline B_{r_x}(\mathcal{C})$ of an $\RCD(-\eta(N-1),N)$ space $(X,\dist,\haus^N)$ satisfying $\haus^N(B_1(p))\ge v$ it holds that:
\begin{itemize}
	\item[(i)] $u:\mathcal{C}\to \setR^{N-1}$ is bi-Lipschitz with its image, more precisely
	\begin{equation}
		||u(x)-u(y)|-\dist(x,y)|\le \eps\dist(x,y)
		\quad \text{for any $x,y\in \mathcal{C}$}\, ,
	\end{equation} 
	where $u:B_{\tau^{-4}}(p)\to\setR^{N-1}$ is as in \autoref{def:Neckregion} (iv);
	\item[(ii)] there exists $c=c(N)\ge 1$ such that, denoting by $\mu$ the packing measure as in \eqref{eq:packingmeasure}, we have
	\begin{equation}\label{eq:ahlforsneck}
		c^{-1}r^{N-1}\le \mu(B_r(x))\le c r^{N-1}
		\quad \text{for any $x\in \mathcal{C}$ and $r_x \le r\le 2$}\, ;
	\end{equation}
	\item[(iii)] at any $x\in \mathcal{C}_0$ the tangent cone is unique and isometric to $\setR_+^N$;

\end{itemize}
\end{theorem}

\begin{remark}
Let us comment on the different parts of the statement of \autoref{thm:neckstructure}.

The combination of points (i) and (ii) is the analogue of \cite[Theorem 3.10]{JiangNaber16} and \cite[Theorem 2.9]{CheegerJiangNaber18}. Together with the existence of neck regions and the neck decomposition theorem it can be summed up to obtain rectifiability of the top dimensional singular stratum and measure estimates.

Point (iii) has an analogue in the context of lower Ricci bounds for the codimension two stratum \cite{CheegerJiangNaber18} and in the context of two sided Ricci bounds for the codimension four stratum \cite{JiangNaber16}, where the tangent cones are also uniquely determined by the neck structure.  There is no analogue in case of general stratum under a lower Ricci bound \cite{CheegerJiangNaber18} however, where uniqueness of symmetries holds in the neck region but not of the whole tangent cone.

\end{remark}

\begin{proof}[Proof of \autoref{thm:neckstructure}]
(i). Let us fix $x,y\in \mathcal{C}$ and set $r:= (2\tau^2)^{-1}\dist(x,y)$. Assuming without loss of generality that $r_y\ge r_x$, we have $r_x<r<\tau^{-3}$ as a consequence of (i) in \autoref{def:Neckregion}. Therefore, by (iv) in \autoref{def:Neckregion}
	\begin{equation}
		u:B_r(x)\to \setR^{N-1}
		\quad\text{is a $\delta$-splitting map.}
	\end{equation}
	Let $\eps'<\eps$ to be fixed later.
	Assuming $\delta <\delta(\eps',N)$, \autoref{splitting vs isometry} yields the existence of a one dimensional manifold $(Z,\dist_Z,z)$ and a function $f:B_{2\dist(x,y)}(x)\to Z$ such that
	\begin{equation}\label{eq:completeepsiso}
		F:=(u-u(x),f):B_{2\dist(x,y)}(x)\to B_{2\dist(x,y)}^{\setR^{N-1}\times Z}((0,z))
		\quad\text{is a $2\dist(x,y) \eps'$-GH isometry}\, .
	\end{equation}
	Since $2\dist(x,y)\ge \tau^2r_x$, taking into account \autoref{def:Neckregion} (ii) and \autoref{remark:subballsofboundaryballs} we know that $B_{2\dist(x,y)}(x)$ is a $\tau^{-2}\eta$-boundary ball. Therefore the triangle inequality gives
	\begin{equation}
		\dist_{GH}(B_{2\dist(x,y)}^{\setR_+^1}(0), B_{2\dist(x,y)}^Z(z))\le 2\dist(x,y)(\tau^{-2}\eta+\eps')\, .
	\end{equation}
	Hence, choosing $\eta,\eps' \le \eps(N)$, we can apply \autoref{lemma:onedimensionalrigidity} below concluding that
	\begin{equation}
		F=(u-u(x),f):B_{\frac{3}{2}\dist(x,y)}(x)\to B^{\setR_+^N}_{\frac{3}{2}\dist(x,y)}(0)
		\quad \text{is a $2\dist(x,y)\eps'$-GH isometry}\, .
	\end{equation}
	In order to get (i) it suffices to check that
	\begin{equation}\label{eq:claim f small}
		|f(z)|\le 30(\eps' + 2\tau^{-2}\eta) \dist(x,y),
		\quad \text{for any $z\in \mathcal{C}\cap B_{\frac{3}{2}\dist(x,y)}(x)$}\, .
	\end{equation}
	Indeed \eqref{eq:claim f small}, when plugged in the defining condition of GH-isometries
	\begin{equation}\label{zz6}
		||F(x)-F(y)|-\dist(x,y)|\le 2 \dist(x,y)\eps'\, ,
	\end{equation}
	gives the sought conclusion provided $\eps', \eta\le C(\eps,N)$. 
	
	To check \eqref{eq:claim f small} we rely on (iii) in \autoref{def:Neckregion}. Set $s:=\frac{3}{2}\dist(x,y)$ to ease notation and recall that $2\tau^{-2}s\ge r_x$.
	Observe that 
	\[
	F\circ \Psi_{x,2\tau^{-2}s}:B_s^{\setR_+^N}(0)\to B_s^{\setR_+^N}(0)\quad \text{is a $2(\eps' + 2\tau^{-2}\eta) s$-GH isometry}\, , 
	\]
	therefore \autoref{lemma:onedimensionalrigidity} gives
	\begin{equation}
	F\circ \Psi_{x,2\tau^{-2}s}(\set{x_N=0})\subset B_{10 (\eps' + 2\tau^{-2}\eta)s}(\set{x_N=0})=\set{x_N< 10(\eps' + 2\tau^{-2}\eta) s}\, .
	\end{equation}
	Since by (iii) in \autoref{def:Neckregion} one has $\mathcal{C}\cap B_{2\tau^{-2}s}(x) \subset B_{10 \eta\tau^{-2} s}(\Psi_{x,2\tau^{-2}s}(\set{x_N=0}))$, we conclude that
	\begin{align*}
	F(\mathcal{C}\cap B_{s}(x))\subset& F(B_{10 \eta\tau^{-2} s}(\Psi_{x,2\tau^{-2}s}(\set{x_N=0}))) \\
	&\subset B_{20(\eps' + 2\tau^{-2}\eta) s}(F\circ\Psi_{x,2\tau^{-2}s}(\set{x_N=0}))
	\subset \set{x_N<30(\eps' + 2\tau^{-2}\eta) s}\, ,
	\end{align*}
	yielding \eqref{eq:claim f small}.

	\medskip
	(ii).  We begin by showing that
	\begin{equation}
		\mu(B_r(z))\le c r^{N-1}
		\quad \text{for any $z\in \mathcal{C}$ and $r_z \le r\le 2$}\, .
	\end{equation}
	Fix $z\in \mathcal{C}$ and $r_z<r<2$. Recall that by i) in \autoref{def:Neckregion}, 
	\begin{equation}\label{eq:disj}
	\set{B_{\frac{\tau^2}{2}r_x}(x)}_{\{x\in \mathcal{C}_+\}}\, \text{is a disjoint family}\, .
	\end{equation}
	In view of (i) we know that $u:B_r(z)\cap \mathcal{C}\to B_{3r}(u(z))$ is a $(1+\eps)$-Lipschitz map, therefore	
	\begin{equation}\label{eq:dist}
	\dist(u(x), u(\mathcal{C}_0))\ge (1-\eps)\tau^2 r_x\,  \text{for any $x\in \mathcal{C}_+\cap B_r(z)$}\, ,
	\end{equation}
	 and
	\begin{equation}\label{eq:dist2}
	\left(u(\mathcal{C}_0\cap B_r(z))\cup\bigcup_{x\neq z,\,x\in B_r(z)\cap \mathcal{C}_+}B_{\frac{\tau^2}{2}r_x}(u(x))\right)\subset B_{3r}(u(z))\, .
	\end{equation}	
	Then we can estimate
	\begin{align*}
		\mu(B_{r}(z)) = & \mathcal{H}^{N-1}(B_r(z)\cap \mathcal{C})+\sum_{x\in B_r(z)\cap \mathcal{C}_+} r_x^{N-1}\\
		\le & c'(N)  \haus^{N-1}(u(\mathcal{C}_0\cap B_r(
		z)))+ c'(N) \sum_{x\neq z,\,x\in B_r(z)\cap \mathcal{C}_+} \haus^{N-1}(B_{\frac{\tau^2}{2}r_x}(u(x)))\\   &+c'(N) \haus^{N-1}(B_{\frac{\tau^2}{2}r_z}(u(z)))\\
		\le & c'(N) \left(\haus^{N-1}(B_{3r}(u(z)))+\haus^{N-1}(B_{\frac{\tau^2}{2}r_z}(u(z)))\right)\\
		\le & c \left(r^{N-1}+r_z^{N-1}\right)\le cr^{N-1}\, .
	\end{align*}
\medskip
	
	Let us now show the opposite inequality:
	\begin{equation}\label{eq:doubling}
	\mu(B_r(z))\ge  c^{-1} r^{N-1}
	\quad \text{for any $z\in \mathcal{C}$ and $0 \le r\le 2$}\, .
	\end{equation}
Observe that, by the very definition of packing measure \eqref{eq:packingmeasure}, it is sufficient to verify \eqref{eq:doubling} for radii $r$ such that $r_z<r<2$.

Let us fix $z\in \mathcal{C}$ and $r_z<r<2$ as above. It suffices to prove that 
	\begin{equation}\label{eq:claimcovering}
		B_{r/8}(u(z))\subset u(\mathcal{C}_0 \cap \bar B_r(z)) \cup \bigcup_{x\in \mathcal{C}_+\cap \bar B_r(z)} \bar  B_{r_x}(u(x))\, .
	\end{equation}
	
%%%%%%	

	Indeed \eqref{eq:claimcovering} gives
	\begin{align*}
	\frac{\omega_{N-1}}{8^{N-1}} r^{N-1}  & \le  \haus^{N-1}(u(\mathcal{C}_0 \cap \bar B_r(z))) + \sum_{x\in \mathcal{C}_+\cap \bar B_r(z)} \haus^{N-1}( B_{r_x}(u(x)))
	\\& \le (1+\eps)^N \haus^{N-1}(\mathcal{C}_0\cap B_r(z)) + \omega_{N-1} \sum_{x\in \mathcal{C}_+\cap \bar B_r(z)} r_x^{N-1}
	\\&\le C(N) \mu(\bar B_r(z))\, .
	\end{align*}
	Let us check \eqref{eq:claimcovering} arguing by contradiction.
	Set for simplicity $u(z)=0$.
	If the conclusion is false we can find $w\in B_{r/8}(0)$ such that $w\notin \bar  B_{r_x}(u(x))$ for any $x\in\mathcal{C}\cap \bar B_r(z)$. Then we can set
	\begin{equation}
		s:=\inf \set{s_x:\  x\in \mathcal{C}\cap \bar B_r(z)\, \text{and}\ w\in B_{s_x}(u(x))}\, .
	\end{equation}
	Observe that, by the very definition of $s$, it holds that $r/8>s=s_x>r_x$ for some $x\in \mathcal{C}\cap \bar B_r(z)$. Therefore the ball $B_s(x)$ is an $\eta$-boundary ball and 
	$u:B_{s}(x)\to \setR^{N-1}$ is $\delta$-splitting. Hence, arguing as in the first part of the proof, we can complete $u$ to an $s\eps'$-GH isometry
	\begin{equation}
		F:=(u,f):B_s(x)\to B_s^{\setR_+^{N}}((u(x),0))\, ,	
	\end{equation}
	provided $\delta< \delta(N,\eps')$, for some $\eps'<1/8$.
	Since $(w,0)\in \bar B_s^{\setR_+^N}((u(x),0))$ we can find $y\in \mathcal{L}_{x,s}$ such that $|u(y)-w|\le 2s \eps'$.\\ 
	Moreover, thanks to the second inclusion in (iii) of \autoref{def:Neckregion} there exists $y'\in \mathcal{C}\cap B_{2s}(x)$ such that $\dist(y,y')\le 10^3\tau s$.
	This implies that
	\begin{equation}\label{zz8}
		|u(y')-w|\le |u(y)-u(y')|+|u(y)-w|\le \Lip u\, \dist(y,y')+2s \eps' 
		\le (10^3 \tau \, \Lip u +2\eps') s<s\, ,
	\end{equation}
	since $u$ is $(1+C(N)\sqrt{\delta} )$-Lipschitz and $\tau < 10^{-4}$, cf. with \autoref{rm:sharpgradientbound}.
	
	We claim that $B_{2s}(x)\subset B_r(z)$. In order to prove this claim let us first point out that 
   \begin{equation}
	\abs{u(x)-u(z)}\le \abs{u(x)-w}+\abs{u(z)-w}<r/8+r/8=r/4\, .
   \end{equation}	
   Hence, since by the result of the previous step,
    \begin{equation}
	\abs{\abs{u(x)-u(z)}-\dist(x,z)}\le \eps\dist(x,z)\, , 
	\end{equation}
   we can infer that, if $\eps<\eps(N)$, then $\dist(x,z)<2r/4=r/2$. In particular, since we already pointed out that $s<r/8$, we obtain that $B_s(x)\subset B_r(z)$, as we claimed.
	
	This gives \eqref{eq:claimcovering}, since \eqref{zz8} and the inclusion $B_{2s}(x)\subset B_r(z)$ contradicts the minimality of $s$.
\medskip

%%%%%%	

    (iii). Let $(Y,\varrho,\haus^N,y)$ be a tangent cone at $x\in \mathcal{C}_0$, originating from a sequence $r_n\downarrow 0$. From (iv) in \autoref{def:Neckregion} we know that there exists $u_{\infty}: Y \to \setR^{N-1}$ satisfying
    \begin{itemize}
    	\item[(a)] $\Hess u_{\infty}=0$,
    	\item[(b)] $\int_{B_1(y)}|\nabla (u_{\infty})_{a} \cdot \nabla (u_{\infty})_{b}- \delta_{a,b}| \di \haus^N \le C(N)\delta$, for $a,b=1,\ldots,N-1$,
    \end{itemize}
    as a limit of the sequence 
    \[
    u_{r_n}:=u: (X, \frac{\dist}{r_n},\frac{\haus^N}{{r_n}^N}, x)\to \setR^{N-1}\quad \text{as $n\to\infty$}\, .
    \]
    This can be easily checked with a by now standard argument, relying on the convergence and stability results of \cite{AmbrosioHonda,AmbrosioHonda18}.\\    
    By the functional splitting theorem (cf. \cite[Lemma 1.20]{AntonelliBrueSemola19}), $Y$ splits off a factor $\setR^{N-1}$. Moreover, it is a metric cone and it is not isometric to $\setR^N$, thanks to \autoref{def:Neckregion} (ii). Therefore it is isometric to $\setR^N_{+}$.
\end{proof}

\subsection{Existence of neck regions}\label{subsec:existenceneck}

The aim of this subsection is to prove that on any ball of a noncollapsed $\RCD(-(N-1),N)$ space, which is sufficiently GH-close to the model ball on the half space 
\begin{equation}
B_1^{\setR^N_+}(0)\subset \setR^N_+
\end{equation} 
it is possible to build a neck region.

It is worth noting that, for the sake of proving stability results, it is key to construct neck regions quite carefully. 
If we were to build neck regions that are much smaller than they should be, then the construction would not force the presence of boundary points in $\mathcal{S}^{N-1}\setminus\mathcal{S}^{N-2}$ on regions which look like half spaces.\\
For this reason below we are going to prove that neck regions verifying an additional \emph{maximality} condition exist, once we assume closeness to the model boundary ball.

\begin{theorem}[Existence of maximal neck regions]\label{thm:existenceofneck}
Let $\tau:=10^{-10N}$ with $0<\eps\le \eps(N)$, $\delta<\delta(\eps,N)$ and $\eta<\eta(\eps,\delta,N)$.  Then if $(X,\dist,\mathcal{H}^N)$ is an $\RCD(-\eta(N-1),N)$ m.m.s., $p\in X$ and $B_{4\tau^{-4}}(p)$ is an $\eta$-boundary ball, then there exists an $(\eps,\delta)$-neck region $\mathcal{N}=B_2(p)\setminus \cup_{x\in \mathcal{C}}\bar{B}_{r_x}(x)$ on $B_2(p)$ which additionally verifies the following maximality condition:
\begin{equation}\label{eq:maxneck}
\mu(\mathcal{C}_+)=\sum_{x\in \mathcal{C}_+}r_x^{N-1}\le \eps\, .
\end{equation}
\end{theorem}

\begin{remark}
The combination of \eqref{eq:maxneck} with the lower Ahlfors bound in \autoref{thm:neckstructure} (ii) and \autoref{thm:neckstructure} (iii) implies that on $\eta$-boundary balls there exists a bunch of boundary points in $\mathcal{S}^{N-1}\setminus\mathcal{S}^{N-2}$.\\
This is the starting point of our rigidity and stability and it is unique to the codimension one setting.
Indeed, as it is pointed out in \cite{CheegerJiangNaber18}, for neck regions on smooth Riemannian manifolds, $\mathcal{C}_0$ is always empty (there is no singular set). Instead, the combination of \eqref{eq:maxneck} and \autoref{thm:neckstructure} above provides an analytic proof of a quantitative (and more general) version of the fact, proved in \cite{CheegerColding97}, that smooth Riemannian manifolds with lower Ricci curvature bounds cannot converge without volume collapse to a half-space.
\end{remark}

\subsubsection{Auxiliary results}

Before proving \autoref{thm:existenceofneck}, we prove three key lemmas. The first one deals with en elementary property of one dimensional Riemannian manifolds with boundary, the second one deals with the propagation of the $\delta$-splitting property, the last one is the fundmental iteration step in the proof of \autoref{thm:existenceofneck}.

\begin{lemma}[One dimensional rigidity]
	\label{lemma:onedimensionalrigidity}
	There exists $\eps_0>0$ such that, if a pointed one dimensional Riemannian manifold (possibly with boundary) $(Z,\dist_Z,z)$ satisfies
	\begin{equation}
		\dist_{GH}(B_1^{\setR_+}(0),B_1^Z(z))\le \eps\le \eps_0\, ,
	\end{equation}
	then there exist $0\le a\le \eps \le 1-\eps\le  b \le \infty$ such that
	$B_1^Z(z)$ is isometric to the ball of radius one centered at $a\in [0,b)$.
\end{lemma}

\begin{proof}
Observe that $\partial Z \neq \emptyset$, since any ball of radius one in $\mathbb{S}^1(r)$ for some $r>0$, or in $\setR$ is $\eps_0$ far from $B_1^{\setR_+}(0)$, provided $\eps_0$ is small enough. This implies that $Z$ is isometric to $[0,b)$ for some $b \le \infty$. It is now immediate to check that a ball of radius one in $[0,b)$ is $\eps$-close to $B_1^{\setR_+}(0)$ if and only if it is centered at some point $0\le a \le \eps$ and $b\ge 1-\eps$.
\end{proof}

Next we give a general auxiliary result about the propagation of the $\delta$-splitting property. Basically, it amounts to saying that given a $\delta$-splitting map at a certain location and scale, the $\delta$-splitting property, up to slightly worsening $\delta$, propagates if we can control the Hessian in a slightly better than scale invariant sense. We refer to \cite[Lemma 5.91]{JiangNaber16} and to \cite{BruePasqualettoSemola19} for previous appearances of this argument.

\begin{lemma}\label{remark:smallHessianimpliessplitting}
	Let $1\le N<\infty$ be fixed. There exists $C=C(N)>0$ such that for any $\RCD(-(N-1),N)$ metric measure space $(X,\dist,\meas)$ with $p\in X$ the following holds.  If $u:B_2(p)\to \setR^{k}$ is a $\delta$-splitting map with $x\in B_1(p)$ such that
	\begin{equation}\label{eq:hessbd}
	s\fint_{B_s(x)}|\Hess u|^2 \di \meas \le \delta^{1/2}\, ,
	\quad\text{for any $\, 0<r<s<1$}\, ,
	\end{equation}
	then $u: B_s(x)\to \setR^{k}$ is a $C(N)\delta^{1/4}$-splitting map for any $r<s<1$.
\end{lemma}
\begin{proof}
	It is enough to check that
	\begin{equation}\label{eq:deltasplitprop}
	\fint_{B_s(x)} |\nabla u_a\cdot \nabla u_b-\delta_{a,b}|\di \meas <C(N) \delta^{1/4}	\qquad \text{for any }\ 0<r<s<1\, ,
	\end{equation}
	for $a,b=1,\ldots,k$.
	
	Let us set $f_{a,b}:= |\nabla u_a\cdot \nabla u_b-\delta_{a,b}|$ and note that $|\nabla f_{a,b}|\le 2 C(N)(|\Hess u_a|+|\Hess u_b|)$, here we have used (i) in \autoref{def:deltasplitting}.
	Using this bound, the local Poincaré inequality (cf \cite{VonRenesse08}) and \eqref{eq:hessbd} we deduce
	\begin{align*}
	\abs{\fint_{B_{2s}(x)}f_{a,b}\di \meas \right. & \left.-\fint_{B_{s}(x)}f_{a,b}\di \meas}\\
	\le & C(N) s\fint_{B_{2s}(x)}|\nabla f_{a,b}| \di \meas \\
	\le & 2C(N)\left( s^2\fint_{B_{2s}(x)} |\Hess u_a|^2\di \meas+s^2\fint_{B_{2s}(x)} |\Hess u_b|^2\di \meas\right)^{1/2}\\
	\le & 4C(N)\delta^{1/4}s^{1-1/2}
	\end{align*}
	for any $0<r<s<1$. Applying a telescopic argument it is simple to see that
	\begin{equation}
	\abs{\fint_{B_1(x)}f_{a,b}\di\meas-\fint_{B_{s}(x)}f_{a,b}\di \meas}\le C(N) \delta^{1/4}\qquad \text{for any}\ 0<r<s<1.
	\end{equation}
	Therefore, using that $u:B_2(p)\to \setR^k$, is a splitting map we deduce
	\begin{align*}
	\fint_{B_{s}(x)}f_{a,b}\di \meas \le & \abs{\fint_{B_1(x)}f_{a,b}\di\meas-\fint_{B_{s}(x)}f_{a,b}\di \meas}+\fint_{B_{1}(x)}f_{a,b}\di \meas\\
	\le & C(N) \delta^{1/4}+ C(N)\fint_{B_2(p)} f_{a,b} \di \meas\\
	\le & C(N)\delta^{1/4},
	\end{align*}
	therefore yielding \eqref{eq:deltasplitprop}.
\end{proof}

The next lemma is the main iterative step in the construction of maximal neck regions over boundary balls.

\begin{lemma}[Iteration step]
	\label{lemma:iteration step}
	Let $\tau=10^{-10N}$ and $\gamma\in (0,1/4)$ be fixed with $\eps_0$ as in \autoref{lemma:onedimensionalrigidity}.
	For any $\eps< 2\gamma\wedge\eps_0/2$ and $\delta\le \delta(\eps,\gamma, N)$ the following property holds. 
	Given an $\RCD(-\eps(N-1),N)$ m.m.s. $(X,\dist,\haus^N)$, an $\eps^2$-boundary ball $B_2(p)\subset X$ and $k\in \setN$, if there exists a $\delta^4$-splitting map $u:B_{2 \gamma^{-k}}(p)\to \setR^{N-1}$ such that
	\begin{equation}\label{eq:bdhessian}
		\tau^4 \gamma^{-m} \fint_{B_{\tau^4 \gamma^{-m}}(p)} |\Hess u|^2\di\haus^N
			\le \delta^2
			\quad\text{for $m=0,\dots,k$\, ,}
	\end{equation}
	then there exists a covering
	\begin{equation}\label{zz2}
		B_1(p)\cap \mathcal{L}_{p,2}\subset \bigcup_{\alpha} B_{ {2\tau^3}\gamma}(x_{\alpha})\bigcup_{\beta} B_{ {2\tau^3}\gamma}(x_{\beta})\, ,
	\end{equation}
	where
	\begin{itemize}
		\item[(i)] the family $\set{B_{\tau^4 \gamma}(x_{\alpha})}\cup \set{B_{\tau^4 \gamma}(x_{\beta})}$ is disjoint;
		\item[(ii)] for any $\alpha$, $B_{2\gamma}(x_{\alpha})$ is an $\eps^2$-boundary ball such that
		\[
		\tau^4\gamma \fint_{B_{\tau^4\gamma}(x_{\alpha})}|\Hess u|^2\di \haus^N \le \delta^2
		\]		
		with
		\begin{equation}
			u:B_s(x_{\alpha})\to \setR^{N-1}
			\quad\text{a $C_{\gamma, N}\delta$-splitting map for any $s\in [\tau^2 \gamma, \gamma^{-k}]$}\, ;
		\end{equation}
		\item[(iii)] for any $\beta$, 
		\begin{equation}\label{eq:maximality}
		\tau^4\gamma \fint_{B_{\tau^4\gamma}(x_{\beta})}|\Hess u|^2\di \haus^N >\delta^2\, ;
		\end{equation}
		\end{itemize}
		\item[(iv)]  $x_{\alpha}\in B_{2\gamma\eps}(\mathcal{L}_{p,2})$ for any $\alpha$, and $x_{\beta}\in\mathcal{L}_{p,2}$ for any $\beta$.
\end{lemma}

\begin{proof}
	We choose $\delta=\delta(\eps,\gamma,N)$ given by \autoref{splitting vs isometry} such that any ball endowed with a $C_{\gamma,N}\delta$-splitting map is $\eps^2/2$-close to a ball in a space splitting $\setR^{N-1}$, where $C_{\gamma,N}$ is the constant in \eqref{zz3} below.

	Let $\set{B_{\tau^3 \gamma}(x_{\xi})}$ be any covering of $B_1(p)\cap \mathcal{L}_{p,2}$ with $x_{\xi}\in \mathcal{L}_{p,2}$ and such that $\set{B_{\tau^4\gamma}(x_{\xi})}$ is a disjoint maximal collection. 
	
	Given $\xi$ such that
	\begin{equation}\label{zz1}
		\tau^4 \gamma \fint_{B_{\tau^4 \gamma}(x_{\xi})}|\Hess u|^2\di \haus^N \le \delta^2\, ,
	\end{equation}
	one  has  
	\begin{equation}
		s\fint_{B_s(x_{\xi})} |\Hess u|^2 \di \haus^N
		\le C_{\gamma, N}\delta^2
		\quad\text{for any $\tau \gamma \le s\le \tau \gamma^{-k}$}\, ,
	\end{equation}
	as a consequence of \eqref{eq:bdhessian}.
	Therefore, by \autoref{remark:smallHessianimpliessplitting},
	\begin{equation}\label{zz3}
			u: B_s(x_{\xi})\to \setR^{N-1}
			\quad\text{is $C_{\gamma,N}\delta$-splitting for any $s\in [\tau\gamma, \tau \gamma^{-k}]$\, .}
	\end{equation}
	In order to conclude the proof we just need to prove the existence of $x_{\alpha}\in B_{2\gamma\eps}(x_{\xi})$ such that $B_{2\gamma}(x_{\alpha})$ is an $\eps^2$-boundary ball.
	
	To do so notice that the following properties hold:
	\begin{itemize}
		\item[(a)] $B_{4\gamma }(x_{\xi})$ is a $2^{-1}\gamma^{-1}\eps^2 $-boundary ball;
		\item[(b)] there exists a one dimensional  manifold  (possibly with boundary) $(Z,\dist_Z)$ such that
		\begin{equation}
			\dist_{GH}(B_{4\gamma}(x_{\xi}),B^{\setR^{N-1}\times Z}_{4\gamma}((0,z)))\le 2\gamma \eps^2\, .
		\end{equation}
	\end{itemize}
	The property (a) follows from \autoref{remark:subballsofboundaryballs} while (b) comes from \eqref{zz3} and \autoref{splitting vs isometry}.

	Let us finally prove that (a) and (b) imply the existence of $x_{\alpha}\in B_{2\gamma\eps}(x_{\xi})$ such that $B_{2\gamma}(x_{\alpha})$ is an $\eps^2$-boundary ball.
	
	By exploiting (a), (b), the triangular inequality and our choice of $\eps$, we deduce
	\begin{equation}
		\dist_{GH}(B^{\setR_+^1}_{4\gamma  }(0),B^Z_{4\gamma}(z))\le 2\eps^2 + 2\gamma \eps^2
		\le 8\gamma\, \eps\, \le  \eps_0\,  4\gamma\, .
	\end{equation}
	Therefore by \autoref{lemma:onedimensionalrigidity} (in scale invariant form) there exist
	$0\le a\le 8\gamma\eps\le 4\gamma-8\gamma\eps\le b\le \infty$ such that
	\begin{equation}
	\dist_{GH}(B_{4\gamma}(x_{\xi}),B^{\setR^{N-1}\times [0,b]}_{4\gamma}((0,a)))\le 2\gamma \eps^2\, .
	\end{equation}
	Denoting by $\Psi: B^{\setR^{N-1}\times [0,b]}_{4\gamma}((0,a))\to B_{4\gamma}(x_{\xi})$ any $2\gamma \eps^2$-GH isometry, we set $x_{\alpha}:=\Psi(0)$. It is easily seen that $x_{\alpha}\in B_{2\gamma\eps}(x_{\xi})$ and $B_{2\gamma}(x_{\alpha})$ is an $\eps^2$-boundary ball (compare with \autoref{remark:subballsofboundaryballs}).
	
Changing $x_{\xi}$ into $x_{\alpha}$ in the above considered case and relabelling $x_{\xi}$ as $x_{\beta}$ in the other one, it is easily verified that the family $\set{B_{\gamma}(x_{\alpha})}\cup\set{B_{\gamma}(x_{\beta})}$ has the sought properties. 
\end{proof}

\subsubsection{Strategy of proof of \autoref{thm:existenceofneck}}
The overall strategy is similar to those \cite[Proposition 10.5]{CheegerJiangNaber18}, \cite[Proposition 7.13, Proposition 7.26]{JiangNaber16} and \cite[Theorem 5.4]{NaberValtorta19}. The key difference with respect to \cite{CheegerJiangNaber18} and \cite{JiangNaber16} is that we wish to build neck regions whose geometry is controlled by the same $\delta$-splitting map. Moreover, as we already pointed out, we need to prove that \emph{bad-balls} $B_{r_x}(x)$ with $x\in\mathcal{C}_+$ have small $(N-1)$-dimensional content.
\medskip

The first step in the construction is based on the iterative application of \autoref{lemma:iteration step}. It turns that the outcome of this construction is a decomposition which shares most of the properties of neck regions with a subtle difference: there might be nearby balls with uncontrollably different sizes. This implies in turn that the second inclusion in condition (iii) of \autoref{def:Neckregion} needs to be relaxed to a slightly weaker inclusion, where we do not look below the scale $r_y$ near to a point $y\in\mathcal{C}_+$, cf. with \eqref{eq:reifweak} below.

\begin{definition}[Weak neck region]\label{def:weakneckregion}
	Fix $\eps,\delta\in (0,1/2)$, an integer $N\ge 1$ and $\tau:=10^{-10N}$.
	Let $(X,\dist,\mathcal{H}^N,p)$ be a pointed $\RCD(-\eps(N-1),N)$ metric measure space. We say that $\mathcal{N}\subset B_2(p)$ is a weak$(\eps,\delta)$-neck region if there exist a closed set $\mathcal{C}\subset B_1(p)$ and a function $r:\mathcal{C}\to [0, 1/8]$ such that $\mathcal{N}:=B_2(p)\setminus \cup_{x\in \mathcal{C}}\bar{B}_{r_x}(x)$ and, setting $\mathcal{C}_0:=\set{x\in \mathcal{C}:\ r_x=0}$ and $\mathcal{C}_+:=\mathcal{C}\setminus\mathcal{C}_0$, the following hold:
	\begin{itemize}
		\item[(i)] the family $\set{\bar{B}_{\tau^4 r_x}(x)}_{x\in \mathcal{C}}\subset B_2(p)$ is disjoint;
		\item[(ii)] for any $x\in \mathcal{C}$ and $r_x\le r\le \tau^{-4}$, $B_r(x)$ is an $\eps^2$-boundary ball, i.e. there exists an $\eps^2 r$-GH isometry
		\begin{equation}
		\Psi_{x,r}: B_r^{\setR_+^N}(0)\to B_r(x)\, ;
		\end{equation}
		\item[(iii)] setting $\mathcal{L}_{x,r}:=\Psi_{x,r}(\set{x_N=0})$, it holds that 
		\begin{equation}\label{eq:reifweak}
		\mathcal{C}\cap B_r(x)\subset B_{2\eps r}(\mathcal{L}_{x,r})
		\quad\text{and}\quad
		\mathcal{L}_{x,r}\cap B_r(x)\subset B_{100\tau^3 \max\{r,r_y\}}(\mathcal{C})\, ,
		\end{equation}
		for any $x\in \mathcal{C}$ and $r_x< r < \tau^{-4}$, where we denoted by
		\begin{equation}
		B_{100\tau^3 \max\{r,r_y\}}(\mathcal{C}):=\bigcup_{y\in\mathcal{C}}B_{100\tau^3 \max\{r,r_y\}}(y) \,;
		\end{equation}
		\item[(iv)] there exists a $\delta^4$-splitting map $u:B_{2\tau^{-4}}(p)\to \setR^{N-1}$ such that, for any $x\in \mathcal{C}$ and $r_x< r <\tau^{-4}$ it holds that
		\begin{equation}
		u:B_r(x)\to \setR^{N-1}
		\quad\text{is a $\delta$-splitting map }
		\end{equation}
		and 
		\begin{equation}\label{eq:weakneckhess}
			r^2 \fint_{B_r(x)} |\Hess u|^2 \di \haus^{N} \le C(N)r \delta^2\, .
		\end{equation}
	\end{itemize}	
\end{definition}
The packing measure associated to a weak neck region is defined as in the case of neck regions, cf. with \autoref{def:packingmeasure}.
\medskip

The outcome of this first step will be a weak neck region for which we can additionally prove a small content bound for bad-balls, cf. with \eqref{eq:maxweakneck} below. This is due to the fact that in \autoref{lemma:iteration step} we only stop the decomposition when the Hessian bound fails. Since we start with an Hessian bound on the ambient ball, this allows to get a content bound for bad-balls via a standard weighted maximal argument.

\begin{proposition}[Existence of weak neck regions with content estimates]\label{thm:existenceofweakneck}
Let $\tau:=10^{-10N}$.
For any $0<\eps\le \eps(N)$, for any $\delta<\delta(\eps,N)$ and $\eta<\eta(\eps,\delta,N)$ the following holds. If $(X,\dist,\mathcal{H}^N)$ is an $\RCD(-\eta(N-1),N)$ m.m.s., $p\in X$ and $B_{4\tau^{-4}}(p)$ is an $\eta$-boundary ball, then there exists a weak $(\eps,\delta)$-neck region $\mathcal{N}=B_2(p)\setminus \cup_{x\in \mathcal{C}}\overline{B}_{r_x}(x)$ on $B_2(p)$ which additionally verifies
\begin{equation}\label{eq:maxweakneck}
\mu(\mathcal{C}_+)=\sum_{x\in \mathcal{C}_+}r_x^{N-1}\le \eps\, .
\end{equation}
\end{proposition}

Once we have built a weak neck region we need to refine the construction to get a neck region out of it, also keeping the small content bound for bad balls. This will be achieved refining the approximate singular set and regularizing the radius function $r_x$.\\ 
Via this procedure we might be enlarging the set $\mathcal{C}_+$ of centers of bad balls and the new bad balls might not verify anymore the maximality condition \eqref{eq:maximality}. The key observation will be that the new center points are all near to old center points at the right scale. Then the structure \autoref{thm:neckstructure} allows to turn the small content bound for the weak neck region into a small content bound for the neck region.

\subsubsection{Proof of \autoref{thm:existenceofweakneck}}\label{subsubsec:exweakneck}

\textbf{Base step.} Let us fix a scale parameter $\gamma=1/8$, $\eps\le \eps(N)$ and $\delta<\delta(\eps,\gamma,N)$ as in \autoref{lemma:iteration step}. We consider $\eta < \eta(N,\delta^4)$ such that, by \autoref{splitting vs isometry} (i), there exists a $\delta^4$-splitting map $u:B_{2\tau^{-4}}(p)\to\setR^{N-1}$.

Let us apply \autoref{lemma:iteration step} with $k=0$ to obtain a covering
	\begin{equation}
		B_1(p)\cap \mathcal{L}_{p,1}\subset \bigcup_{\alpha} B_{2\tau^3\gamma}(x_{\alpha}^1)\bigcup_{\beta} B_{2\tau^3\gamma}(x_{\beta}^1)\, ,
	\end{equation}
	where
	\begin{itemize}
		\item[(i)] the family $\set{B_{\tau^4\gamma}(x_{\alpha}^1)}\cup \set{B_{\tau^4 \gamma}(x^1_{\beta})}$ is disjoint;
		\item[(ii)] $\alpha$-balls $B_{2\gamma}(x^1_{\alpha})$ and $\beta$-balls $B_{2\gamma}(x^1_{\beta})$ verify the following:
		\begin{itemize}
		\item[($\alpha$-balls)] for any $\alpha$, $B_{2\gamma}(x^1_{\alpha})$ is an $\eps^2$-boundary ball and
		\begin{equation}
			u:B_s(x^1_{\alpha})\to \setR^{N-1}
			\quad\text{is a $C_{\gamma, N}\delta$-splitting map for any $s\in [\tau^3 \gamma, 1]$}\, ;
		\end{equation}
		\item[($\beta$-balls)] for any $\beta$, $\tau^4\gamma \fint_{B_{\tau^4\gamma}(x^1_{\beta})}|\Hess u|^2\di \haus^N >\delta^2$; 
		\end{itemize}
		\item[(iii)] $x^1_{\alpha}\in B_{2\gamma\eps}(\mathcal{L}_{p,2})$ for any $\alpha$, and $x^1_{\beta}\in\mathcal{L}_{p,2}$ for any $\beta$.
	\end{itemize}

\textbf{Iteration steps.} 
Observe that
\begin{equation}
\mathcal{N}^1:=B_2(p)\setminus\left(\bigcup_{\alpha}\bar{B}_{\tau^2\gamma}(x^1_{\alpha})\cup \bigcup_{\beta}\bar{B}_{\tau^2\gamma}(x^1_{\beta})\right)
\end{equation}
is easily seen to be a neck region. However, it could be that this neck region is not nearly maximal and can be extended. Therefore we
wish to iteratively refine the construction by decomposing the $\alpha$-balls in the decomposition via \autoref{lemma:iteration step}.
\medskip

To this aim, let us apply \autoref{lemma:iteration step} with $k=1$ and same choice of parameters as in the base step of the iteration to any approximate singular set $\mathcal{L}_{x_{\alpha}^1,\gamma}$ (after scaling the scale $\gamma$ to scale $1$).\\ 
We obtain a covering
\begin{align}\label{e:tau3_cL}
		\left(\bigcup_{\alpha}\mathcal{L}_{{x}_{\alpha}^1,\gamma}\right)\setminus \bigcup_{\beta} B_{\tau^3\gamma}(x^1_{\beta})\subset \bigcup_{\alpha}B_{2\tau^3\gamma^2}(x^2_{\alpha})\cup\bigcup_{\beta}B_{2\tau^3\gamma^2}(x_{\beta}^2)\, ,
	\end{align}
such that $B_{\tau^4\gamma^2}(x^2_{\alpha})$ and $B_{\tau^4\gamma^2}(x^2_{\beta})$ are disjoint, 
\begin{align}\label{eq:centrebetaclose}
 	x_{\beta}^2\in \Big(\bigcup_{\alpha}\mathcal{L}_{x_{\alpha}^1,\gamma}\Big)\setminus\bigcup_{\beta}B_{\tau^3\gamma}(x_{\beta}^1)
\end{align}
and 
\begin{align}\label{eq:centreclose}
 	x_{\alpha}^2\in \Big(\bigcup_{\alpha}B_{2\gamma^2\eps}\mathcal{L}_{x_{\alpha}^1,\gamma}\Big)\setminus\bigcup_{\beta} B_{\tau^3\gamma}(x_{\beta}^1)\, .
\end{align}
Moreover, the balls in the covering are labeled as $\alpha$-balls or $\beta$-balls according to the convention of \autoref{lemma:iteration step}. In particular:
\begin{itemize}
\item for any center of $\alpha$-ball $x^2_{\alpha}$, it holds that $B_{r}(x^2_{\alpha})$ is a $\gamma^{-1}\eps^2$-boundary ball for any $\gamma^2<r< \tau^{-4}$ and 
\begin{equation}
\tau^4\gamma^k\fint_{B_{\tau^4\gamma^k}(x^2_{\alpha})}\abs{\Hess u}^2\di\haus^N\le \delta^2\,\,\,\text{for $k=0,1,2$}\, .
\end{equation}
\item for any $k=1,2$ and any centre of $\beta$-ball $x^k_{\beta}$, it holds that
\begin{equation}
\tau^4\gamma^k\fint_{B_{\tau^4\gamma^k}(x^k_{\beta})}\abs{\Hess u}^2\di\haus^N>\delta^2.
\end{equation}
\end{itemize}
Moreover, being $\gamma\le 1/4$, the balls $B_{\tau^4\gamma^2}(x_{\alpha}^2)$ and $B_{\tau^4\gamma^2}(x_{\beta}^2)$ are also mutually disjoint with all the balls $B_{\tau^4\gamma}(x_{\beta}^1)$.

At this stage of the decomposition we have a weak neck region
 \begin{align}
  	\mathcal{N}^2:=  B_2(p)\setminus  \left(\bigcup_{\alpha} \bar{B}_{\gamma^2}(x_{\alpha}^2)\cup \bigcup_{1\le j\le 2} \bigcup_{\beta} \bar{B}_{\gamma^j}(x_{\beta}^j)\right)\, .
 \end{align}
Observe that after this second step we pass from a neck region to a weak neck region since there might already be balls of different sizes nearby. As we already pointed out this motivates the necessity for a refinement of the construction later on.
\medskip
  
After repeating this decomposition $i$ times, decomposing at each step the $\alpha$-balls via \autoref{lemma:iteration step} (applied with $k=i$ after scaling each $\alpha$-ball to radius $2$),
we get an approximate weak neck region
\begin{equation}\label{eq:weaknecki}
		\mathcal{N}^i:=  B_2(p)\setminus  \left(\bigcup_{\alpha} \bar{B}_{\gamma^i}(x_{\alpha}^i)\cup \bigcup_{1\le j\le i}\bigcup_{\beta} \bar{B}_{\gamma^j}(x_{\beta}^j)\right)\,.
\end{equation}
The balls in the covering in particular satisfy the following:
\begin{itemize}
\item for any centre of $\alpha$-ball $x^i_{\alpha}$, it holds that $B_{r}(x^i_{\alpha})$ is a $\gamma^{-1}\eps^2$-boundary ball for any $\gamma^i<r<\tau^{-4}$ and 
\begin{equation}
\tau^4\gamma^k\fint_{B_{\tau^4\gamma^k}(x^i_{\alpha})}\abs{\Hess u}^2\di\haus^N\le \delta^2\,\,\,\text{for $k=0,1,\dots,i$}\, .
\end{equation}

\item for any $k=1,2,\dots, i$, and any centre of $\beta$-ball $x^k_{\beta}$, it holds that
\begin{equation}
\tau^4\gamma^k\fint_{B_{\tau^4\gamma^k}(x^k_{\beta})}\abs{\Hess u}^2\di\haus^N>\delta^2\, .
\end{equation}
\end{itemize}
\medskip

\textbf{Limiting argument.}	
Set 
\begin{equation}
\mathcal{C}_{\alpha}^i:=\{x_{\alpha}^i\}\, .
\end{equation}
By construction (cf. with \eqref{eq:centrebetaclose} and \eqref{eq:centreclose}) we have 
\begin{equation}
\mathcal{C}_{\alpha}^{i+1}\subset B_{\gamma^i}(\mathcal{C}_{\alpha}^i)\, .
\end{equation}
Therefore, we can define the Hausdorff limit 
	\begin{align}
		\mathcal{C}_0:= \lim_{i\to\infty}\mathcal{C}_{\alpha}^{i}\, .
	\end{align}
By letting $i\to\infty$ and passing to the limit the weak neck regions $\mathcal{N}^i$ in \eqref{eq:weaknecki}, letting 
\begin{equation}
\mathcal{C}_+:=\set{x_{\beta}}=\bigcup_{i}\set{x_{\beta}^i}\, ,\,\,\,\text{and }\, r_{x_{\beta}}:=\gamma^i\,\,\,\text{if $x_{\beta}=x_{\beta}^i$}\, ,
\end{equation}
we get
	\begin{align}
		\mathcal{N}:= B_2(p)\setminus \left(\mathcal{C}_0\cup
 \bigcup_{\beta}\bar B_{r_{\beta}}(x_{\beta})\right)=B_2(p)\setminus \left(\mathcal{C}_0\cup
 \bigcup_{x\in \mathcal{C}_+}\bar B_{r_x}(x)\right)\, .
	\end{align}
By  construction, the balls $B_{\tau^4r_x}(x)$ are disjoint 
for $x\in\mathcal{C}_+$,
\begin{equation}\label{eq:hessbdbelow}
\tau^4r_x\fint_{B_{\tau^4r_x}(x)}\abs{\Hess u}^2\di\haus^N>\delta^2\, .
\end{equation}

Passing to the limit the bounds in the intermediate steps of the construction it is possible to infer that $\mathcal{N}$ is a $(\gamma^{-1}\eps,\delta)$-weak neck region.
\medskip

%%%%%%%

\textbf{Content bound.} Let us now verify the small content bound \eqref{eq:maxweakneck}. Taking into account the disjointness of the balls $B_{\tau^4r_x}(x)$ for $x\in\mathcal{C}_+$ and \eqref{eq:hessbdbelow} we can estimate
   \begin{align*}
   	\sum_{x\in \mathcal{C}_+} \frac{\mathcal{H}^N(B_{\tau^4 r_x}(x))}{\tau^4 r_x} 
   	\le  & \delta^{-2}\sum_{x\in\mathcal{C}_+} \int_{ B_{\tau^4 r_x}(x)}|\Hess u|^2 \di \mathcal{H}^{N}\\
   	\le  & \delta^{-2}\int_{B_{2}(p)}|\Hess u|^2\di \mathcal{H}^N\\
   	\le & C(N) \delta^{-2}\fint_{B_{2}(p)}|\Hess u|^2\di \mathcal{H}^N\\
   	\le & C(N,\tau) \delta^2\\
   	= & C(N)\delta^2 \le \eps\, .
   \end{align*}
The sought conclusion follows from \autoref{rm:volumeonboundary balls}, which yields $\mathcal{H}^N(B_{\tau^4 r_x}(x))\ge \frac{1}{4}\omega_N(\tau^4 r_x)^{N}$.

\subsubsection{Proof of \autoref{thm:existenceofneck}}\label{subsubsec:exneck}

We can now complete the proof of the existence of neck regions, together with the small content estimate. First we are going to apply \autoref{thm:existenceofweakneck}. Then we refine the approximate singular set and the radius function to get a neck region. In the last step of the proof we prove the content bound relying on \eqref{eq:maxweakneck} and on the structure of neck regions \autoref{thm:neckstructure}.

\medskip

\textbf{Refinement of approximate singular set and radius.}
Let $\eps',\delta'>0$ to be chosen later in terms of $\eps$, $\delta$ and $N$. Given $\eta<\eta(\eps',\delta',N)$ we apply \autoref{thm:existenceofweakneck} with $\eps'$ and $\delta'$ in place of $\eps$ and $\delta$ to obtain a weak neck region
	\begin{align}
		\tilde{\mathcal{N}}:=B_2(p)\setminus \left(\tilde{\mathcal{C}}_0\cup
 \bigcup_{\tilde{x}\in \tilde{\mathcal{C}}_+}\bar B_{\tilde{r}_{\tilde{x}}}(\tilde{x})\right)\, .
	\end{align}
Let $\gamma=1/8$ be the iteration scale in the proof of \autoref{thm:existenceofweakneck}.	
%	

%%%%%%%%%%

In order to refine the weak neck region $\tilde{\mathcal{N}}$, let us build an approximate singular set $\tilde{\mathcal{S}}$ as follows.	 
For any $x\in B_2(p)$ we denote by $\tilde{x}\in\tilde{ \mathcal{C}}$ a point verifying $\dist(\tilde{ \mathcal{C}},x) = \dist(\tilde{x},x)$ and set $s_x:= 2\max\{\dist(\mathcal{C}, x), \tilde{r}_{\tilde{x}}\}$.

We say that $x\in B_2(p)$ belongs to $\tilde{ \mathcal{S}}$ if $x\in \cup_{\tilde{y}\in\tilde{\mathcal{C}}} B_{ \tau \tilde{r}_{\tilde{y}}}(\tilde{y})$ and 
\begin{equation}\label{eq:enlsing}
x\in B_{\eps's_x}\, \mathcal{L}_{\tilde{x}_, s_x} \, .
\end{equation}

Now, let us define a radius function on $\tilde{\mathcal{S}}$ as
	\begin{equation}\label{eq:defry}
	r_x:= \tau^2 \max \{ \dist(x,\tilde{C}), \tau^4 \tilde r_{\tilde x} \}.
	\end{equation}
It is easily seen that $\Lip r_x\le \tau^2$, $\tilde{\mathcal{C}}\subset \tilde{\mathcal{S}}$ and $r_x=0$ for any $x\in \tilde{ \mathcal{C}}_0$.

Choose a maximal disjoint collection $\{B_{\tau^2 r_x}(x),x\in \tilde{\mathcal{S}}\}$ whose set of centers $\mathcal{C} = \mathcal{C}_0 \cup \mathcal{C}_+$ satisfies $\mathcal{C}_0 = \tilde{ \mathcal{C}}_0$ and  $\tilde{\mathcal{C}}_{+}\subset \mathcal{C}_+$.

We claim that
\begin{align}
	\mathcal{N}:=  B_2(p)\setminus \Big(\mathcal{C}_0\cup \bigcup_{x\in\mathcal{C}_{+}}\bar B_{r_x}(x)\Big)
\end{align}
is an $(\eps,\delta)$-neck region for $\eps'\le \eps'(\eps,N)$ and $\delta'\le \delta'(\eps,\delta,N)$.
\medskip

\textbf{Proof of the neck region properties for $\mathcal{N}$.}
The Vitali covering condition (i) in the neck region \autoref{def:Neckregion} is satisfied by the construction. Moreover, as we already pointed out, it holds $\Lip r_x\le \tau^2$.
\medskip

 Next, note that $B_{r}(x)$ is an $(\eps'^2\tau^{-6})$-boundary ball for all $x\in\mathcal{C}$ and $r_x \le r \le \gamma \tau^{-4}$. This follows from \autoref{remark:subballsofboundaryballs} thanks to the following observation: with this choice of parameters, all the balls $B_r(\tilde{x})$ for $\tilde{x}\in\tilde{\mathcal{C}}$ and $r>\tilde{r}_{\tilde{x}}$ in the weak neck region are $\gamma^{-1}\eps'^2$-boundary balls and the center point $x$ belongs to the approximate singular set by \eqref{eq:enlsing}.\\ 
 Since $\gamma^{-1}\eps'^2\tau^{-6}\le \eps^2$ for $\eps'\le \eps'(N,\eps)$, this proves condition (ii) in \autoref{def:Neckregion}.

\medskip

The first inclusion in \autoref{def:Neckregion} (iii) is also satisfied by the very construction of $\tilde{\mathcal{S}}$, since the center points all belong to the approximate singular set above their own scale.
\medskip

Let us now verify the second inclusion in (iii) of \autoref{def:Neckregion}, which is the main reason for the refinement of the weak neck region $\tilde{\mathcal{N}}$ into $\mathcal{N}$.

Let $x\in \mathcal{C}$ and $r_x \le r \le \tau^{-3}$. Then, if $\tilde{x}\in\tilde{\mathcal{C}}$ is such that $\dist(x,\tilde{x})=\dist(x,\tilde{\mathcal{C}})$,
\begin{equation}
	\mathcal{L}_{x,r}\cap B_r(x) \subset  B_{ 2 \tau^{-2} r } (\tilde{x})\, ,
\end{equation}
since $r\ge r_x \ge \tau^{2}\dist(x,\tilde{x})$ by \eqref{eq:defry}.\\
The cone splitting \autoref{thm:conesplittingquant} gives
\begin{equation}\label{eq:conscon}
	\mathcal{L}_{x,r}\cap B_r(x) \subset B_{\tau r/4} (\mathcal{L}_{\tilde{x}, 2\tau^{-2} r})
\end{equation}
for $\eps\le \eps(N)$. The application of the cone splitting theorem is justified by the fact that, as we already pointed out above, the balls $B_s(x)$ and $B_t(\tilde{x})$ are $\eps^2$-boundary balls for $s\ge r_x$ and $t\ge r_x$.

In order to get (iii) of \autoref{def:Neckregion} it is enough to see that
\begin{equation}\label{z102}
	\mathcal{L}_{\tilde{x}, 2\tau^{-2} r}\cap B_r(x) \subset B_{200 \tau \max\{2r,r_y\}}(\mathcal{C})\, .
\end{equation}
Indeed, if \eqref{z102} holds, then the Lipschitz property $\Lip r_x \le \tau^{2}$ gives
\begin{align*}
\mathcal{L}_{x,r}\cap B_r(x) &\overset{\eqref{eq:conscon}}{\subset} B_{\tau r/4} (\mathcal{L}_{\tilde{x}, 2\tau^{-2} r}) \cap B_r(x)\\
&\overset{\eqref{z102}}{\subset}
B_{400 \tau \max\{ 2 r,r_y\}}(\mathcal{C}) \cap B_r(x) \\
&\subset 
B_{10^3 \tau r}(\mathcal{C}\cap B_{3r}(x))\, .
\end{align*}

Let us prove \eqref{z102}.\\
The property (iii) in the definition of weak neck region $\tilde{\mathcal{N}}$ (see \autoref{def:weakneckregion}) gives the inclusion
\begin{equation}
	\cL_{\tilde{x}, 2\tau^{-2} r} 
	\subset \overline{B}_{100\tau^3 \max\{2\tau^{-2} r, \tilde r_{\tilde y}\}}(\tilde{ \mathcal{C}}) = \bigcup_{\tilde{r}_{\tilde y} \le 2 \tau^{-2}} r \overline{B}_{200\tau r}(\tilde y) 
	\quad
	\cup \quad
	\bigcup_{\tilde{r}_{\tilde y}\ge  2 \tau^{-2}} r   \overline{B}_{100 \tau^3 \tilde{r}_{\tilde y} }(\tilde y)\, .
\end{equation}
Moreover, the cone splitting \autoref{thm:conesplittingquant} implies
\begin{equation}
	\cL_{\tilde{x}, 2 \tau^{-2} r} \cap \overline{B}_{100\tau^3 \tilde{r}_{\tilde y}}(\tilde{y}) 
	\subset 
	\overline{B}_{\tau^{20}\tilde{r}_{\tilde y}} \cL_{\tilde{y}, \tau^2 \tilde{r}_{\tilde y}} \quad  \text{for any $\tilde{y}\in \tilde{ \mathcal{C}}$ with $\tilde{r}_{\tilde y}\ge  2 \tau^{-2}r$}\, ,
\end{equation}
for any $\eps'\le \eps'(\eps,N)$.\\ 
To finish the proof of \eqref{z102} we are going to prove that
\begin{align}
	 (B_{\tau^{20}\tilde{r}_{\tilde y}}\cL_{\tilde{y}, \tau^2 \tilde{r}_{\tilde y}}) \cap B_{2\tau^{-2}r}(\tilde x) &\subset B_{2\tau^{20}\tilde{r}_{\tilde y
	 }}(\tilde S) \cap B_{2\tau^{-2}r}(\tilde x) \label{z103} \\ 
	 &\subset B_{\tau/4 r }(\tilde S) \label{eq:secondincl}\\ 
	 &\subset B_{\tau/4 \max\{2r,2}r_y\}(\mathcal{C})\label{eq:lastincl}\, .
\end{align}
Let us first check \eqref{z103}. In particular we claim that $\cL_{\tilde{y}, \tau^2 \tilde{r}_{\tilde y}} \subset B_{\tau^{20}\tilde{r}_{\tilde y}}(\tilde S)$.\\ 
Given $z\in \cL_{\tilde{y}, \tau^2 \tilde{r}_{\tilde y}}$ we consider $\tilde z \in \mathcal{C}$ such that $\dist(z,\tilde{z})=\dist(z,\tilde{\mathcal{C}})$. If $\dist(z,\tilde z)\le \tau^{20} \tilde{r}_{\tilde{y}}$ then $z\in B_{\tau^{20}\tilde{r}_{\tilde y}}(\tilde S)$, since $\tilde z\in \tilde S$. If $\dist(z,\tilde z) > \tau^{20} \tilde{r}_{\tilde{y}}$ then
\begin{equation}
	z\in \cL_{\tilde{y}, \tau^2 \tilde{r}_{\tilde y}} \cap B_{s_z}(\tilde z) \subset  B_{\eps s_z} \cL_{\tilde z, s_z}\, , \quad \text{where $s_z:= 2\max\{\dist(z,\tilde z), \tilde{r}_{\tilde z}\}\ge 2\tau^{20}\tilde{r}_{\tilde y}$ }\, ,
\end{equation}
as a consequence of the cone splitting \autoref{thm:conesplittingquant}. This implies $z\in \tilde{S}$ by \eqref{eq:enlsing}, since we already know that $z\in B_{\tau^{2}\tilde{r}_{\tilde y}}(\tilde y) \subset \cup_{\tilde{w}\in\tilde{\mathcal{C}}} \bar B_{\tau \tilde{r}_{\tilde w}}(\tilde w)$.

To prove \eqref{eq:secondincl} we rely on the inequalities $\tau^{20} \tilde{r}_{\tilde y} \le \tau^{10} r_y$, $r_{\tilde x} \le r$ and the Lipschitz bound $\Lip r_y \le \tau^2$. The last inclusion \eqref{eq:lastincl} follows from the definition of $\mathcal{C}$.

\medskip

Let us now verify condition (iv) in \autoref{def:Neckregion}. 
For any $x\in \cC$ we consider a point in $\tilde{x}\in\tilde{\mathcal{C}}$ such that $\dist(x,\tilde{\mathcal{C}})=\dist(x,\tilde{x})$. Observe that, for any $r_x\le r\le \tau^3$, it holds $B_r(x)\subset B_{\tau^{-6}r}(\tilde{x})$. The sought conclusion follows from 
\begin{align}
r^2 \dashint_{B_r(x)} |\Hess u|^2 \di \haus^N
&\le C(\tau,N) \tau^{-12}r^2 \dashint_{B_{\tau^{-6}r}(\tilde x)} |\Hess u|^2 \di \haus^N\\
&\overset{\eqref{eq:weakneckhess}}{\le} C(\tau,N) r \delta'^2\, , 
\end{align}
by choosing $\delta'\le \delta'(N,\delta)$. 

If $r\ge \tau^3$ instead, the sought bound can be easily obtained by recalling that $u: B_{4\tau^{-4}}(p)\to \setR^{N-1}$ is a $\delta'$-splitting map and, again, choosing $\delta'\le \delta'(\delta,N)$.

\medskip

\textbf{Proof of the small content bound \eqref{eq:maxneck}. }
By the very construction of the neck region $\mathcal{N}$, the following holds:
\begin{equation}\label{eq:inclusionbadballs}
\mathcal{C}_+\subset\bigcup_{\tilde{x}\in\tilde{\mathcal{C}}_+}B_{\tau\tilde{r}_{\tilde{x}}}(\tilde{x})\, .
\end{equation}
Moreover, for any $\tilde{x}\in\tilde{\mathcal{C}}_+$ and for any $x\in\mathcal{C}_+\cap B_{\tau\tilde{r}_{\tilde{x}}}(\tilde x)$ it holds that
\begin{equation}\label{z1000}
r_x\le \tau^6 \tilde{r}_{\tilde{x}}\le \tau \tilde{r}_{\tilde x}  \, ,
\end{equation}
by definition of radius function \eqref{eq:defry}.

Denoting by $\mu$ the packing measure of the neck region $\mathcal{N}$ as in \autoref{def:packingmeasure}, it holds
\begin{align}
\sum_{x\in\mathcal{C}_+}r_x^{N-1}&=\mu(\mathcal{C}_+)\overset{\eqref{eq:inclusionbadballs}}{\le}\sum_{\tilde{x}\in \tilde{ \mathcal{C}}_+}\mu(B_{\tau\tilde{r}_{{\tilde{x}}}}(\tilde{x}))\\
&  \overset{\eqref{z1000}\, \eqref{eq:ahlforsneck}}{\le}  C(N) \sum_{\tilde{x}\in \tilde{ \mathcal{C}}_+} \tilde{r}_{\tilde x}^{N-1} \\
&\overset{\eqref{eq:maxweakneck}}{\le} C(N) \eps' \le \eps \, .
\end{align}

% % % % % %
% % % % % %
% % % % % %
\vspace{.5cm}

\section{Neck Decomposition}\label{sec:neckdecomposition}

The main result of this part is \autoref{thm:neckdecompositiontheorem} below. Its proof is based on the combination of the boundary-interior decomposition \autoref{thm:decompositiontheorem}, dealing with a decomposition of the space into a singular part, regular balls and boundary balls with content bounds, the existence of neck regions \autoref{thm:existenceofneck} and the structure theorem for neck regions \autoref{thm:neckstructure}. The covering arguments needed in the proof are essentially those of  \cite[Section 7]{JiangNaber16} (see also \cite[Section 10]{CheegerJiangNaber18}) and we will only sketch them in most cases.

\begin{theorem}[Neck decomposition theorem]\label{thm:neckdecompositiontheorem}
Let $\eta>0$ with $\delta<\delta(N,v,\eta)$ and consider a noncollapsed $\RCD(-(N-1),N)$ m.m.s. $(X,\dist,\haus^N)$.   For any $p\in X$ such that $\mathcal{H}^N(B_1(p))\ge v$, there exists a decomposition
\begin{equation}
B_1(p)\subset\bigcup_a\left(\mathcal{N}_a\cap B_{r_a}(x_a)\right)\cup\bigcup_bB_{r_b}(x_b)\cup\mathcal{S}^{\delta,\eta}\, ,
\end{equation}

\begin{equation}\label{eq:decompositionboundary}
\mathcal{S}^{\delta,\eta}\subset\bigcup_a\left(\mathcal{C}_{0,a}\cap B_{r_a}(x_a)\right)\cup\tilde{\mathcal{S}}^{\delta,\eta}\, ,
\end{equation}
such that the following hold:
\begin{itemize}
\item[i)] for any $a$ the set $\mathcal{N}_a=B_{2r_a}(x_a)\setminus\bar{B}_{r_x}(\mathcal{C}_a)$ is an $(\eta,\delta)$-neck region with $(N-1)$-singular set $\mathcal{C}_{0,a}\subset B_{2r_a}(x_a)$;
\item[ii)] for any $b$ the ball $B_{2r_b}(x_b)$ is $(N,\eta)$-symmetric and $r_b^2\le \eta$;
\item[iii)] $\sum_ar_a^{N-1}+\sum_br_b^{N-1}+\mathcal{H}^{N-1}(\mathcal{S}^{\delta,\eta})\le C(N,v,\delta,\eta)$;
\item[iv)] it holds $\mathcal{H}^{N-1}(\tilde{\mathcal{S}}^{\delta,\eta})=0$;
\item[v)] the singular set $\mathcal{S}^{\delta,\eta}$ is $(N-1)$-rectifiable;
\item[vi)] if $\eta<\eta(N,v)$ and $\delta<\delta(N,v,\eta)$, then $\mathcal{S}^{N-1}\setminus \mathcal{S}^{N-2}\subset \mathcal{S}^{\delta,\eta}$.
\end{itemize}
\end{theorem}

The next theorem is typically the first step in the proof of the Neck Decomposition \autoref{thm:neckdecompositiontheorem}.  We emphasize it here, as once we have proven the $\eps$-regularity \autoref{thm:epsregintro} this leads to our Boundary Structure \autoref{thm:boundatystructure}:

% % % %

\begin{theorem}[Boundary-Interior decomposition theorem]\label{thm:decompositiontheorem}
For any $\eta>0$ and $\RCD(-(N-1),N)$ m.m.s. $(X,\dist,\mathcal{H}^N)$ with $p\in X$ such that $\mathcal{H}^N(B_1(p))\ge v$, there exists a decomposition
\begin{equation}
B_1(p)\subset\bigcup_{a}B_{r_a}(x_a)\cup\bigcup_{b}B_{r_b}(x_b)\cup\tilde{\mathcal{S}}\, ,
\end{equation}
such that the following hold:
\begin{itemize}
\item[i)] the balls $B_{4\tau^{-4}r_a}(x_a)$ are $\eta$-boundary balls and $r_a^2\le \eta$;
\item[ii)] the balls $B_{2r_b}(x_b)$ are $(N,\eta)$-symmetric and $r_b^2\le \eta$;
\item[iii)]  $\tilde{\mathcal{S}}\subset\mathcal{S}$ and $\mathcal{H}^{N-1}(\tilde{\mathcal{S}})=0$;
\item[iv)] $\sum_br_b^{N-1}\le C(N,v,\eta)$;
\item[v)] $\sum_ar_a^{N-1}\le C(N,v)$.
\end{itemize}
\end{theorem}

The proof of \autoref{thm:decompositiontheorem} proceeds via an iterative recovering argument. In the next subsection, we introduce various rougher decomposition which include different types of balls, and we show how we can iteratively get rid of them. The arguments are those of  \cite[Section 10]{CheegerJiangNaber18}, with some simplifications due to the strong rigidity of singularities in codimension one. We only sketch the proofs, omitting the details and referring to \cite{CheegerJiangNaber18} for all the details.

\subsection{Proof of the neck decomposition theorem}

Let us state an intermediate decomposition result and show how it can be used to derive the boundary-interior decomposition \autoref{thm:decompositiontheorem}. Here a different type of balls appears, which are not $\eta$-boundary balls nor $(N,\eta)$-symmetric balls but have a definite volume drop with respect to the background scale. 

% % % %

\begin{proposition}\label{prop:inductionstepdecomposition}
For any $\eta>0$ there exists ${\nu}^0(N,v,\eta)>0$ such that, if $(X,\dist,\mathcal{H}^N)$ is an $\RCD(-(N-1),N)$ m.m.s. and $B_1(p)\subset X$ verifies $\mathcal{H}^N(B_1(p))\ge v$, then there exists a decomposition
\begin{equation}
B_1(p)\subset\bigcup_a B_{r_a}(x_a)\cup\bigcup_bB_{r_b}(x_b)\bigcup_vB_{r_{\nu}}(x_{\nu})\cup\tilde{\mathcal{S}}
\end{equation}
such that the following hold:
\begin{itemize}
\item[i)] for any $a$ the ball $B_{4\tau^{-4}r_a}(x_a)$ is an $\eta$-boundary ball and $r_a^2\le \eta$;
\item[ii)] for any $b$ the ball $B_{2r_b}(x_b)$ is $(N,\eta)$-symmetric and $r_b^2\le \eta$;
\item[iii)] $\sum_br_b^{N-1}+\sum_{\nu}r_{\nu}^{N-1}\le C(N,\eta,v)$;
\item[iv)] $\sum_ar_a^{N-1}\le C(N,v)$;
\item[v)] $\tilde{\mathcal{S}}\subset \mathcal{S}$ and $\mathcal{H}^{N-1}(\tilde{\mathcal{S}})=0$;
\item[vi)] for any $\nu$, $\inf_{y\in B_{4r_{\nu}}(x_{\nu})}\mathcal{V}_{r_{\nu}}(y)\ge \inf_{y\in B_4(p)}\mathcal{V}_1(y)+\nu^0$.
\end{itemize}

\end{proposition}

Let us now prove the boundary-interior decomposition \autoref{thm:decompositiontheorem}. In order to do so we just need to iteratively apply a finite number of times \autoref{prop:inductionstepdecomposition} to get rid of $\nu$-balls in the decomposition.

\begin{proof}[Proof of \autoref{thm:decompositiontheorem}]
In order to prove \autoref{thm:decompositiontheorem} given \autoref{prop:inductionstepdecomposition} we just need to follow the first part of the proof of \cite[Theorem 2.12]{CheegerJiangNaber18}. After a finite number of iterations of the induction step decomposition we get a decomposition with only $a$-balls, $b$-balls and a subset of the singular stratum $\mathcal{S}^{N-2}$.
\end{proof}

The remainder of this subsection is devoted to the proof of \autoref{prop:inductionstepdecomposition}.\\ 
We are going to consider constants $\xi,\delta,\gamma,\eps$ which in general will satisfy
\begin{equation}
0<\xi\ll \delta<\gamma<\eps<\eps(N)\, .
\end{equation}
We will assume additionally that $(X,\dist,\mathcal{H}^{N})$ is an $\RCD(-\xi(N-1),N)$ space. The general cases of all the statements can be achieved via additional covering arguments.

Let us introduce the notation for the various families of balls we will use in the intermediate steps of our arguments.

We recall that the set with small volume pinching has been introduced in \eqref{eq:setsmallvolpinch} and the almost cone splitting via content \autoref{thm:almostconesplittingcontent}, to which we refer for the various constants appearing below.

Any ball $B_r(x)$ will be of one or more of these types, indexed by letters b, c, d and e:

\begin{itemize}
\item[i)] $b$-balls $B_{r_b}(x_b)$ are balls such that $B_{2r_b}(x_b)$ is $(N,\eta)$-symmetric;
\item[ii)] $c$-balls $B_{r_c}(x_c)$ are balls which are not $b$-balls but satisfy
\begin{equation}\label{eq:cball}
\mathcal{H}^{N}(B_{\gamma r_c}(\mathcal{P}_{r_c,\xi}(x_c)))\ge\eps\gamma r_c^N\, ;
\end{equation}
\item[iii)] $d$-balls are balls $B_{r_d}(x_d)$ for which $\mathcal{P}_{r_d,\xi}(x_d)\neq\emptyset$  but 
\begin{equation}\label{eq:dballs}
\mathcal{H}^{N}(B_{\gamma r_d}(\mathcal{P}_{r_d,\xi}(x_d)))<\eps\gamma r_d^N\, ;
\end{equation}
\item[iv)] $e$-balls $B_{r_e}(x_e)$ for which $\mathcal{P}_{r_e,\xi}(x_e)=\emptyset$.
\end{itemize}

\begin{remark}\label{rm:recognizingnuballs}
Let us point out that any $e$-ball $B_{r_e}(x_e)$ can be covered by $\nu$-balls as in vi) of \autoref{prop:inductionstepdecomposition} in such a way that
\begin{equation}
B_{r_e}(x_e)\subset \bigcup_{\nu}B_{r_{\nu}}(x_{\nu})
\end{equation}
and $\sum_{\nu}r_{\nu}^{N-1}\le C(N,v)r_e^{N-1}$.\\ 
In order to do so it is sufficient to consider a Vitali covering of $B_{r_e}(x_e)$ with balls $B_{\xi r_e}(x_e^i)$ such that $x_e^i\in B_{r_e}(x_e)$ and the balls $B_{\xi r_e/5}(x_e^i)$ are disjoint. At the end of the proof of \cite[Proposition 10.2]{CheegerJiangNaber18} it is verified that the balls $B_{\xi r_e}(x_e^i)$ are $\nu$-balls with $\nu_0=\xi$ and the content estimate follows from the Vitali covering property. 
\end{remark}

\begin{remark}\label{rm:recognizingaballs}
Let us see how to recover boundary balls starting from $c$-balls.\\
We wish to prove that, for $\delta$ sufficiently small, any $c$-ball $B_{r_c}(x_c)$ is such that $B_{\tau^{-6}r_c}(x'_c)$ is an $\eta$-boundary ball for some $x'_c\in B_{4r_c}(x_c)$, in particular $B_{r_c}(x_c)\subset B_{\tau^{-6}r_c}(x_c')$.\\ 
In order to do so we argue by contradiction. Recall that the parameters are set in such a way that the assumptions of the cone splitting via content \autoref{thm:almostconesplittingcontent} are satisfied. Observe that, if $B_{r_c}(x_c)$ is a $c$-ball (see \eqref{eq:cball}), then there exists $x_c'\in B_{4r_c}(x_c)$ such that $B_{\delta^{-1}r_c}(x_c')$ is $(N-1,\delta^2)$-symmetric. Since by assumption $B_{r_c}(x_c)$ is not $(N,\eta)$-symmetric, it is easy to check arguing by contradiction that, for $\delta$ sufficiently small, $B_{\tau^{-6}r_c}(x'_c)$ is an $\eta$-boundary ball.

\end{remark}

\begin{proposition}\label{prop:firstdec}
Let $v>0$ be fixed. For any $\eps\le\eps(N,v)$, $\gamma\le \gamma(N,v,\eps)$, $\delta\le \delta(N,v,\eta)$ and $\xi\le\xi(N,v,\eps,\gamma,\delta,\eta)$ and for any $\RCD(-\xi(N-1),N)$ m.m.s. $(X,\dist,\mathcal{H}^N)$ and $B_1(p)\subset X$ such that $\mathcal{H}^N(B_1(p))\ge v$ the following holds.
There exists a decomposition
\begin{equation}
B_1(p)\subset\bigcup_bB_{r_b}(x_b)\cup\bigcup_cB_{r_c}(x_c)\cup\bigcup_eB_{r_e}(x_e)\cup\tilde{\mathcal{S}}\, ,
\end{equation}
where we are adopting the usual notation for the various types of balls,
\begin{equation}
\sum_br_b^{N-1}+\sum_cr_c^{N-1}+\sum_er_e^{N-1}\le C(N,\gamma)\, ,
\end{equation}

\begin{equation}
\sum_cr_c^{N-1}\le C(N,v)
\end{equation}
and $\tilde{\mathcal{S}}\subset\mathcal{S}$, with $\haus^{N-1}(\tilde{\mathcal{S}})=0$. 
\end{proposition}

\begin{proof}
Specializing \cite[Proposition 10.3]{CheegerJiangNaber18} to the case $k=N-1$, we obtain that there exist
$\eps\le\eps(N,v)$, $\gamma\le\gamma(N,v,\eps)$ and $\delta\le \delta(N,v,\eta)$ such that, if the additional assumption $\mathcal{H}^N(B_{\gamma}(\mathcal{P}_{1,\xi}(p)))<\eps\gamma$ is satisfied (that is to say $B_1(p)$ is a $d$-ball), the following holds: there exists a decomposition 
\begin{equation}
B_1(p)\subseteq \tilde{\mathcal{S}}_d\cup\bigcup_bB_{r_b}(x_b)\cup\bigcup_cB_{r_c}(x_c)\cup\bigcup_eB_{r_e}(x_e)\, ,
\end{equation} 
where 
\begin{itemize}
\item[i)] $\haus^{N-1}(\tilde{\mathcal{S}}_d)=0$
\item[ii)]$\sum_br_b^{N-1}+\sum_{e}r_e^{N-1}\le C(N,\gamma)$;
\item[iii)] $\sum_cr_c^{N-1}\le C(N,v)$.
\end{itemize}

To conclude, let us observe that, if $B_1(p)$ is either a $b$-ball, a $c$-ball or an $e$-ball, then the statement is trivially verified. Therefore we can assume that $B_1(p)$ is a $d$-ball and the conclusion follows from what we observed in the first part of the proof.
\end{proof}

\begin{proof}[Proof of \autoref{prop:inductionstepdecomposition}]
We divide the proof into two steps.

In the first one we reduce ourselves to balls such that, after rescaling of the space at the scale of their radii the lower Ricci curvature bound is $-\xi(N-1)$. Then, relying on \autoref{rm:recognizingnuballs} and \autoref{rm:recognizingaballs}, we get the sought decomposition starting from \autoref{prop:firstdec}.

\textbf{Step 1.} Considering a Vitali covering of $B_1(p)$ with balls of sufficiently small radius we reduce to balls that, when rescaled to radius one, verify the assumptions of \autoref{prop:firstdec}. The number of these balls can be controlled due to the Vitali property.

\textbf{Step 2.}
Given any ball arising from the first step, we apply \autoref{prop:firstdec}. Thanks to \autoref{rm:recognizingnuballs} we cover any $e$-ball with $\nu$-balls keeping to content bound. Next, relying on \autoref{rm:recognizingaballs}, for any $c$-ball $B_{r_c}(x_c)$ we find $x_a\in B_{4r_c}(x_c)$ such that $B_{\tau^{-6}r_c}(x_a)$ is an $\eta$-boundary ball and we substitute the given $c$-ball with the $a$-ball $B_{4\tau^{-4}r_c}(x_a)$. Since $B_{r_c}(x_c)\subset B_{r_a}(x_a)$ we keep the covering property and also the content bound is preserved.  

\end{proof}

The proof of the neck decomposition \autoref{thm:neckdecompositiontheorem} with properties i) to iv) relies on the iterative application of the boundary-interior decomposition \autoref{thm:decompositiontheorem} together with the existence of neck regions \autoref{thm:existenceofneck} and the structure \autoref{thm:neckstructure} to take care of the content and $\mathcal{H}^{N-1}$-measure estimates.

\begin{proof}[Proof of \autoref{thm:neckdecompositiontheorem}]

% % %
In the following we are going to denote by $B_{r_f}(x_f)$ a ball which has not been identified with an $a$-ball or $b$-ball yet.

Let us start combining and rephrasing \autoref{thm:existenceofneck} and \autoref{thm:neckstructure} in a way convenient for our purposes: for any $\eps>0$, $\eta<\eta(\eps,N)$ and $\delta\le\delta(N,\eta)$ there exists $\eta'>0$ such that, if a ball $B_{4\tau^{-4}r}(p)$ is an $\eta'$-boundary ball, then there exists an $(\eta,\delta)$-neck region $\mathcal{N}=B_{2r}(p)\setminus\bar{B}_{r_x}(\mathcal{C})$ over $B_{2r}(p)$ such that

\begin{itemize}

\item[i)]$
B_r(p)\subset\left(\mathcal{N}\cap B_{r}(p)\right)\cup \mathcal{C}_0\cup\bigcup_f B_{2r_f}(x_f);$ 
\item[ii)] $\mathcal{H}^{N-1}(\mathcal{C}_0)\le A(N)r^{N-1}$; 
\item[iii)] the singular set $\mathcal{C}_0$ is biLipschitz to a subset of $\setR^{N-1}$;
\item[iv)]  and $\sum_{f}r_f^{N-1}\le \eps r^{N-1}$;
\item[v)] for any $f$ it holds that 
\begin{equation}\label{eq:volboundneck}
\mathcal{H}^N(B_{r_f}(x_f))\le \frac{2}{3}\omega_Nr^N\, . 
\end{equation}
\end{itemize}
Let us just point out that item v) above follows from \autoref{rm:volumeonboundary balls}.

In order to get the sought covering we proceed inductively. First we apply \autoref{thm:decompositiontheorem} with $\eta=\eta'$ given in the discussion above. Then we build $(\eta,\delta)$-neck regions verifying i) to v) above on any ball $B_{2r_a}(x_a)$ of the decomposition. After this first stage of the procedure we get  
\begin{equation}
B_1(p)\subset\bigcup_a\left(\mathcal{C}_{0,a}\cap B_{r_a}(x_a)\right)\cup\bigcup\left(\mathcal{N}_a\cap B_{r_a}(x_a)\right)\cup\bigcup_bB_{r_b}(x_b)\cup\bigcup_{f}B_{r_f}(x_f)\cup\tilde{\mathcal{S}}\, ,
\end{equation}
with
\begin{itemize}
\item[i)] $\sum_br_b^{N-1}\le C(N,v,\eta)$;
\item[ii)] $\sum_a\mathcal{H}^{N-1}(\mathcal{C}_{0,a})\le A(N)\sum_{a}r_a^{N-1}\le C'(N,v)$;
\item[iii)]$\tilde{\mathcal{S}}\subset\mathcal{S}$ and $\haus^{N-1}(\tilde{\mathcal{S}})=0$;
\item[iv)] $\sum_{f}r_f^{N-1}\le \eps\sum_{a}r_a^{N-1}\le C(N,v)\eps.$
\end{itemize}
Next we apply again the procedure above to the balls $B_{r_f}(x_f)$: first we perform the boundary-interior decomposition of \autoref{thm:decompositiontheorem}, then we build neck regions on any new $\eta$-boundary ball appearing. At the first iteration we get
\begin{equation*}
B_1(p)\subset\bigcup_{a_1} \left(\mathcal{C}_{0,a_1}\cap B_{r_{a_1}}(x_{a_1})\right)\cup\bigcup_{a_1}\left(\mathcal{N}_{a_1}\cap B_{r_{a_1}}(x_{a_1})\right)\cup
\bigcup_{b_1}B_{r_{b_1}}(x_{b_1})\cup\bigcup_{f_1}B_{r_{f_1}}(x_{f_1})\cup\tilde{\mathcal{S}}_1\, ,
\end{equation*}
with
\begin{itemize}
\item[i)] $\sum_{b_1}(r_{b_1})^{N-1}\le C(N,v,\eta)(1+C(N,v)\eps)$;
\item[ii)] $\sum_{a_1}(r_{a_1})^{N-1}\le C(N,v)(1+\eps C(N,v))$;
\item[iii]$\sum_{a_1}\mathcal{H}^{N-1}(\mathcal{C}_{0,{a_1}})\le A(N)C(N,v)(1+\eps C(N,v))$;
\item[iv)] $\tilde{\mathcal{S}}_1\subset\mathcal{S}$ and $\haus^{N-1}(\tilde{\mathcal{S}}_1)=0$;
\item[v)] $\sum_{f_1}(r_{f_1})^{N-1}\le \eps\sum_{f}r_{f}^{N-1}\le C(N,v)\eps^2.$
\end{itemize} 
% % %
Arguing by induction, after $n$ iterations of the scheme we get 
\begin{equation*}
B_1(p)\subset\bigcup_{a_n}\left(\mathcal{C}_{0,a_n}\cap B_{r_{a_n}}(x_{a_n})\right)\cup\bigcup\left(\mathcal{N}_{a_n}\cap B_{r_{a_n}}(x_{a_n})\right)\cup\bigcup_{b_n}B_{r_{b_n}}(x_{b_n})\cup\bigcup_{f_n}B_{r_{f_n}}(x_{f_n})\cup\tilde{\mathcal{S}}_n\, ,
\end{equation*}
with
\begin{itemize}
\item[n-i)] $\sum_{b_n}(r_{b_n})^{N-1}\le C(N,v,\eta)(1+C(N,v)\eps+\dots+(C(N,v)\eps)^n)$;
\item[n-ii)] $\sum_{a_n}(r_{a_n})^{N-1}\le C(N,v)(1+C(N,v)\eps+\dots+(C(N,v)\eps)^n)$;
\item[n-iii)] $\sum_{a_n}\mathcal{H}^{N-1}(\mathcal{C}_{0,a_n})\le C(N,v)(1+C(N,v)\eps+\dots+(C(N,v)\eps)^n)$;
\item[n-iv)] $\tilde{\mathcal{S}}_n\subset\mathcal{S}$ and $\haus^{N-1}(\tilde{\mathcal{S}}_n)=0$;
\item[n-v)] $\sum_{f_n}(r_{f_n})^{N-1}\le C(N,v)\eps^n.$
\end{itemize}

Next we wish to pass to the limit the construction above.\\
To this aim choose $\eps$ small enough to ensure that $C(N,v)\eps<1$, $\eta$ and $\delta$ accordingly and let us set 
\begin{equation}
\tilde{\mathcal{S}}_f:=\bigcap_{n\ge 1}\bigcup_{f_n}B_{2r_{f_n}}(x_{f_n})\, .
\end{equation}
Furthermore we denote by $a$ any index belonging to $\cup_n\left\lbrace a_n\right\rbrace $ and by $b$ any index belonging to $\cup_n\left\lbrace b_n\right\rbrace $. Observe that $\left\lbrace a_n \right\rbrace\subset\left\lbrace a_m\right\rbrace  $ if $n\le m$ and and analogous inclusion holds for the indexes b.\\
Then it is easy to check that
\begin{equation}
B_1(p)\subset\bigcup_a\left(\mathcal{C}_{0,a}\cap B_{r_a}(x_a)\right)\cup\bigcup\left(\mathcal{N}_a\cap B_{r_a}(x_a)\right)\cup\bigcup_bB_{r_b}(x_b)\cup\tilde{\mathcal{S}}_f\cup\bigcup_{n\ge 1}\tilde{\mathcal{S}}_n\, .
\end{equation}
Passing to the limit n-i),n-ii) and n-iii) we can easily verify that:
\begin{itemize}
\item[i)] $\sum_ar_a^{N-1}+\sum_br_b^{N-1}\le C(N,v,\eta)$;
\item[ii)] $\sum_a\mathcal{H}^{N-1}(\mathcal{C}_{0,a})\le C(N,v)$
\item[iii)]$\cup_{n\ge 1}\tilde{\mathcal{S}}_n\subset\mathcal{S}$ and $\haus^{N-1}(\cup_{n\ge 1}\tilde{\mathcal{S}}_n)=0$.
\end{itemize}
To conclude we are left to verify that
\begin{equation}
\mathcal{H}^{N-1}\left(\tilde{\mathcal{S}}_f\right)=0\quad\text{ and }\quad \tilde{\mathcal{S}}_f\subset\mathcal{S}\, .
\end{equation}
The first conclusion can be checked relying on n-iv) above, taking into account the definition of the Hausdorff pre-measures $\mathcal{H}^{N-1}_{\xi}$.\\
The second conclusion can be verified since the balls $B_{r_{f_n}}(x_{f_n})$ satisfy the volume bounds \eqref{eq:volboundneck}. Therefore, at any point $x\in\tilde{\mathcal{S}}_f$, it holds that $\lim_{r\to 0}\mathcal{H}^{N}(B_r(x))/\omega_Nr^N<1$, hence $x\in\mathcal{S}$.

All in all, letting
\begin{equation}\label{eq:decsingularset}
\mathcal{S}^{\delta,\eta}:=\bigcup_a\left(\mathcal{C}_{0,a}\cap B_{r_a}(x_a)\right)\cup\tilde{\mathcal{S}}_f\cup\bigcup_{n\ge 1}\tilde{\mathcal{S}}_n\, ,
\end{equation}
we get the neck decomposition verifying the sought properties in the statement.

% % % 
\smallskip
To address v) we just point out that, by \eqref{eq:decsingularset}, $\mathcal{S}^{\delta,\eta}$ is covered by the countable union 
\begin{equation}
\bigcup_{a}\left(\mathcal{C}_{0,a}\cap B_{r_a}(x_a)\right)
\end{equation}
up to $\mathcal{H}^{N-1}$-negligible sets. Therefore it is $(N-1)$-rectifiable by the neck structure \autoref{thm:neckstructure}. 

\smallskip
Now we deal with vi). In order to do so we follow the last part of the proof of \cite[Theorem 2.12]{CheegerJiangNaber18}, with simplifications due to the rigidity of codimension one. 

We claim that the following hold: 
\begin{itemize}
\item[a)] if $\eta<\eta(N)$ and  $B_{\tau^{-1}r_b}(x_b)$ is an $(N,\eta)$-symmetric ball such that $r_b^2(N-1)\le \eta$, then there is no point of $\mathcal{S}^{N-1}\setminus\mathcal{S}^{N-2}$ in $B_{r_b}(x_b)$;
\item[b)] if $\eta<\eta(N)$ and $\delta<\delta(\eta,N)$ then no $(\eta,\delta)$-neck region $\mathcal{N}_a=B_{2r_a}(x_a)\setminus\bar{B}_{r_x}(\mathcal{C})$ can contain points of $\mathcal{S}^{N-1}\setminus\mathcal{S}^{N-2}$. 
\end{itemize}
This will certainly suffice to establish vi), so let us prove (a) and (b) above.

To prove (a) let us fix $\eps<\dist_{GH}(B_1^{\setR^N}(0),B_1^{\setR^N_+}(0))/2$. Then by volume convergence, volume monotonicity and volume rigidity, if $\eta<\eta(\eps)=\eta(N)$, any tangent cone at any point $x\in B_{r_b}(x_b)$ has unit ball $\eps$-close to the unit ball of $\setR^N$, therefore $x\notin \mathcal{S}^{N-1}\setminus\mathcal{S}^{N-2}$.

In order to prove (b) let us consider $x\in \mathcal{N}_a=B_{2r_a}(x_a)\setminus\bar{B}_{r_x}(\mathcal{C}_a)$ and let $y\in \mathcal{C}_a$ be such that $\dist(x,y)=\dist(x,\mathcal{C}_a)$. Then by the first defining condition of neck region $B_{4\dist(x,y)}(y)$ is an $\eta$-boundary ball. Therefore, $B_{\dist(x,y)/2}(x)$ is $(N,\eps)$-symmetric if $\eta<\eta(\eps)$. Then, arguing as in the proof of (a) we infer that, if $\eps<\eps(N)$, then $x\notin \mathcal{S}^{N-1}\setminus\mathcal{S}^{N-2}$.  

\end{proof}

% % % % % %
% % % % % %

\section{Boundary rectifiability and stability}\label{sec:Boundary rectifiability and stability}
This section is dedicated to the proofs of the rectifiability and first stability results for boundaries of noncollapsed $\RCD$ spaces by means of the tools developed in \autoref{sec:neckdecomposition} and \autoref{sec:neckregion}.

\subsection{Proof of the stability results}\label{subsec:stability}

Let us start with a weak $\eps$-regularity result. Basically, it amounts to saying that balls sufficiently close in the GH sense to a model boundary ball have a definite amount of boundary points. This will be sharpened later on in \autoref{cor:bdryvolconvergencereg}.

\begin{theorem}\label{thm:stabilityS}
Let $N\ge 1$ be fixed. There exists $\eta(N)>0$ and $c(N)>1$ such that, if $\eta\le \eta(N)$ and
\begin{equation}\label{eq:closehalf}
	\dist_{GH}(B_1(p), B_1^{\setR_+^N}(0))\le \eta\, ,
\end{equation}
	where $B_1(p)$ is a ball of an $\RCD(-\eta(N-1),N)$ space $(X,\dist,\haus^N)$, then 
\begin{equation}\label{eq:massest}
	c(N)^{-1} \le \haus^{N-1}(\mathcal{S}^{N-1}\cap B_1(p)) \le c(N)\, .
\end{equation}
\end{theorem}
\begin{proof}
	The lower bound in \eqref{eq:massest} follows by combining \autoref{thm:neckstructure} and \autoref{thm:existenceofneck}. Indeed by means of the latter, for $\eta>0$ small enough, we can build an $(\eps,\delta)$-neck region over $B_{4^{-1}\tau^4}(p)$ and from (ii), (iii) in \autoref{thm:neckstructure} and \eqref{eq:maxneck} we deduce that, up to take $\eps, \delta$ sufficiently small, it holds
	\[
	\haus^{N-1}(\mathcal{S}^{N-1}\cap B_1(p)) \ge \haus^{N-1}(\mathcal{C}_0) \ge \mu(B_{4^{-1}\tau^4}(p)) -\mu(\mathcal{C}_+) \ge c(N)\, .
	\]	
	
	The upper bound in \eqref{eq:massest} instead follows from \autoref{thm:neckdecompositiontheorem}. Indeed, it is sufficient to apply the neck decomposition with parameters $\eta$ and $\delta$ sufficiently small in such a way that, thanks to (vi) of \autoref{thm:neckdecompositiontheorem}, $\mathcal{S}^{N-1}\setminus\mathcal{S}^{N-2}\subset\mathcal{S}^{\delta,\eta}$ and then to rely on (iii) of the same statement to infer that 
	\[
	\haus^{N-1}\left((\mathcal{S}^{N-1}\setminus\mathcal{S}^{N-2})\cap B_1(p)\right)=\haus^{N-1}(\mathcal{S}^{N-1}\cap B_1(p))\le C(N,v)\, .
	\]
	To conclude we observe that the dependence of the constant on the volume can be removed taking into account the volume convergence \autoref{thm:volumeconvergence} and \eqref{eq:closehalf}.
\end{proof}

The $\eps$-regularity theorem above directly yields a stability result for the absence of boundary under noncollapsing pGH convergence.

\begin{theorem}\label{thm:stabimp}
	Let $(X_n,\dist_n,\haus^N,x_n)$ be a sequence of noncollapsed $\RCD(K,N)$ spaces with no boundary on $B_2(x_n)$ in the sense of \autoref{def:withoutboundary}. 
		Assume that
	\begin{equation}
		(X_n,\dist_n,\haus^N,x_n) \xrightarrow{\text{pGH}} (Y,\dist_Y,\haus^N,y)\, .
	\end{equation}
	Then $(Y,\dist_Y,\haus^N)$ has no boundary on $B_1(y)$.
\end{theorem}

\begin{proof}
	Let us argue by contradiction. Assume that there exists $z\in B_1(y)\cap \mathcal{S}^{N-1}\setminus \mathcal{S}^{N-2}$. Then we can find $r\in (0,1/5)$ such that $\dist_{GH}(B_r(z),B^{\setR_+^N}_r(0))\le \frac{\eta(N)}{2} r$ where $\eta(N)$ is as in \autoref{thm:stabilityS}. Let $X_n\ni z_n\to z\in Y$. Then we have
	\[
	\dist_{GH}(B_r(z_n), B^{\setR_+^N}_r(0)) \le \dist_{GH}(B_r(z_n),B_r(z)) + \dist_{GH}(B_r(z),B^{\setR_+^N}_r(0))
	<\eta(N) r
	\]
	for $n$ big enough. Thanks to \autoref{thm:stabilityS} above we can infer that
	\[
	(\mathcal{S}^{N-1}\setminus \mathcal{S}^{N-2})\cap B_r(z_n)\neq \emptyset\, ,
	\]
 contradicting the assumption that $X_n$ has no boundary in $B_2(x_n)\supset B_r(z_n)$.
\end{proof}

\vspace{.1cm}

\subsection{Rectifiable structure and volume estimates}

The main goal of this section is to prove \autoref{thm:boundatystructure} (i), (ii) and (iii). This will be achieved through some intermediate steps.

\begin{theorem}\label{thm:struc}
	Let $1\le N<\infty$ and $v>0$ be fixed. Let $(X,\dist,\haus^N)$ be a noncollapsed $\RCD(-(N-1),N)$ space and $p\in X$ be such that $\haus^N(B_1(p))>v>0$.
	Then the following hold:
\begin{itemize}	
\item[(i)] the singular set $\mathcal{S}^{N-1}$ is $(N-1)$-rectifiable; 
\item[(ii)] there exists a constant $C=C(N,v)>0$ such that
	\begin{equation}\label{eq:tubneighbound}
	\haus^N(B_r(\mathcal{S}^{N-1}\setminus \mathcal{S}^{N-2})\cap B_1(p)) \le C r
	\quad \text{for any $r\in (0,1)$, $p\in X$}\, .
	\end{equation} 
	In particular
\begin{equation}\label{eq:measurebound}
	\haus^{N-1}(\mathcal{S}^{N-1}\cap B_1(p))\le C
	\quad \text{for any $p\in X$}\, ;
\end{equation}
\item[(iii)] at any $x\in \mathcal{S}^{N-1}\setminus \mathcal{S}^{N-2}$, the tangent cone is unique and isomorphic to the Euclidean half space $\setR_+^{N-1}:=\set{x\in \setR^N:\, x_N\ge 0}$.
\end{itemize}
\end{theorem}

\begin{proof}[Proof of \autoref{thm:struc} (i)]
	The rectifiability of $\mathcal{S}^{N-1}$ immediately follows from \autoref{thm:neckdecompositiontheorem}. Indeed
	\[
	\mathcal{S}^{N-1}\cap B_1(p) \subset \bigcup_a \mathcal{C}_{0,a} \cup \tilde{ \mathcal{S}}^{\delta,\eta}\, ,
	\]
	where $\haus^{N-1}(\mathcal{S}^{\delta,\eta})=0$ and $\mathcal{C}_{0,a}$ is $(N-1)$-rectifiable by (iv), (v) and (vi) of \autoref{thm:neckdecompositiontheorem}.
\end{proof}

\begin{proof}[Proof of \autoref{thm:struc} (ii)]

	 We divide the proof of \eqref{eq:tubneighbound} in three steps: volume estimate for the tubular neighbourhood intersected with neck regions (Step 1), volume estimate for the tubular neighbourhood intersected with regular balls (Step 2) and combination of the previous estimates (Step 3).
	 
	 Let us point out that \eqref{eq:measurebound} can be obtained either as a consequence of \autoref{thm:neckdecompositiontheorem}, or as a consequence of the volume bound for the tubular neighbourhood \eqref{eq:tubneighbound} by a standard argument (cf. \cite[Lemma 2.5]{AntonelliBrueSemola19}). 
	
	\textbf{Step 1.} We claim that if $\eps \le \eps(N)$ and $\delta\le \delta(N,v,\eps)$, then for any $(\eps,\delta)$-neck region $\mathcal{N}_a=B_{2r_a}(x_a)\setminus \overline B_{r_x}(\mathcal{C}_a)$  it holds
	\begin{equation}\label{eq:tubestbigneck}
	\haus^N(B_r(\mathcal{S}^{N-1}\setminus \mathcal{S}^{N-2}) \cap \mathcal{N}_a\cap B_{r_a}(x_a)) \le C(N,v) r r_a^{N-1}\, ,
	\quad \text{for any $r\in (0,1)$}.
	\end{equation}
	Observe that \eqref{eq:tubestbigneck} is trivially verified when $r> r_a/2$. Indeed
	\[
	\haus^N(B_r(\mathcal{S}^{N-1}\setminus \mathcal{S}^{N-2}) \cap \mathcal{N}_a\cap B_{r_a}(x_a)) \le \haus^N(B_{r_a}(x_a))\le C(N,v) r_a^N \le 2C(N,v) r r_a^{N-1}\, .
	\]
	Therefore let us assume $r \le  r_a/2$. Notice that
	\begin{equation}\label{zzz2}
	B_r(\mathcal{S}^{N-1}\setminus \mathcal{S}^{N-2}) \cap \mathcal{N}_a\cap B_{r_a}(x_a) \subset B_{2r}(\mathcal{C}_a)\, .
	\end{equation}
	Indeed, if this is not the case we could find $x\in B_r(\mathcal{S}^{N-1}\setminus \mathcal{S}^{N-2}) \cap \mathcal{N}_a\cap B_{r_a}(x_a)$  and $y\in \mathcal{C}_a$ such that $2r\le \dist(x,\mathcal{C}_a)=\dist(x,y)=:s$. 
	Observe that $B_{2s}(y)$ is an $\eps$-boundary ball and 
	\begin{equation}
	(\mathcal{S}^{N-1}\setminus \mathcal{S}^{N-2})\cap B_{2s}(y) \subset B_{8\tau s}(\mathcal{C}_a)\, ,
	\end{equation}
	as a consequence of (iii) in \autoref{def:Neckregion} (recall that we set $\tau:=10^{-10N}$).
	Being $x\in B_r(\mathcal{S}^{N-1}\setminus \mathcal{S}^{N-2}) \cap \mathcal{N}_a\cap B_{r_a}(x_a)$ there exists 
	\begin{equation}
	z\in (\mathcal{S}^{N-1}\setminus \mathcal{S}^{N-2})\cap B_{2s}(y)\subset B_{8\tau s}(\mathcal{C}_a)
	\end{equation}
	such that $\dist(x,z)\le r$. This yields to a contradiction since
	\begin{equation}
	r\ge \dist(x,z) \ge \dist(x,\mathcal{C}_a) - 8\tau s \ge s(1-8\tau)\ge 2r(1-8\tau)\, .
	\end{equation}
	
	With \eqref{zzz2} at our disposal let us conclude the proof of \eqref{eq:tubestbigneck}. Let $x_1,\ldots ,x_m\in \mathcal{C}_a$ be such that $\set{B_{r}(x_i)}$ is a disjoint family, $2r> r_{x_i}$ for any $i=1,\dots,m$ and 
	\begin{equation}
	\set{x\in\mathcal{C}_a:\, r_x < 2r }\cap B_{r_a}(x_a)\subset \cup_{i=1}^m B_{5r}(x_i)\, .
	\end{equation}
	It is immediate to check that
	\begin{equation}
	B_{2r}(\mathcal{C}_a)\cap\mathcal{N}_a\cap B_{r_a}(x_a)
	\subset \bigcup_{i=1}^m B_{10r}(x_i)\, .
	\end{equation}
	Hence from \eqref{zzz2} we deduce
	\begin{equation}
	\haus^N(B_r(\mathcal{S}^{N-1}\setminus \mathcal{S}^{N-2}) \cap \mathcal{N}_a\cap B_{r_a}(x_a)) \le \sum_{i=1}^m \haus^N( B_{10 r}(x_i)) \le C(N,v) m r^N\, .
	\end{equation}
	It remains only to show that $m\le C(N) r^{1-N} r_a^{N-1}$. In order to do so we use (ii) in \autoref{thm:neckstructure} which gives
	\begin{equation}
	c r_a^{N-1} \ge \mu(B_{2r_a}(x_a)) \ge \sum_{i=1}^m \mu(B_r(x_i)) \ge c^{-1} m r^{N-1}\, , 
	\end{equation}
	with $\mu$ packing measure associated to the neck region as in \eqref{eq:packingmeasure}.

	\textbf{Step 2.} We claim that, for any $\eps<\eps(N)$, it holds
	\begin{equation}\label{eq:tubestbigsymmball}
	\haus^N(B_r(\mathcal{S}^{N-1}\setminus \mathcal{S}^{N-2}) \cap  B_{r_b}(x_b)) \le C(N,v) r r_b^{N-1}\quad \text{for any $r\in (0,1)$}\, ,
	\end{equation}
	whenever $B_{2r_b}(x_b)$ is an $(N,\eps)$-symmetric ball.
	
	Let us choose $\eps(N)$ small enough to ensure that 
	\[
	(\mathcal{S}^{N-1}\setminus \mathcal{S}^{N-2}) \cap B_{\frac{3}{2} r_b}(x_b) = \emptyset
	\]
	whenever $B_{2r_b}(x_b)$ is $(N,\eps)$-symmetric for some $\eps\le \eps(N)$. This choice gives the implication
	\begin{equation}
	B_r(\mathcal{S}^{N-1}\setminus \mathcal{S}^{N-2}) \cap  B_{r_b}(x_b)\neq \emptyset  \implies r_b\le 2 r
	\end{equation}
	that easily leads to \eqref{eq:tubestbigsymmball}.

	\textbf{Step 3.} Let us conclude the proof of \eqref{eq:tubneighbound} relying on \autoref{thm:neckdecompositiontheorem} and the previous two steps. Let $\eps(N)>0$ be smaller than the ones in Step 1 and Step 2, and let $\delta\le \delta(N,v,\eps(N))$, smaller than the one in Step 2 and in \autoref{thm:neckdecompositiontheorem}. By \autoref{thm:neckdecompositiontheorem} we can find a covering 
	\begin{equation}\label{eq:decreminder}
	B_1(p)\subset\bigcup_a\left(\mathcal{N}_a\cap B_{r_a}(x_a)\right)\cup\bigcup_bB_{r_b}(x_b)\cup\mathcal{S}^{\delta,\eta}\, ,
	\end{equation}
	where any $B_{2r_b}(x_b)$ is an $(N,\eps(N))$-symmetric ball, $\mathcal{N}_a=B_{2r_a}\setminus \overline B_{r_x}(\mathcal{C}_a)$ is an $(\eps(N), \delta)$ neck region and $\mathcal{S}^{\delta,\eta}$ is $\haus^N$-negligible.
	Then we can estimate
	\begin{align*}
	\haus^N  (B_r(\mathcal{S}^{N-1}\setminus \mathcal{S}^{N-2})\cap B_1(p))
	 \le&
	\sum_a \haus^N(B_r(\mathcal{S}^{N-1}\setminus \mathcal{S}^{N-2}) \cap \mathcal{N}_a\cap B_{r_a}(x_a))\\
	&+ \sum_b \haus^N(B_r(\mathcal{S}^{N-1}\setminus \mathcal{S}^{N-2}) \cap  B_{r_b}(x_b))
	\\\le & C(N,v) r\left( \sum_a r_a^{N-1} + \sum_b r_b^{N-1} \right) 
	\\\le & C(N,v) r\, ,
	\end{align*}
	 where the first inequality follows from \eqref{eq:decreminder}, the second one from Step 1 and Step 2 and the last one from (iii) in \autoref{thm:neckdecompositiontheorem}.
\end{proof}

\begin{proof}[Proof of \autoref{thm:struc} (iii)]
	When $x\in \mathcal{S}^{N-1}\setminus \mathcal{S}^{N-2}$ any tangent cone has the density at the tip equal to $\Theta_X(x)=1/2$ since $(\setR_+^N,\dist_{\text{Eucl}},\haus^N,0)\in \Tan_x(X,\dist,\haus^N)$. Hence, \autoref{lemma:conerigidity} below implies that
	if $(Y,\varrho,\haus^N,y)\in \Tan_x(X,\dist,\haus^N)$ then, either $Y=\setR_+^{N}$ or it has no boundary according to \autoref{def:withoutboundary}.\\	
	This along with a classical result (see for instance \cite[Theorem 4.2]{CheegerJiangNaber18}) ensuring that the set of tangent cones at given point $x\in X$ is connected with respect to the pmGH topology, implies the sought conclusion. Indeed the set of pointed $\RCD(K,N)$ spaces without boundary is closed with respect to noncollapsed pGH convergence by \autoref{thm:stabimp}.	
\end{proof}

\begin{remark}\label{rm:bdrytobdryhalf}
It follows from the lower semicontinuity of the density $\Theta$ and the observation that $\Theta(x)=1/2$ for any $x\in\mathcal{S}^{N-1}\setminus\mathcal{S}^{N-2}$, that for any noncollapsed $\RCD(K,N)$ m.m.s. $(X,\dist,\haus^N)$, it holds that $\Theta(x)\le 1/2$ for any $x\in\partial X$.	
\end{remark}

\begin{lemma}\label{lemma:conerigidity}
	Let $C(Y)$ be a noncollapsed $\RCD(0,N)$ m.m.s. which is a cone over an $\RCD(N-2,N-1)$ m.m.s. $(Y,\dist_Y,\haus^{N-1})$ with tip $p$. If $C(Y)$  has boundary according to \autoref{def:withoutboundary} then $\Theta(p)\le \frac{1}{2}$. Moreover, the equality holds if and only if $C(Y)$ is isometric to the Euclidean half-space $\setR_+^N$. 
\end{lemma}

\begin{proof}
	
	It is simple to verify that if $C(Y)$ has boundary then $Y$ has boundary as well. Observe that $(Y,\dist_Y,\haus^{N-1})$ is an $\RCD(N-2,N-1)$ space, therefore by \cite{Ketterer15} it has diameter less than $\pi$. Let $y\in \mathcal{S}^{N-2}\setminus \mathcal{S}^{N-3}(Y)$. Then the Bishop-Gromov inequality for the $\CD(N-2,N-1)$ condition ensures that
	\begin{equation}\label{eq:densbound}
	\haus^{N-1}(Y) = \haus^{N-1}(B_{\pi}(y)) \le \Theta_Y(y) N\omega_N \le  \frac{1}{2} N \omega_N\, .
	\end{equation}
	Therefore
	\begin{equation}
	\Theta_{C(Y)}(p) = \lim_{r\to 0} \frac{\haus^N(B_r(p))}{\omega_N r^N}
	=\frac{\haus^{N-1}(Y)}{N \omega_N} \le \frac{1}{2}\, ,
	\end{equation}
	where the second equality follows from the definition of metric measure cone while the inequality follows from \eqref{eq:densbound} (cf. with \eqref{eq:volcross}).
	
	Let us now deal with the equality case. Assume that $\Theta(p)=\frac{1}{2}$. We claim that $C(Y)$ is a cone with respect to any $x\in \mathcal{S}^{N-1}\setminus \mathcal{S}^{N-2}$.
	Indeed, for any $x\in C(Y)$, it holds
	\begin{equation}\label{eq:densityatinfinity}
	\lim_{r\to\infty}\frac{\mathcal{H}^N(B_r(x))}{\omega_Nr^N}
	= \lim_{r\to\infty}\frac{\mathcal{H}^N(B_r(p))}{\omega_Nr^N}
	=1/2\, ,
	\end{equation}
	since
	\begin{align*}
	\lim_{r\to\infty}\frac{\haus^N(B_r(x))}{\omega_Nr^N}=&\lim_{r\to\infty}\frac{\haus^N(B_{r+R}(x))}{\omega_N(r+R)^N}\\
	\ge&\lim_{r\to\infty}\frac{\haus^N(B_r(p))}{\omega_Nr^N}\cdot\frac{r^N}{(r+R)^N}\\
	=&\lim_{r\to\infty}\frac{\haus^N(B_r(p))}{\omega_Nr^N}\, ,
	\end{align*}
where we set $R:=\dist(x,p)$ and the converse inequality can be obtained switching the roles of $x$ and $p$.

Therefore, if we additionally assume that $x\in\mathcal{S}^{N-1}\setminus\mathcal{S}^{N-2}$, then
	\begin{equation}
	r\mapsto\frac{\mathcal{H}^N(B_r(x))}{\omega_Nr^N}
	\quad \text{is constant on $(0,\infty)$}\, .
	\end{equation}
	The volume cone implies metric cone theorem \cite{DePhilippisGigli16} (see also \cite{CheegerColding96} for the previously considered case of Ricci limits) gives then the claimed conclusion.

	Arguing inductively and relying on the cone splitting theorem we can now conclude that $C(Y)=\setR_+^N$. 	
\end{proof}

    \subsection{A second notion of boundary and further regularity properties}
	Recall that our working definition of boundary $\partial X$, taken from \cite{DePhilippisGigli18}, is as topological closure of the top dimensional singular stratum:
	\begin{equation}
	\partial X:=\overline{\mathcal{S}^{N-1}\setminus\mathcal{S}^{N-2}}\, .
	\end{equation}
	
	In \cite{KapovitchMondino19} an alternative definition of boundary has been proposed, inspired by the one adopted for Alexandrov spaces \cite{Perelman91}:
	\begin{equation}\label{eq:bdkm}
	\mathcal{F}X:=\left\lbrace x\in X: \text{ there exists a cone with boundary  $(Y,\dist_Y)\in\Tan_x(X,\dist,\mathcal{H}^N)$} \right\rbrace, 
	\end{equation}
	where cones with boundary are cones for which the cross section, that is a noncollapsed $\RCD(N-2,N-1)$ space thanks to \cite{DePhilippisGigli18,Ketterer15}, has boundary. Arguing recursively we reduce to $\RCD(0,1)$ spaces and, thanks to the classification in \cite{KitabeppuLakzian16}, we know that they are isometric to manifolds of dimension one, possibly with boundary (in the topological sense).  In this case we say that the space has boundary if and only if it is a manifold with boundary.

\begin{theorem}\label{thm:lipbdry}
	Let $1\le N<\infty$ and $v>0$ be fixed. Let $(X,\dist,\haus^N)$ be an $\RCD(-(N-1),N)$ space  and $p\in X$ be such that $\haus^{N}(B_1(p))>v$. Then, either $(\mathcal{S}^{N-1}\setminus \mathcal{S}^{N-1})\cap B_2(p)=\emptyset$ 
	or the following hold:
	\begin{itemize}
		\item[(i)] $\mathcal{F} X \neq \emptyset$ and $\mathcal{F} X\subset \partial X$;
		\item[(ii)] $\partial X$ is $(N-1)$-rectifiable and
		\[
		\haus^{N-1}(\partial X\cap B_r(x))\le C(N,v) r^{N-1}\quad \text{for any $x\in \partial X\cap B_1(p)$ and $r\in (0,1)$}\, ;
		\]
		\item[(iii)]
		\begin{equation}
		\haus^N(B_r(\partial X)\cap B_1(p)) \le C(N,v) r
		\quad \text{for any $r\in (0,1)$, $p\in X$}\, ,
		\end{equation} 
	\end{itemize}
\end{theorem}

\begin{proof}
	Let us begin by proving that 
	\begin{equation}\label{zzz4}
	\mathcal{F} X\neq \emptyset \implies 
	 \mathcal{S}^{N-1}\setminus\mathcal{S}^{N-2}\neq \emptyset\, ,
	\end{equation}
	by induction on $N$. The case $N=1$ is trivial, thanks to the classification of $\RCD(0,1)$ spaces \cite{KitabeppuLakzian16}. Let us deal with the inductive step. Given a noncollapsed $\RCD(K,N)$ m.m.s. $(X,\dist,\haus^N)$ and $x\in \mathcal{F}X$ there exists a cone $(C(Y),\varrho,\haus^N,y)\in \Tan_x(X,\dist,\meas)$ where the cross section $(Y,\dist_Y,\haus^{N-1})$ is an $\RCD(N-2,N-1)$ space such that $\mathcal{F} Y \neq \emptyset$. The inductive assumption gives
	\[
	 \mathcal{S}^{N-1}\setminus\mathcal{S}^{N-2}(Y)\neq \emptyset\, ,
	\]
	which easily yields the claimed conclusion.
	
	Let us now prove the inclusion $\mathcal{F} X \subset \partial X$.
	Being $\partial X$ closed, for any $x\in X\setminus \partial X$ there exists $r>0$ such that 
	\[
	B_r(x)\cap (\mathcal{S}^{N-1}\setminus \mathcal{S}^{N-2}) = \emptyset\, .
	\]
	Therefore any tangent cone $(Y,\varrho, \haus^N,y)$ at $x$ satisfies $\mathcal{S}^{N-1}\setminus \mathcal{S}^{N-2}=\emptyset$ as a consequence of \autoref{thm:stabilityS}. Hence, from \eqref{zzz4} we deduce $\mathcal{F} Y=\emptyset$ which, by definition, yields $x\notin \mathcal{F} X$.
	
	The rectifiability and the measure estimate in (ii) follow from \autoref{thm:struc} (i) and \eqref{eq:measurebound} respectively, taking into account the dimension estimate $\dim{\mathcal{S}^{N-2}}\le N-2$. The volume bound for the tubular neighbourhood is a consequence of \eqref{eq:tubneighbound} and the very definition of $\partial X$. 
\end{proof}

\begin{remark}
Thanks to \autoref{thm:lipbdry} (i), the notion of having boundary for a noncollapsed $\RCD$ space is independent of the definition of boundary we choose, between the ones in \cite{DePhilippisGigli18} and \cite{KapovitchMondino19}. This gives a positive answer to \cite[Question 4.8]{KapovitchMondino19}.
\end{remark}

\begin{remark}
	With the techniques of this paper we are not able to show the identity $\mathcal{F}X = \partial X$ in full generality, which would give a positive answer to \cite[Question 4.9]{KapovitchMondino19}. 
	Nevertheless the analysis of the Laplacian of the distance from the boundary performed in \autoref{sec:distance from the boundary} allows us to prove this identity, together with the local Ahlfors regularity of the boundary volume
	measure, in the case of Ricci limits with boundary. Moreover, the improved neck structure \autoref{thm:improvedneckregion} gives the same conclusion on $\delta$-boundary balls whenever $\delta < \delta(N)$.
\end{remark}

	\begin{corollary}\label{cor:hauseprehau}
Let $1\le N<\infty$ be a fixed natural number. Then, for any $v>0$ there exists a constant $C=C(N,v)>0$ such that the following holds. If $(X,\dist,\haus^N)$ is an $\RCD(-(N-1),N)$ space and $x\in X$ is such that $\haus^N(B_1(x))>v$, then 
\begin{equation}\label{eq:precfnonpre}
\haus^{N-1}_{\infty}\res (\partial X\cap B_1(x))\le \haus^{N-1}\res (\partial X\cap B_1(x))\le C(n,v)\haus^{N-1}_{\infty}\res (\partial X\cap B_1(x))\, .
\end{equation} 
		
	\end{corollary}
	
	\begin{proof}
The first inequality above is true in great generality by the very definition of the Hausdorff and pre-Hausdorff measures. 

Let us pass to the verification of the second one.\\
In order to do so let $C(N,v)$ be such that $\haus^{N-1}(\partial X\cap B_r(x))\le C(N,v) r^{N-1}$ for any $x\in \partial X\cap B_1(p)$ and $r\in (0,1)$ given by \autoref{thm:boundatystructure} (i). Let $B_{r_i}(x_i)$ any covering of a Borel set $A\subset\partial X\cap B_1(x)$. Then, up to worsening the constant $C(n,v)$ we can estimate 
\begin{equation}
\haus^{N-1}(A)\le \sum_i\haus^{N-1}(A\cap\partial X)\le C(N,v)\sum_i r_i^{N-1}\, .
\end{equation} 		
Passing to the infimum on the family of all coverings of $A$ we get the sought estimate
\begin{equation}
\haus^{N-1}(A)\le C(N,v)\haus^{N-1}_{\infty}(A)\, .
\end{equation}
	\end{proof}

	\begin{corollary}\label{cor:stabwithlowerbound}
	Let $1\le N<\infty$ be a fixed natural number and $v>0$, then the following holds. Assume that $(X_n,\dist_n,\haus^N,p_n)$ are noncollapsed $\RCD(-(N-1),N)$ spaces converging in the pGH topology to $(X,\dist,\haus^N,p)$ and verifying the noncollapsing assumption $\haus^N(B_1(p_n))>v$ for any $n\in\setN$. 
Then 
\begin{equation}\label{eq:distusc}
\haus^{N-1}(\partial X\cap \overline{B}_1(p))\ge \frac{1}{C(N,v)}\limsup_{n\to\infty}\haus^{N-1}(\partial X_n\cap \overline{B}_1(p_n))\, ,
\end{equation}
where $C(N,v)>0$ is the constant appearing in \autoref{cor:hauseprehau} above.

	\end{corollary}	
	
	\begin{proof}
		
	Let us denote by $C\subset X$ the limit of the sequence of compact sets $\partial X_n\cap\overline{B}_1(p_n)$ in the Hausdorff topology, possibly after passing to a subsequence. Here it is understood that the convergence of the ambient spaces is realized in a common proper metric space $(Z,\dist_Z)$. Since, as we already remarked, any boundary point has density less than $1/2$ and the density is lower semicontinuous along pGH converging sequences, we infer that $\Theta_X(x)\le 1/2$ for any $x\in C$. In particular $C\subset\mathcal{S}\cap\overline{B}_1(p)$. Moreover, it easily follows from the Hausdorff dimension estimate $\dim_H(\mathcal{S}^{N-2})\le N-2$ that $\haus^{N-1}_{\infty}(C)\le \haus^{N-1}_{\infty}(\partial X\cap\overline{B}_1(p))$.\\
	Taking into account the general inequality $\haus^{N-1}_{\infty}\le \haus^{N-1}$ and the discussion above, in order to prove \eqref{eq:distusc} it suffices now to observe that
\begin{align*}
\haus^{N-1}_{\infty}(C)\ge& \limsup_{n\to\infty}\haus^{N-1}_{\infty}(\partial X_n\cap\overline{B}_1(p_n))\\
\ge& \frac{1}{C(N,v)}\limsup_{n\to\infty}\haus^{N-1}(\partial X_n\cap\overline{B}_1(p_n))\, ,
\end{align*} 
where the first inequality is a consequence of \eqref{eq:limsuphaus} while the second one follows from \autoref{cor:hauseprehau}.
\end{proof}

\vspace{.4cm}

\section{Distance from the boundary and noncollapsing of boundaries}\label{sec:distance from the boundary}
In this section we are going to study some key properties of the distance function from the boundary. They will be useful to better understand the convergence of boundaries of $\RCD$ spaces under noncollapsing pGH convergence and their topological regularity in the next sections.

Given a noncollapsed $\RCD(-(N-1),N)$ space $(X,\dist,\haus^N)$ with boundary we denote by
\[
\dist_{\partial X} : X \to \setR_+,\;\;\;\dist_{\partial X}(x):=\min_{p\in\partial X}\dist(x,p)
\]
the distance function from $\partial X$.

Let us start with a key lemma regarding convergence of distance functions from the boundary in case the limit space is the model half space.

\begin{lemma}\label{lemma:distancetodistance}
	Let $1\le N<\infty$ be a fixed natural number. For any sequence of pointed $\RCD(-(N-1),N)$ spaces $(X_n,\dist_n,\haus^N,x_n)$ such that $B_8(x_n) \to B_8^{\setR_+^N}(0)$ in the GH-topology one has that
	\begin{equation}\label{eq:boundarytoboundary}
	\partial X_n \cap \overline B_3(p_n) \to \partial \setR_+^N \cap \overline B_3(0)\quad \text{in the Hausdorff sense}.
	\end{equation}
	
	Moreover $\dist_{\partial X_n} \to \dist_{\partial \setR_+^N}$ uniformly and in $W^{1,2}$ on $B_2(0)$.
\end{lemma}

\begin{proof}
    Taking into account \autoref{rm:compvarconv} it is sufficient to prove that the convergence holds in the Kuratowski sense.
	
	Let us first prove that any limit point of a sequence of points $y_n\in\partial X_n\cap \overline B_3(p_n)$ belongs to $\partial\setR^N_+\cap \overline B_3(0)$. To this aim it is sufficient to take into account \autoref{rm:bdrytobdryhalf} and the lower semicontinuity of the density along pGH converging sequences of noncollapsed spaces. We conclude that the limit point has density less than $1/2$ and therefore it belongs to the boundary of the half space, since those are the only singular points.
	
	Next we wish to prove that any point in $\partial\setR^N_+\cap \overline B_3(0)$ is the limit of a sequence of points in $\partial X_n\cap \overline B_3(p_n)$. To prove this claim we rely on the stability of the boundary. If the claim were false then we could find a scale $r>0$ and points $y_n\in X_n$ such that $y_n\to 0\in\setR^N_+$ and $B_r(y_n)\subset X_n$ has no boundary for any $n\in\setN$. The contradiction follows by \autoref{thm:stabilityS}, since the ball $B_r(0)\subset \setR^N_+$ has boundary.
	
	The uniform convergence $\dist_{\partial X_n} \to \dist_{\partial \setR_+^N}$ on $B_2(0)$ is a simple consequence of \eqref{eq:boundarytoboundary} (see again \autoref{rm:compvarconv}). 
	
	To obtain the $W^{1,2}$ convergence it is sufficient to observe that, as pointed out in \cite{AmbrosioHondaTewodrose}, for uniformly continuous functions the uniform and the $L^2$ convergence are equivalent. Moreover, $\abs{\nabla \dist_{\partial X_n}}=\abs{\nabla\dist_{\partial X}}=1$ $\haus^N$ a.e., therefore the $W^{1,2}$ convergence follows from the volume convergence \autoref{thm:volumeconvergence}, since 
		\begin{equation}
		\int_{B_2(p_n)}\abs{\nabla \dist_{\partial X_n}}^2\di\haus^N=\haus^N(B_2(p_n))\to \haus^N(B_2(0))=\int_{B_2(0)}\abs{\nabla \dist_{\partial X}}^2\di\haus^N\, .
		\end{equation}
\end{proof}

Given the stability \autoref{thm:stabilityS} and \autoref{lemma:distancetodistance} we can provide a useful improvement upon the form of the $\eps$-isometry in \autoref{splitting vs isometry} in the case of $\delta$-boundary balls.	

\begin{corollary}\label{lemma:improvedepsio}
	Let $1\le N<\infty$ be a fixed natural number. Then for any $\eps>0$ there exists $\delta=\delta(\eps,N)>0$ such that, if $(X,\dist,\haus^N)$ is a noncollapsed $\RCD(-\delta(n-1),N)$ space and $B_2(x)\subset X$ is a $\delta$-boundary ball, then for any $\eps$-splitting map $u:B_1(x)\to\setR^{N-1}$ such that $u(x)=0$ one has that
	\begin{equation}
	(u,\dist_{\partial X}):B_1(x)\to\setR^{N}_+\;\;\;\text{is an $\eps$-isometry}\, ,
	\end{equation}	
	and 
	\begin{equation}
	\sum_{k=1}^{N-1}\fint_{B_1(p)} |\nabla u_k \cdot \nabla \dist_{\partial X}|\di \haus^N \le \eps\, . 
	\end{equation}
\end{corollary}

\begin{proof}
	Both conclusions can be obtained arguing by contradiction as in the proof of \autoref{splitting vs isometry} and relying on \autoref{lemma:distancetodistance}.
\end{proof}

\subsection{Laplacian of the distance from the boundary}
Next we study the Laplacian of the distance function from the boundary,
which plays a fundamental role in establishing noncollapsing estimates for the boundary measure.

Let us begin by recalling that $\dist_{\partial X}$ has locally measure valued Laplacian $\boldsymbol{\Delta}\dist_{\partial X}$ on $X\setminus\partial X$ as a consequence of the general representation theorem for Laplacians of distance functions \cite[Corollary 4.16]{CavallettiMondino18}. Moreover in \cite[Corollary 4.16]{CavallettiMondino18} it is also proved that the singular part of $\boldsymbol{\Delta}\dist_{\partial X}$ on $X\setminus\partial X$ is non positive. 
The following conjecture regards the absolutely continuous part.

\begin{openquestion}\label{conj:laplasign}
	Let $(X,\dist,\haus^N)$ be a noncollapsed $\RCD(K,N)$ m.m.s. for some $K\in\setR$ and $1\le N<\infty$. Assume that $\partial X\neq\emptyset$. Then 
	\begin{equation}\label{eq:conjlapl}
	\boldsymbol{\Delta}^{\text{ac}}\dist_{\partial X}\le -K\dist_{\partial X}
	\quad \text{on $X\setminus\partial X$}\, .
	\end{equation}	
	where $\boldsymbol{\Delta}^{\text{ac}}\dist_{\partial X}$ denotes the absolutely continuous part of $\boldsymbol{\Delta}\dist_{\partial X}$ on $X\setminus \partial X$.
\end{openquestion}

As we shall see below, \autoref{conj:laplasign} can be verified for Alexandrov spaces with curvature bounded from below, Riemannian manifolds with convex boundary and interior lower Ricci curvature bounds and their noncollapsed pGH limits. 

Let us first present the main analytic and geometric implications of a positive answer to \autoref{conj:laplasign}.

\begin{theorem}\label{thm:consequencesconj}
	Let $N\in \setN$, $N\ge 1$ and $K\in \setR$ be fixed. Given an $\RCD(K,N)$ m.m.s. $(X,\dist,\haus^N)$ with $\partial X\neq 0$ such that \autoref{conj:laplasign} is verified, the following hold:
	\begin{itemize}
		\item[(i)] $\dist_{\partial X}$ has measure valued Laplacian on $X$ and $\boldsymbol{\Delta}\dist_{\partial X} \res \partial X = \haus^{N-1}\res \partial X$;
		\item[(ii)] for any $p\in \partial X$ one has
		\begin{equation}\label{eq:lowerareabound}
		\haus^{N-1}(B_2(p)\cap\partial X) > C(K) \haus^N(B_1(p))\, ;
		\end{equation}
		\item[(iii)] any tangent cone at $x\in \partial X$ has boundary, in particular $\mathcal{F}X = \partial X$ (recall that $\mathcal{F}X$ is defined in \eqref{eq:bdkm}). 
	\end{itemize}
\end{theorem}

	Let us state and prove a lemma that is independent of the validity of \autoref{conj:laplasign} and will play a role in the proof of \autoref{thm:consequencesconj}.

\begin{lemma}\label{lemma:distancebdry}
	Let $1\le N<\infty$ be a natural number and $(X,\dist,\haus^N)$ be a noncollapsed $\RCD(-(N-1),N)$ metric measure space. Assume that $\partial X\neq\emptyset$. Then $\dist_{\partial X}:X\to [0,\infty)$ has locally measure valued Laplacian on $X\setminus\partial X$ and the singular part of $\boldsymbol{\Delta}\dist_{\partial X}$ is non positive on $X\setminus\partial X$. Moreover
		\begin{equation}\label{eq:replapl}
		\int_X\nabla\phi\cdot\nabla\dist_{\partial X}\di\haus^N=-\int\phi\di\mu-\lim_{r_i\to 0}\int_{\{\dist_{\partial X}>r_i\}}\phi\di\boldsymbol{\Delta}\dist_{\partial X}
		\quad \text{for any $\phi \in \Lip_c(X)$}\, ,
		\end{equation}
		for some sequence $r_i\downarrow 0$ and locally finite measure $\mu$ on $X$ (a priori depending on the chosen sequence).
	\end{lemma}

\begin{proof}
	We have already observed that $\boldsymbol{\Delta}\dist_{\partial X}$ has locally measure valued Laplacian on $X\setminus\partial X$ and the singular part of $\boldsymbol{\Delta}\dist_{\partial X}$ is non positive on $X\setminus\partial X$ as a consequence of \cite[Corollary 4.16]{CavallettiMondino18}.

		Let us verify that $\dist_{\partial X}$ verifies \eqref{eq:replapl}. We treat only the case when $(X,\dist)$ (and a fortiori $\partial X$) is compact, the general one can be handled with an additional cut-off argument.
	
		In order to do so we wish to pass to the limit the integration by parts formula on (sufficiently regular) superlevel sets of the distance from the boundary.\\ 
		Observe that, by the coarea formula \autoref{thm:coarea}, for almost every $r>0$, the superlevel set $\{\dist_{\partial X}>r\}$ has finite perimeter. Moreover, the volume bound for the tubular neighbourhood of the boundary 
		\begin{equation}
		\haus^N(\{\dist_{\partial X}<r\})\le Cr\, ,
		\end{equation}
		obtained thanks to \autoref{thm:lipbdry} (iii) via a covering argument,
		together with a further application of the coarea formula, yield the existence of a sequence $(r_i)$ with $r_i\downarrow 0$ as $i\to\infty$ and 
		\begin{equation}\label{eq:boundper}
		\Per (\{\dist_{\partial X}>r_i\})\le C\;\;\;\text{for any $i\in\setN$}\, .
		\end{equation}

		Since $\dist_{\partial X}$ has measure valued Laplacian on $X\setminus\partial X=\{\dist_{\partial X}>0\}$, the bounded vector field $\nabla \dist_{\partial X}$ has measure valued divergence on the same domain. Therefore, applying the Gauss Green theorem \cite[Section 6]{BuffaComiMiranda} to the vector field $\phi\nabla\dist_{\partial X}$ on the domain $\{\dist_{\partial X}>r_i\}$ we infer that
		\begin{equation}\label{eq:approxgg}
		\int_{\{\dist_{\partial X}>r_i\}}\nabla\phi\cdot\nabla\dist_{\partial X}\di\haus^N=-\int_{\{\dist_{\partial X}>r_i\}}\phi\di\boldsymbol{\Delta}\dist_{\partial X}-\int\phi f_i\di\Per(\{\dist_{\partial X}>r_i\})\, ,
		\end{equation} 
		for some function $f_i$ verifying 
		\begin{equation}\label{eq:boundf}
		\norm{f_i}_{L^{\infty}(\Per(\{\dist_{\partial X}>r_i\}))}\le 1\, .
		\end{equation}
		Thanks to \eqref{eq:boundper} and \eqref{eq:boundf} we can assume that, up to extracting a subsequence, the measures $f_i\Per(\{\dist_{\partial X}>r_i\})$ weakly converge to a finite measure $\nu$ on $X$ in duality with continuous functions. Therefore we can pass to the limit in \eqref{eq:approxgg} as $i\to\infty$ to get that 
		\begin{equation}
		\int_X\nabla\phi\cdot\nabla\dist_{\partial X}\di\haus^N=-\int\phi\di\mu-\lim_{r_i\to 0}\int_{\{\dist_{\partial X}>r_i\}}\phi\di\boldsymbol{\Delta}\dist_{\partial X}\, ,
		\end{equation}
		as we claimed.
\end{proof}

\begin{proof}[Proof of \autoref{thm:consequencesconj}]

	Let $\mu$ and $r_i\downarrow 0$ be as in \autoref{lemma:distancebdry}. If \autoref{conj:laplasign} holds then we deduce that
	\begin{equation*}
	\int \nabla\phi\cdot\nabla\dist_{\partial X}\di\haus^N \le -\int\phi\di\mu -K\int\phi\dist_{\partial X}\di\haus^N
	\quad \text{for any $\phi \in \Lip_c(X)$ s.t. $\phi\ge 0$}\, .
	\end{equation*}
	In particular $\phi \mapsto \int \nabla\phi\cdot\nabla\dist_{\partial X}\di\haus^N + \int\phi\di\mu + K\int\phi\dist_{\partial X}\di\haus^N$ is a negative linear map. Hence there exists a nonnegative locally finite measure $\nu$ such that
	\[
	\int \nabla\phi\cdot\nabla\dist_{\partial X}\di\haus^N + \int\phi\di\mu +K\int\phi\dist_{\partial X}\di\haus^N = - \int \phi \di \nu,\quad \text{for any $\phi \in \Lip_c(X)$}\, .
	\]
	This implies that $\dist_{\partial X}$ has measure valued Laplacian on $X$.\\

	Let us now prove that 
	\begin{equation}\label{eq:reslapldist}
	\boldsymbol{\Delta}\dist_{\partial X}\res\partial X = \haus^{N-1}\res \partial X\, .
	\end{equation}

   Observe first that $\boldsymbol{\Delta}\dist_{\partial X}\ll\haus^{N-1}$ as consequence of the following more general observation, applied to $b=\nabla\dist_{\partial X}$: if $b$ is a bounded vector field with measure valued divergence $\boldsymbol{\mathrm{div}}\,b$ on a noncollapsed $\RCD$ space $(X,\dist,\haus^N)$, then $\boldsymbol{\mathrm{div}}\,b\ll\haus^{N-1}$. 
    	
    In order to prove the property above we rely on the integration by parts formula for bounded vector fields with measure valued divergence proved in this framework in \cite{BuffaComiMiranda}, taking into account the bound $\Per(B_r(x))\le C_{K,N}r^{N-1}$ for any $0<r<1$ and following the Euclidean strategy in \cite{NguyenTorres}. Relying on these tools we infer that
    \begin{equation}
    	\abs{\boldsymbol{\mathrm{div}}\,b(B_r(x))}\le C_{K,N}\norm{b}_{\infty}r^{N-1},\;\;\;\text{for any $0<r<1$}\, ,
    \end{equation}		
    which suffices to conclude that $\boldsymbol{\mathrm{div}}\,b\ll\haus^{N-1}$.
    	
    Next observe that $\boldsymbol{\Delta} \dist_{\partial X}\res\partial X$ is absolutely continuous with respect to $\haus^{N-1}\res\partial X$, which is a locally finite measure on an $(N-1)$-rectifiable set by \autoref{thm:lipbdry}. By standard differentiation of measures tools we infer that, in order to prove \eqref{eq:reslapldist}, we need to show that
    	\begin{equation}\label{eq:density}
    	\lim_{r\to 0}\frac{\boldsymbol{\Delta} \dist_{\partial X}(B_r(x))}{\omega_{N-1}r^{N-1}}=1,\;\;\;\text{for $\haus^{N-1}$-a.e. $x\in\partial X$}\, .
    	\end{equation}
    The validity of \eqref{eq:density} can be checked at any regular boundary point $x\in\mathcal{S}^{N-1}\setminus\mathcal{S}^{N-2}$, applying \autoref{lemma:distancetodistance} to the sequence of scaled spaces converging to the tangent half-space. Indeed the $W^{1,2}$ convergence of the distance functions, yields the weak convergence of their Laplacians, which yields in turn \eqref{eq:density} by scaling.
    
    \medskip
    
    We can now pass to the proof of (ii). Let $\phi\in \Lip(X)$ be nonnegative with bounded support. The coarea formula \autoref{thm:coarea} yields 
	\[
     \frac{\di}{\di s} \int_{B_s(\partial X)} \phi \di \haus^N =
	\int \phi \di \Per(\set{\dist_{\partial X}>s})\quad \text{for a.e. $s\ge 0$}\, .
	\]
	Using again the Coarea formula, the $\haus^N$-a.e. identity $|\nabla \dist_{\partial X}|^2=1$ and integrating by parts, we get 
	\begin{equation}\label{z33}
	\int_0^{\infty} a'(s) \int \phi \di \Per(\set{\dist_{\partial X}>s}) = -\int a(\dist_{\partial X}) \div (\phi \nabla \dist_{\partial X}) \di \haus^{N}\, ,
	\end{equation}
	for any $a\in C^{\infty}_c(0,\infty)$. A simple approximation and Lebesgue points argument allows plugging $a(s)=\chi_{\{s\le r\}}$ on \eqref{z33} for almost every $r>0$, yielding
	\[
	\int \phi \di \Per(\set{\dist_{\partial X}> r}) = \int_{\{\dist_{\partial X}<r\}} \div (\phi \nabla \dist_{\partial X}) \di \haus^{N} \quad \text{for a.e. $r\ge 0$}\, .
	\]
	All in all we have
	\begin{align*}
	\frac{\di}{\di s} \int_{B_s(\partial X)} \phi \di \haus^N &=
	\int \phi \di \Per(\set{\dist_{\partial X}>s}) 
	\\& = \int_{B_s(\partial X)} \div(\phi \nabla \dist_{\partial X}) \di \haus^N
	\\& \le \int_{B_s(\partial X)} |\nabla \phi| \di \haus^N + \int_{B_s(\partial X)} \phi \di \boldsymbol{\Delta}\dist_{\partial X}
	\\& \overset{\eqref{eq:conjlapl}, (i)}{\le} \int_{B_s(\partial X)} |\nabla \phi| \di \haus^N + \int_{\partial X} \phi \di \haus^{N-1} + K^- s \int_{B_s(\partial X)} \phi \di \haus^N\, ,
	\end{align*}
	for a.e. $s\ge 0$. 
	
	Let $t\in (0,1)$. An additional approximation argument allows to plug $\phi= \chi_{B_t}(p)$ in the previous inequality. Setting $f(s,t):= \haus^N(B_s(\partial X)\cap B_t(p))$, we then infer that $s\mapsto f(s,t)$ is absolutely continuous and
	\begin{equation}\label{z31}
	\frac{\di}{\di s} f(s,t) \le  \haus^{N-1}(\partial X\cap B_t(p)) +  \Per(B_t(p), B_s(\partial X)) + K^- s f(s,t) \le C(N,K)\, ,
	\end{equation}
	for a.e. $s\in (0,1)$. This yields in turn
	\begin{equation}
		\frac{\di}{\di s} \int_0^t f(s,r)\di r \le  t \haus^{N-1}(\partial X\cap B_t(p)) +  f(s,t) + K^- s \int_0^t f(s,r) \di r\, ,
	\end{equation}
	for a.e. $s\in (0,1)$,
	thanks to the coarea formula and the bound
	\begin{equation}\label{z32}
		|f(s,t) - f(s',t)|\le |s-s'| C(N,K) \quad \text{for any $s,s',t\in (0,1)$}\, .
	\end{equation}	

   By using \eqref{z31} and \eqref{z32} it is simple to verify that $s\mapsto \int_0^{1-s}f(s,1-r)\di r$ is absolutely continuous in $(0,1)$. We now prove that
	\begin{equation}\label{z30}
	\frac{\di}{\di s} \int_0^{1-s}f(s,1-r)\di r \le \haus^{N-1}(\partial X\cap B_{1-s}(p)) + s K^{-} \int_0^{1-s}f(s,1-r)\di r\, ,
	\end{equation}
	for a.e. $s\in (0,1)$.\\
    For any $t\in (0,1)$ we denote by $I_t\subset [0,1]$ the set of $s\in (0,1)$ such that \eqref{z31} holds true. Given $s\in \cap_{t\in \setQ\cap (0,1)} I_t=:I$ and $\eps>0$ we consider $q\in \setQ$ such that $s<q<s+\eps$. Then we have
	\begin{equation}
	\begin{split}
		 \int_0^{1-s} &  \frac{f(s+h,1-r)-f(s,1-r)}{h}\di r \\ 
		\le & \int_0^{1-q} \frac{f(s+h,1-r)-f(s,1-r)}{h}\di r + \eps C(N,K)\, ,
	\end{split}
	\end{equation}
	for any $0<h<1$ small enough, as a consequence of \eqref{z31}. Therefore, using the fact that $s<q$, we get
	\begin{equation*}
	\begin{split}
	\limsup_{h\to 0}  &\int_0^{1-s}  \frac{f(s+h,1-r)-f(s,1-r)}{h}\di r
	 \\
	\le& (1-q) \haus^{N-1}(\partial X\cap B_{1-q}(p)) +  f(s,1-q)\\ 
	&+ K^- s \int_0^{1-q} f(s,1-r) \di r + \eps C(N,K)
	\\
	\le &(1-s) \haus^{N-1}(\partial X\cap B_{1-s}(p)) +  f(s,1-s)\\ 
	&+ K^- s \int_0^{1-s} f(s,1-r) \di r + \eps C(N,K)\, ,
	\end{split}
	\end{equation*}
	for any $s\in I$. Letting $\eps\to 0$ we conclude
	\begin{equation}\label{z35}
    \begin{split}
    \limsup_{h\to 0}  &\int_0^{1-s}  \frac{f(s+h,1-r)-f(s,1-r)}{h}\di r
   \\&
   \le (1-s) \haus^{N-1}(\partial X\cap B_{1-s}(p)) +  f(s,1-s) + K^- s \int_0^{1-s} f(s,1-r) \di r\, ,
    \end{split}
	\end{equation}
	for any $s\in I$.

	Using once more \eqref{z32} we easily deduce
	\begin{equation}
	\lim_{h\to 0}-\frac{1}{h}\int_{1-s-h}^{1-s} f(s+h,1-r)\di r = - f(s,1-s)\, , 
	\end{equation}
	which along with \eqref{z35} gives \eqref{z30} for any $s\in I$ such that the derivative $\frac{\di}{\di s} \int_0^{1-s}f(s,1-r)\di r$ exists.

	We can finally conclude the proof of \eqref{eq:lowerareabound} by integrating \eqref{z30} in $(0,1/2)$:
	\[
	\frac{1}{2} e^{\frac{1}{8}K^{-}}\haus^N(B_{1/2}(p))
	\le \int_0^{1/2} e^{\frac{1}{8}K^{-}} \haus^N(B_{1/2}(\partial X)\cap B_{1/2}(p))
	\le \frac{1}{2} \haus^{N-1}(\partial X\cap B_{1}(p)).
	\]

   \medskip
   
   Let us now address (iii). This assertion can be obtained by combining \autoref{cor:stabwithlowerbound} and the inequality
      \begin{equation}\label{eq:scalinginvariantnoncollapsboundary}
     \liminf_{r\to 0} \frac{\haus^{N-1}(B_r(x)\cap \partial X)}{\omega_{N-1} r^{N-1}} \ge C(K) \Theta_X(x)\quad \text{for any $x\in \partial X$}
      \end{equation}
   which follows in turn from  the scaling invariant version of \eqref{eq:lowerareabound}. 
\end{proof}

\subsection{Alexandrov spaces and noncollapsed Ricci limits with boundary}

We are able to verify \autoref{conj:laplasign} in the setting of Alexandrov spaces with curvature bounded below and in the case of Ricci limits with boundary.
\medskip
	
Let us recall that an $N$-dimensional Alexandrov space with curvature bounded from below by $k$ and equipped with the measure $\haus^N$ is a noncollapsed $\RCD(k(N-1),N)$ space, see \cite{Petrunin,ZhangZhu10}.

\begin{proposition}\label{prop:laplalex}
Let $(X,\dist,\haus^N)$ be an Alexandrov space with curvature bounded from below by $k$ and assume that $\partial X\neq\emptyset$. Then $\dist_{\partial X}$ has locally measure valued Laplacian,
\begin{equation}\label{eq:intbdlaplaalex}
\boldsymbol{\Delta}^{ac}\dist_{\partial X}\le -k(N-1)\dist_{\partial X}\;\;\text{and}\;\;
\boldsymbol{\Delta}^{s}\dist_{\partial X}\le 0
\quad
\text{on $X\setminus\partial X$}\, .
\end{equation}	
In particular, there exists a constant $C(k,N)>0$ such that the following holds: if $p\in \partial X$, then
\begin{equation}\label{eq:lowerareaboundalex}
\haus^{N-1}(B_2(p)\cap\partial X) > C(k,N) \haus^N(B_1(p))\, .\footnote{Let us point out that this lower bound for the boundary volume seems to be not present in the literature, although it is pointed out, without proof and in the compact case, by Petrunin in a discussion on MathOverflow.}
\end{equation}
\end{proposition}

Next we deal with the case of $\RCD$ spaces with boundary that are also smooth Riemannian manifolds, see \cite{Han17}. This is a key tool in order to address the case of limits of Riemannian manifolds with boundary later. Let us point out that bounds for the Laplacian on the distance from the boundary in the sense of barriers and under different assumptions have been considered in \cite{Perales16} (see also the references therein).

	\begin{proposition}\label{prop:laplnonriem}
		Let $(X,\dist,\haus^N)$ be a smooth $N$-dimensional Riemannian manifold with convex boundary $\partial X$ and Ricci curvature bounded from below by $K$ in the interior. Then $\dist_{\partial X}$ has locally measure valued Laplacian,
		\begin{equation}\label{eq:laplacompriem}
		\boldsymbol{\Delta}^{ac}\dist_{\partial X}\le - K \dist_{\partial X}\;\;\text{and}\;\;
		\boldsymbol{\Delta}^{s}\dist_{\partial X}\le 0
		\quad
		\text{on $X\setminus\partial X$}\, .
		\end{equation}	
		In particular
		\begin{equation}\label{eq:lowerareaboundriema}
		\haus^{N-1}(B_2(p)\cap\partial X) > C(K) \haus^N(B_1(p))\quad \text{ for any $p\in \partial X$}\, .
		\end{equation}
	\end{proposition}

The lower bound for the volume of the boundary in the case of smooth $\RCD$ spaces with boundary allows us to get a more complete picture about their pGH limits.
The last result of this section provides, in particular, a positive answer to \cite[Questions 4.4, 4.7, 4.9]{KapovitchMondino19} in this setting.

	\begin{theorem}\label{thm:riccilimits}
		Let $(X,\dist,\haus^N,p)$ be the noncollapsed pGH limit of a limit of smooth $N$-dimensional Riemannian manifolds $(X_n,\dist_n)$ with convex boundary and Ricci curvature bounded from below by $K$ in the interior. Assume that there exists $R>0$ such that $B_R(p_n)\cap \partial X_n\neq\emptyset$ for any $n\in\setN$. Then
		\begin{itemize}
			\item[i)]  $\partial X\neq\emptyset$. Moreover, if points $x_n\in \partial X_n$ converge to $x\in X$, then $x\in\partial X$;
			\item[ii)] $\dist_{\partial X}$ has measure valued Laplacian,
			\begin{equation}\label{eq:laplbound}
			\boldsymbol{\Delta}^{ac}\dist_{\partial X}\le - K \dist_{\partial X}\;\;\text{and}\;\;
			\boldsymbol{\Delta}^{s}\dist_{\partial X}\le 0
			\quad
			\text{on $X\setminus\partial X$.}
			\end{equation}
			and
			\begin{equation}\label{eq:laplarep}
			\boldsymbol{\Delta}\dist_{\partial X}\res\partial X=\haus^{N-1}\res\partial X;
			\end{equation}
			\item[iii)] $\partial X=\mathcal{F}X$, $\haus^{N-1}\res\partial X$ is locally Ahlfors regular and for any $x\in\partial X$ any tangent cone at $x$ has boundary.
		\end{itemize}	
	\end{theorem}	

The remaining part of this subsection is devoted to the proof of \autoref{prop:laplalex}, \autoref{prop:laplnonriem} and \autoref{thm:riccilimits}.

\begin{proof}[Proof of \autoref{prop:laplalex}]
We avoid introducing all the relevant background about calculus on Alexandrov spaces, since this is not the main topic of the paper. We refer to \cite{AmbrosioBertrand18,BuragoBuragoIvanov01} for the relevant notions and references.
	
In \cite{AlexanderBishop} it is proved that on any Alexandrov space with curvature bounded from below by $k$ and non empty boundary, the distance function from the boundary is $\mathcal{F}k$-concave, that is to say its restriction to any unit speed geodesic $\gamma:[0,1]\to X$ verifies
\begin{equation}\label{eq:barrier}
\left(\dist_{\partial X}\circ\gamma\right)''+k\dist_{\partial X}\circ\gamma\le 0
\end{equation}		
in the sense of barriers.
We already know that the singular part of $\boldsymbol{\Delta}\dist_{\partial X}$ is non positive on $X\setminus\partial X$ (see \autoref{lemma:distancebdry}). It is then sufficient to prove that its absolutely continuous part verifies 
\begin{equation}\label{z34}
\boldsymbol{\Delta}^{ac}\dist_{\partial X}\le - K \dist_{\partial X}
\quad
\text{on $X\setminus\partial X$.}
\end{equation}	

Combining \eqref{eq:barrier} with Alexandrov's theorem (see \cite[Proposition 7.5]{AmbrosioBertrand18} for its proof in the setting of DC functions on Alexandrov spaces) and the fact that the restriction of $\dist_{\partial X}$ to a minimizing geodesic connecting a regular point to a point of minimal distance on the boundary is affine, we infer that 
\begin{equation}\label{eq:hessac}
\mathrm{tr}\Hess^{\rm{ac}}\dist_{\partial X}\le- k(N-1)\dist_{\partial X},\;\;\;\text{$\haus^N$-a.e. on $X$}\, ,
\end{equation}
where $\Hess^{\rm{ac}}\dist_{\partial X}$ denotes the absolutely continuous part of the Hessian of $\dist_{\partial X}$.  Indeed, we recall that the tangent cone on an Alexandrov space can be equivalently characterised as the set of geodesic directions. Then \eqref{eq:hessac} follows by tracing, since there is a direction $v$ along which $\Hess^{\rm{ac}}\dist_{\partial X}(v,v)=0$, while along all the others the estimate $\Hess^{\rm{ac}}\dist_{\partial X}(v,v)\le -k\dist_{\partial X}$ holds by \eqref{eq:barrier}.

The conclusion \eqref{z34} follows from the fact that the Laplacian is the trace of the Hessian in this context, see \cite[Proposition 5.8, 5.9]{AmbrosioBertrand18}.

The second conclusion of \autoref{prop:laplalex} follows from \autoref{thm:consequencesconj}.

\end{proof}

\begin{proof}[Proof of \autoref{prop:laplnonriem}]
Observe that, as pointed out in \cite{Han17}, smooth Riemannian manifolds with convex boundary and Ricci curvature bounded from below in the interior are $\RCD$ spaces.

Also in this case it is sufficient to deal only with the absolutely continuous part of the Laplacian. The bound \eqref{eq:laplacompriem} can be obtained starting from the following observation: since the boundary is smooth and convex it has non positive second fundamental form. Therefore by tracing we infer that it has nonpositive mean curvature. By smoothness again we infer that the restriction of the a.c. part of the Laplacian of the distance from the boundary on the boundary is non positive (it coincides with the mean curvature). By the localization technique \cite{CavallettiMondino18} we can then propagate the non positivity of the Laplacian on the boundary to the interior to obtain \eqref{eq:laplacompriem}.\\ 
More in detail, we recall that $\dist_{\partial X}$ induces a decomposition of the Riemannian manifold into \emph{rays} $(X_{\alpha},\meas_{\alpha},\dist)$ and an associated disintegration of the volume measure:
\begin{equation}
\int_M\phi\di\, \haus^N=\int_{\mathcal{Q}}\int_{X_{\alpha}}\phi\di\, \meas_{\alpha}\di\, \mathrm{q}(\alpha)\, ,
\end{equation} 
where $\mathcal{Q}$ is a set parametrizing the rays in the decomposition and $\mathrm{q}$ is a suitable weight on $\mathcal{Q}$.\\
Moreover, for $\mathrm{q}$-a.e. ray $(X_{\alpha})$ the one dimensional metric measure space $(X_{\alpha},\meas_{\alpha},\dist)$ is a $\CD(K,N)$ space, hence $\meas_{\alpha}=h_{\alpha}\haus^1$ for some density $h_{\alpha}$ which is $\log-K$ concave, once we parametrize the ray $X_{\alpha}$ with a geodesic $\gamma_{\alpha}$ such that its ending point belongs to $\partial X$.

By \cite{CavallettiMondino18} we also know that 
\begin{equation}\label{eq:reprlaploc}
\boldsymbol{\Delta}^{ac}\dist_{\partial X}=-\left(\log h_{\alpha}\right)'\, ,
\end{equation}
with canonical identifications. The sought conclusion follows observing that since $h_{\alpha}$ is $\log-K$ concave, it holds that
\begin{equation}\label{eq:Kmon}
(-\log h_{\alpha})'(s)\ge (-\log h_{\alpha})'(t)+K(s-t)\, .
\end{equation}  
Let $p=\gamma_{\alpha}(t)$ and $\gamma_{\alpha}(s)$ is the footpoint on $\partial X$ of the minimizing geodesic for the distance to the boundary $\gamma_{\alpha}$.
Since we already pointed out that the convexity of the boundary yields $(-\log h_{\alpha})'(\gamma_{\alpha}(s))\le 0$ when $\gamma_{\alpha}(s)\in\partial X$, we infer from \eqref{eq:Kmon} and \eqref{eq:reprlaploc} that, 
\begin{align}
(\boldsymbol{\Delta}^{ac}\dist_{\partial X})(p)&=(-\log h_{\alpha})'(s-\dist_{\partial X}(\gamma_{\alpha}(t)))\\
&\le (-\log h_{\alpha})'(s)-K\dist_{\partial X}(\gamma_{\alpha}(t))\le -K\dist_{\partial X}(\gamma_{\alpha}(t))=-K\dist_{\partial X}(p)\, .
\end{align}

The inequality \eqref{eq:lowerareaboundriema} and the fact that $\dist_{\partial X}$ has measure valued Laplacian globally now follow from \autoref{thm:consequencesconj}.
\end{proof}	

\begin{proof}[Proof \autoref{thm:riccilimits}]
Let us start by proving the second part of (ii). The first part will directly follow. In order to do so we observe that, if $x_n\in\partial X_n$, then by (the scale invariant version of) \eqref{eq:lowerareaboundriema} we infer that
\begin{equation}
\haus^{N-1}(\partial X_n\cap B_r(x_n))>C(K,N)\haus^N(B_1(x_n))r^{N-1}\;\;\;\text{for $n\in\setN$ and $0<r<1$}\, .
\end{equation}
Therefore we can apply \autoref{cor:stabwithlowerbound} to get that
\begin{equation}\label{eq:scalelow}
\haus^{N-1}(\partial X\cap B_r(x))\ge C(K,N,v)r^{N-1}\;\;\;\text{for any $0<r<1$}\, ,
\end{equation}
where $v>0$ is a noncollapsing bound for $\haus^N(B_1(x_n))$. From \eqref{eq:scalelow} we infer in particular that $x\in\partial X$.

By the stability \autoref{thm:stabilityboundary} we know that if $x\in\partial X$ then there exists a sequence $\partial X_n\ni x_n\to x$, as $n\to\infty$. Therefore $\dist_{\partial X_n}$ converge locally uniformly and locally in $W^{1,2}$ to $\dist_{\partial X}$. Hence we can pass to the limit their measure valued Laplacians and the bounds obtained in \autoref{prop:laplnonriem} to infer that $\dist_{\partial X}$ has measure valued Laplacian satisfying the bound \eqref{eq:laplbound}.

The remaining conclusions follow immediately from \autoref{thm:consequencesconj}.
\end{proof}

\section{Improved neck structure and boundary measure convergence}\label{sec:improved neck structure}
 
 In this section we are going to improve upon the regularity of balls sufficiently close to the model boundary ball on the half-space. This will provide a key tool in order to obtain topological regularity of boundaries, away from sets of ambient codimension two, and convergence of the boundary measures under noncollapsing pGH convergence as stated in \autoref{intro:stabilityboundary}.

\subsection{Improved neck structure theorem and Boundary volume rigidity}

Below we state the key result we will rely on. As in the case of $(N,\delta)$-symmetric balls, where regularity propagates at all locations and scales, in the case of a $\delta$-boundary ball we shall see that balls centered at boundary points are still $\delta'$-boundary balls at any scale, while balls centered at interior points become $(N,\delta')$-symmetric at sufficiently small scales.

\begin{theorem}\label{thm:improvedneckregion}
	Let $1\le N<\infty$ be a fixed natural number. For any $0<\eps<\eps(N)$, $\delta\le \delta(N,\eps)$ and for any $\RCD(-\delta(N-1),N)$ m.m.s. $(X,\dist,\haus^N)$, $p\in X$ such that $B_2(p)$ is a $\delta$-boundary ball, the following properties hold:
	\begin{itemize}
		\item[(i)] for any $x\in \partial X\cap B_1(p)$ and any $r\in (0,1)$ there exists an $\eps r$-GH isometry
		\[
		\Psi_{x,r}:B_r^{\setR_+^N}(0) \to B_r(x)\, ;
		\]
		\item[(ii)] for any $x\in\partial X\cap B_1(p)$ and for any $0<r<1$, setting $\mathcal{L}_{x,r}:= \Psi_{x,r}(\set{x_N=0})$ it holds
		\begin{equation}\label{eq:reifbdry}
		\dist_H(\partial X\cap B_r(x), \mathcal{L}_{x,r}\cap B_r(x)) \le  \eps r\, .
		\end{equation}
		It follows in particular that 
		\begin{equation}\label{eq:reifbound}
			\dist_{GH}(\partial X\cap B_r(x),B^{\setR^{N-1}}_r(0))\le 2\eps r\, ;
		\end{equation}
		\item[(iii)] for any $x\in B_{1/2}(p)\setminus \partial X$ and $r\in (0, \dist_{\partial X}(x)/2)$ the ball $B_r(x)$ is $(N,\eps)$-symmetric.
	\end{itemize}       
\end{theorem}

The proof of \autoref{thm:improvedneckregion} builds upon the following {$\eps$-regularity results for the Boundary,} which is based in turn on the stability of boundaries \autoref{thm:stabilityS} and the volume rigidity for cones with boundary \autoref{lemma:conerigidity}.

\begin{theorem}[Volume $\eps$-regularity for the Boundary]\label{cor:volrigidityb}
	Let $N\in \setN$, $N\ge 1$ be fixed. For any $\eps>0$, if $\delta\le \delta(N,\eps)$ and $(X,\dist,\haus^N,x)$ is an $\RCD(-\delta (N-1),N)$ pointed m.m.s., $x\in \partial X$ and 
	$\haus^N (B_1(x))  \ge  \frac{1}{2}\omega_N-\delta$,	 
	then 
	\begin{equation}
			\dist_{GH}(B_{1/2}(x),B_{1/2}^{\setR_+^N}(0))\le \eps\, .
	\end{equation}
\end{theorem}

\begin{proof}
We divide the proof into two steps, first proving a weak version of the statement, where we additionally assume a definite size of boundary points in the given ball, and then passing to the strong form via bootstrap.

\textbf{Step 1.}
We claim that the following holds: for any $\eps>0$ and $c>0$, if $\delta\le \delta(N,\eps,c)$ and $(X,\dist,\haus^N,x)$ is an $\RCD(-\delta (N-1),N)$ pointed m.m.s., $x\in \partial X$,
	$\haus^N (B_1(x))  \ge  \frac{1}{2}\omega_N-\delta$ and $\haus^{N-1}(\partial X\cap B_{1/2}(x))\ge c$,
	then 
	\begin{equation}\label{eq:bdryvolumerigidity}
			\dist_{GH}(B_{1/2}(x),B_{1/2}^{\setR_+^N}(0))\le \eps\, .
	\end{equation}

  We argue by contradiction. Let us assume that for some $\eps>0$ and $c>0$ and a sequence $(X,\dist_n,\haus^N,x_n)$ of $\RCD(-1/n,N)$ p.m.m.s. verifying with $x_n\in \partial X_n$, such that $\haus^N(B_1(x_n))\ge \frac{1}{2}\omega_N -\frac{1}{n}$,  $\haus^{N-1}(\partial X_n\cap B_{1/2}(x_n))\ge c$ and
   \begin{equation}
  \dist_{GH}(B_1(x_n),B_1^{\setR_+^N}(0))\ge \eps\quad \text{for any $n\in \setN$}\, .
   \end{equation}
  Up to extracting a subsequence $(X_n, \dist_n,\haus^N, x_n) \to (X,\dist,\haus^N,x)$ in the pmGH topology.\\ 
  Notice that, by \autoref{cor:stabwithlowerbound} and the uniform lower bound on the boundary measure, we infer that $\partial X\cap \bar{B}_{1/2}(x)\neq\emptyset$.

 Since $x$ is limit of boundary points it holds $\Theta_X(x)\le 1/2$, by lower semicontinuity of the density. Therefore, by volume convergence and thanks to the volume pinching assumption, $\haus^N(B_r(x))=1/2$ for any $r\in [0,1]$. Hence the volume cone implies metric cone theorem (see \cite{DePhilippisGigli16}) gives that $B_{1/2}(x)$ is isometric to the ball of radius $1/2$ of a cone $C(Z)$ centered at a tip point $z\in C(Z)$ with isometry sending $x$ to $z$. The sought contradiction comes from \autoref{lemma:conerigidity}. Indeed $Z$ has boundary and $\Theta(z)=\frac{1}{2}$, therefore $C(Z)$ is isometric to $\setR^N_+$.

\textbf{Step 2.}
Next we wish to remove the lower volume boundary assumption. In order to do so we first observe that, by a limiting argument, it is sufficient to prove the statement under the assumption that $x\in\mathcal{S}^{N-1}\setminus\mathcal{S}^{N-2}$.

Let us set $c:=4^{-N}c(N)^{-1}$ and assume without loss of generality that $\eps<\eta(N)$, where $c(N)$ and $\eta(N)$ are as in \autoref{thm:stabilityS}. We wish to prove that $\delta(N,\eps)=\delta(N,\eps,c)$ given by the previous step does the right job. Let us argue by contradiction. If this is not the case then we can find $0<r<1$ such that 
\begin{equation}
\dist_{GH}(B_{r/4}(x), B_{r/4}^{\setR^N_+}(0))\le\eps r/4\, ,
\end{equation}
but 
\begin{equation}\label{eq:contrboot}
\dist_{GH}(B_{r/2}(x), B_{r/2}^{\setR^N_+}(0)>\eps r/2\, ,
\end{equation}
since we know that the only element of the tangent cone at $x$ is $\setR^N_+$ (cf. with \autoref{thm:struc} (iii)).\\
Observe that 
\begin{equation}
\haus^{N-1}(\partial X\cap B_{r/2}(x))\ge \haus^{N-1}(\partial X\cap B_{r/4}(x)) \ge cr^{N-1}\, ,
\end{equation}
by our choice of $c$ and (the scale invariant version of) \autoref{thm:stabilityS}.\\
Moreover, by volume monotonicity, 
\begin{equation}
\haus^N(B_r(x))\ge \frac{1}{2}(\omega_N-\delta)r^N\, .
\end{equation}
Applying the result of Step 1 (in scale invariant form) we infer that
\begin{equation}
\dist_{GH}(B_{r/2}(x), B_{r/2}^{\setR^N_+}(0)\le \eps r/2\, ,
\end{equation}
therefore reaching a contradiction with \eqref{eq:contrboot}.
\end{proof}

\begin{remark}\label{rm:improvedbdryvolrigidity}
Under the same assumptions of \autoref{cor:volrigidityb}, it follows by volume monotonicity and scaling that \eqref{eq:bdryvolumerigidity} can be slightly improved to the statement
 \begin{equation}
			\dist_{GH}(B_{r/2}(x),B_{r/2}^{\setR_+^N}(0))\le \eps r\, ,
	\end{equation}
 for any $0<r<1$.
\end{remark}

\begin{proof}[Proof of \autoref{thm:improvedneckregion}]

	Let us begin by observing that for any $\delta \le \delta(\delta_0,N)$ if $B_2(p)$ is a $\delta$-boundary ball of an $\RCD(-\delta(N-1),N)$ space then
	\begin{equation}\label{zzz7}
		\text{$B_1(x)$ is a $\delta_0$-boundary ball
			for any $x\in \partial X \cap B_1(p)$}\, .
	\end{equation}
	This claim can be checked arguing by contradiction and exploiting the fact that boundary points converge to boundary points when the limit ball is isometric to a ball centered on the boundary of the half space $\setR_+^N$ (cf. with \autoref{lemma:distancetodistance}).
\smallskip

Next, by volume convergence, choose $\delta_0=\delta_0(\eps,N)<\eps/2$ such that if $B_1(x)$ is a $\delta_0$-boundary ball on an $\RCD(-\delta(N-1),N)$ and $x\in\partial X$, then the assumptions of \autoref{cor:volrigidityb} are satisfied.\\
If we choose $\delta$ accordingly given by the observations above, then by \autoref{cor:volrigidityb} we infer that, for any $x\in\partial X$ and for any $0<r<1$, $B_r(x)$ is an $\eps$-boundary ball.

  From now on we let 
   \begin{equation}
    \Psi_{x,r}:B_r^{\setR_+^N}(0) \to B_r(x)
   \end{equation}
   be $\eps r$-isometries, for any $x\in \partial X$ and any $0<r<1$.
   	\medskip

	Let us prove (ii). We first prove that for any $z\in B_r^{\setR^N_+}(0)\cap \{x_N=0\}$ there exists $y\in B_r(x)\cap\partial X$ such that $\dist(\Psi_{x,r}(z),y)\le C(N)\eps r$. In order to do so it is sufficient to observe that for any $s>0$ the ball $B_{sr}(\Psi_{x,r}(z))$ is an $s^{-1}\eps$-boundary ball. By \autoref{thm:stabilityS}, if $s^{-1}\eps<\eta(N)$, then there exists a boundary point $y\in B_{rs}(\Psi_{x,r}(z))\cap\partial X$. Therefore, minimizing we infer that there exists a boundary point $y\in B_r(x)\cap\partial X$ such that $\dist(\Psi_{x,r}(z),y)\le C(N)\eps r$, where $C(N):=1/\eta(N)$.
	
    It remains to prove that for any $y\in B_r(x)\cap\partial X$ there exists $z\in B_r^{\setR^N_+}(0)\cap\{x_N=0\}$ such that $\dist(y,\Psi_{x,r}(z))\le C(N)\eps r$. In order to do so it is sufficient to observe that, by elementary considerations, if $B_1(p)\subset \setR^N_+$ is an $\eps$-boundary ball, then $\dist(p,\{x_N=0\})\le C(N)\eps$, therefore by scaling invariance of the half-space, if $B_r(p)$ is an $\eps$-boundary ball, then $\dist(p,\{x_N=0\})\le C(N)\eps r$.

   \medskip

	Let us finally prove (iii). In order to obtain the conclusion it is sufficient to prove the following statement: for any $\eps>0$ there exists $\delta>0$ such that if $(X,\dist,\haus^N)$ is an $\RCD(-\delta(N-1),N)$ m.m.s., $p\in X$ and $B_4(p)$ is a $\delta$-boundary ball, then for any $x\in B_1(p)\setminus\partial X$ the ball $B_{\dist_{\partial X}(x)/2}(x)$ is $(N,\eps)$-symmetric.
	Indeed, if $B_{\dist_{\partial X}(x)/2}(x)$ is $(N,\eps)$-symmetric, then by volume convergence \autoref{thm:volumeconvergence} it has almost Euclidean volume and by volume almost rigidity \cite[Theorem 1.6]{DePhilippisGigli18} we infer that $B_r(x)$ is $(N,\eps)$-symmetric for any $0<r<\dist_{\partial X}(x)/2$, up to worsening $\eps$.

   Let us now prove the claimed conclusion. Let $q\in \partial X$ be such that $\dist_{\partial X}(x) = \dist(q,x)$ and set $r:= 2\dist_{\partial X}(x)$. As a consequence of (i) and (ii) we know that there exists an $\eps r$-GH isometry $\Psi_{x,r}: B_r^{\setR_+^N}(0) \to B_r(q)$.\\ 
   By elementary considerations we can find $z\in \setR_+^N$ such that 
   \begin{equation}
   B_{\dist_{\partial X}(x)/2}(z) \subset B_r^{\setR_+^N}(0) \setminus \partial \setR_+^N
   \end{equation}
   	and
   	\begin{equation}
   	\Psi_{x,r} :B_{\dist_{\partial X(x)/2}}(z) \to B_{\dist_{\partial X(x)/2}}(x)
   \end{equation}
    is an $8\eps r$-GH isometry.

\end{proof}

%%%%%%%%%%%%%%%%%

\subsection{Topological regularity of the boundary}
Thanks to \autoref{thm:improvedneckregion} we better understand the geometry of $\delta$-boundary balls. Below we build a parametrization of the boundary of a $\delta$-boundary ball well suited for its geometry. In particular this parametrization will put us in position to control both the topology and the volume near to sufficiently regular boundary points (cf. with \cite{Anderson90,CheegerColding97,KapovitchMondino19} in the case of interior regular points).\\ 
With respect to \autoref{thm:neckstructure} here we heavily rely on \autoref{thm:improvedneckregion} and on the transformation \autoref{prop:transformation} to get bi-H\"older continuity of the splitting map, as in \cite{CheegerJiangNaber18}.

\begin{theorem}\label{thm:goodpar}
	Let $1\le N<\infty$ be a fixed natural number. Then for each $0<\eps<1/5$ there exists $\delta=\delta(\eps,N)>0$ such that for any $\RCD(-\delta(N-1),N)$ space $(X,\dist,\haus^N)$ and for any $\delta$-boundary ball $B_{16}(p)\subset X$, $\partial X\cap B_8(p)$ is homeomorphic to a smooth $(N-1)$-dimensional manifold without boundary.\\ 
	Moreover, there exists a map $u:B_8(p)\to\setR^{N-1}$ verifying the following properties:
	\begin{itemize}
		\item[i)] $u:B_8(p)\to\setR^{N-1}$ is an $\eps$-splitting map;
		\item[ii)] there exists a closed set $U\subset B_1(p)\cap \partial X$ such that  
		\begin{equation}\label{eq:volest}
		\haus^{N-1}\left((B_1(p)\cap \partial X)\setminus U\right)\le \eps
		\end{equation}
		and 
		\begin{equation}\label{eq:bilipest}
		(1-\eps)\dist(x,y)\le \abs{u(x)-u(y)}\le (1+\eps)\dist(x,y)
		\quad \text{for any $x,y\in U$}\, ;
		\end{equation}
		\item[iii)] for any $x,y\in\partial X\cap B_1(p)$ it holds that
		\begin{equation}\label{eq:holderest}
		(1-\eps)\dist(x,y)^{1+\eps}\le \abs{u(x)-u(y)}\le (1+\eps)\dist(x,y)\, ;
		\end{equation}
		\item[iv)] $u(B_1(p)\cap\partial X)\supset B_{1-2\eps}^{\setR^{N-1}}(0)$.
	\end{itemize}
	
\end{theorem}

\begin{proof}

The fact that $B_8(p)\cap\partial X$ is homeomorphic to a smooth $N-1$-dimensional manifold without boundary follows from \autoref{thm:improvedneckregion} (ii) thanks to Reifenberg's theorem for metric spaces \cite[Theorem A.1.1]{CheegerColding97}, see also \cite{Reifenberg,Anderson90} and \cite{KapovitchMondino19}.
\smallskip

	Let us fix $0<\eta<\eps$ to be specified later. Choosing $\delta$ small enough we can build an $\eta$-splitting map $u:B_8(p)\to \setR^{N-1}$ by \autoref{splitting vs isometry}. This in particular proves (i).
	\medskip
	
	Let us now show (ii). We set
	\[
	U:=\left\lbrace x\in B_1(p) \cap \partial X\, :\, s\fint_{B_s(x)} |\Hess u|^2 \di \haus^N \le \eta^{1/2}\quad \text{for any $s\in (0,5)$} \right\rbrace\, .
	\]
	Notice that $U$ is closed and $u: B_s(x) \to \setR^{N-1}$ is a $C(N)\eta^{1/4}$-splitting map for any $s\in (0,5)$, by \autoref{remark:smallHessianimpliessplitting}. For $\delta$ and $\eta$ small enough we deduce from \autoref{lemma:improvedepsio} that 
	\[
	(u,\dist_{\partial X}): B_s(x) \to \setR_+^N\quad \text{is an $\frac{\eps}{1+\eps}$-GH isometry}
	\]
	for any $x\in U$ and $s\in (0,5/2)$. In particular, given $x,y\in U$ and $s=\dist(x,y)(1+\eps)$, it holds
	\[
	||u(x) - u(y)|-\dist(x,y)| \le \frac{\eps}{1+\eps} s = \eps \dist(x,y)\, ,
	\]
	therefore yielding \eqref{eq:bilipest}.
	
    Let us prove \eqref{eq:volest}.
	A standard Vitali's covering argument produces a disjoint family of balls $\set{B_{s_i}(x_i)}_{i\in \setN}$ with $x_i\in B_1(p)\cap  \partial X$, $s_i\in (0,1)$ such that 
	\begin{equation}\label{zzzz1}
	(B_1(p)\cap \partial X) \setminus U\subset \bigcup_{i\in \setN} B_{5r_i}(x_i)
	\quad \text{and}\quad
	5s_i\fint_{B_{5s_i}(x)} |\Hess u|^2 \di \haus^N > \eta^{1/2}\, .
	\end{equation}
	Relying on \eqref{zzzz1}, the Bishop-Gromov inequality, \autoref{rm:volumeonboundary balls} and \autoref{thm:boundatystructure} (i) (see also \autoref{thm:lipbdry} (ii)), we conclude that
	\begin{align*}
		\haus^{N-1}((B_1(p)\cap \partial X) \setminus U) & \le \sum_{i\in \setN} \haus^{N-1}(B_{5s_i}(x_i)\cap \partial X)
		 \le C(N) \sum_{i\in \setN} s_i^{N-1} 
		\\& \le C(N) \eta^{-1/2} \sum_{i\in\setN} \int_{B_{s_i}(x_i)} |\Hess u|^2 \di \haus^N
		\\&\le C(N) \eta^{1/2} \le \eps\, ,
	\end{align*}
    for $\eta$ sufficiently small.
\medskip

The H\"older estimate \eqref{eq:holderest} will follow from \autoref{thm:neckdecompositiontheorem} (i) and the transformation \autoref{prop:transformation}, arguing as in the proof of \cite[Theorem 7.10]{CheegerJiangNaber18}. 

More precisely, if $x,y\in B_1(p)\cap\partial X$ we set $r:=\dist(x,y)$. Then by \autoref{thm:improvedneckregion} (i) we infer that $B_r(x)$ is an $\eta$-boundary ball, for $\eta$ small to be chosen later, if $\delta\le\delta(N,\eta)$. Then we apply the transformation \autoref{prop:transformation} to obtain existence of a lower triangular matrix $T_{x,r}$ such that $T_{x,r}u:B_r(x)\to\setR^{N-1}$ is an $\eps'$-splitting map for $\eps'$ small to be chosen later and for any $\eta\le\eta(N,\eps')$. Taking into account \autoref{lemma:improvedepsio} we obtain that
\begin{equation}
\abs{T_{x,r}u(x)-T_{x,r}u(y)}\ge (1-\eps)\dist(x,y)\, .
\end{equation}
Taking into account the matrix growth estimate $\abs{T_{x,r}}\le r^{-\eps}$ (cf. \autoref{cor:growth}) and that $r=d(x,y)$ we get that
\begin{equation}
\abs{u(x)-u(y)}\ge(1-\eps)\dist(x,y)^{1+\eps}\, .
\end{equation}
The upper bound in \eqref{eq:holderest} follows from \autoref{rm:sharpgradientbound}.

\medskip

To prove the last assertion we argue as in \cite[Remark 2.10]{KapovitchMondino19}. We claim that there exists $0<s<1$ such that
	\[
	u(\overline B_s(p)\cap\partial X) \cap \overline B_{1-2\eps}^{\setR^{N-1}}(0) = \overline B_{1-2\eps}^{\setR^{N-1}}(0)\, .
	\]
	To this aim we observe that $u(\overline B_s(p)\cap\partial X) \cap \overline B_{1-2\eps}^{\setR^{N-1}}(0)$ is non empty and closed in  $\overline B_{1-2\eps}^{\setR^{N-1}}(0)$. Moreover, it holds
	\[
		u(\overline B_s(p)\cap\partial X) \cap \overline B_{1-2\eps}^{\setR^{N-1}}(0) = 	u( B_s(p)\cap\partial X) \cap \overline B_{1-2\eps}^{\setR^{N-1}}(0)\, ,
	\]
	whenever $s^{1+\eps}(1-\eps)>1-2\eps$. Indeed for any $q\in \partial B_s(p)\cap \partial X$ one has
	\[
	|u(q)|=|u(q)-u(p)|\ge (1-\eps) s^{1+\eps} > 1-2\eps\, ,
	\]
	as a consequence of \eqref{eq:holderest}. Therefore, for such a choice of $s$ the set $u(\overline B_s(p)\cap\partial X) \cap \overline B_{1-2\eps}^{\setR^{N-1}}(0)$ is also open in $\overline B_{1-2\eps}^{\setR^{N-1}}(0)$ as a consequence of the invariance of the domain (here we are using that $B_8(p)\cap\partial X$ is an $(N-1)$-dimensional topological manifold, as we already pointed out). 
\end{proof}

    	\begin{remark}
    		In the smooth case, i.e. when $(X,\dist,\haus^N)$ is a Riemannian manifold with convex boundary and Ricci curvature bounded below by $-\delta(N-1)$, the map $u:B_1(p)\cap\partial X\to \setR^{N-1}$ obtained in \autoref{thm:goodpar} is also a diffeomorphism onto its image. This follows by observing that $u$ is smooth and 
    		\begin{equation}
    		\di u_x : T_x \partial X \to \setR^{N-1}\quad \text{is nondegenerate for any}\, x\in B_1(p)\cap\partial X\, .
    		\end{equation}

Observe that $u:B_2(p) \to \setR^{N-1}$ is a $\eps$-splitting map where $\eps\le \eps(N)$. Let $x\in B_1(p)\cap\partial X$. By the transformation \autoref{prop:transformation} for any $r\le 1$, there exists an $N\times N$ matrix $A_{x,r}$ such that $A_{x,r}\circ u : B_r(x) \to \setR^{N-1}$ is a $\delta$-splitting map, where $\delta\le \delta(\delta',N)$.
    		For $r\le r_x$ small enough, standard elliptic regularity estimates up to the boundary give 
    		\begin{equation}
    			\sup_{B_r(x)} | \nabla (A_{x,r}\circ u)^a \cdot \nabla (A_{x,r}\circ u)^b - \delta_{\alpha,\beta}| \le C(N)\delta'
    			\quad \text{for any $a,b=1,\ldots,N-1$}\, ,
    		\end{equation}
    		which implies that $\di u_x$ is nondegenerate.  		
    	\end{remark}

\vspace{.3cm}

\begin{corollary}
Let $1\le N<\infty$ be a natural number and $(X,\dist,\haus^N)$ be an $\RCD(-(N-1),N)$ metric measure space. Assume that $\partial X\neq\emptyset$, then for any $0<\alpha<1$ there exists $U_{\alpha}\subset\partial X$ such that:
\begin{itemize}
	\item[i)] $U_{\alpha}$ is relatively open and dense in $\partial X$;
	\item[ii)] $\dim_{H}(\partial X\setminus U_{\alpha})\le N-2$;
	\item[iii)] $U_{\alpha}$ is an $(N-1)$-dimensional $\alpha$-H\"older topological manifold without boundary and the charts can be chosen with components that are restriction of harmonic maps (on the ambient space).
\end{itemize}
\end{corollary}

\begin{proof}
	Fix $\alpha\in (0,1)$. Thanks to \autoref{thm:goodpar} we can find $\delta<\delta(N,\alpha)$ with the property that if $B_{16}(p)$ is a $\delta$-boundary ball of an $\RCD(-\delta(N-1),N)$ m.m.s. $(X,\dist,\haus^N)$, then $\partial X\cap B_1(p)$ is a $C^{\alpha}$ manifold of dimension $(N-1)$.\\	
	For any $x\in \mathcal{S}^{N-1}\setminus\mathcal{S}^{N-2}$ we consider $r_x\in (0,\delta)$ such that $B_{16 r_x}(x)$ is a $\delta$-boundary ball and we set 
	\[
	U_{\alpha}:=\bigcup_{x\in \mathcal{S}^{N-1}\setminus\mathcal{S}^{N-2}} B_{r_x}(x) \cap\partial X\, .
	\]
	By construction $U_\alpha$ satisfies (iii).
	Notice that $U_{\alpha}$ is open and dense in $\partial X$ and $\partial X\setminus U \subset \mathcal{S}^{N-2}$, yielding (i) and (ii). 
\end{proof}

\begin{corollary}\label{cor:bdryvolconvergencereg}
Let $1\le N<\infty$ be a fixed natural number. For any $\eps>0$ and $\delta<\delta(\eps,N)>0$ the following holds. If $(X,\dist,\haus^N)$ is an $\RCD(-\delta(N-1),N)$ space and $B_{16}(p)$ is a $\delta$-boundary ball, then 
\begin{equation}\label{eq:sharpahlfors}
(1-\eps)\omega_{N-1}\le\haus^{N-1}(\partial X\cap B_1(p))\le (1+\eps)\omega_{N-1}\, . 
\end{equation}
Moreover, $\mathcal{F}X\cap B_1(p)=\partial X\cap B_1(p)$ and for any $x\in\partial X\cap B_1(p)$, any tangent cone at $x$ has boundary.
 
\end{corollary}	

\begin{proof}

Let $\eps'<\eps$ to be fixed later. For $\delta\le \delta(N,\eps')$ we find a $\eps'$-splitting function $u:B_8(p)\to \setR^{N-1}$ satisfying (i)-(iv) in \autoref{thm:goodpar} with $\eps'$ in place of $\eps$. Let in particular $U\subset B_1(p)\cap\partial X$ be the good set appearing in \autoref{thm:goodpar} (ii).

The inclusion $u(U)\subset B_{1+\eps'}(0)$ implies
\begin{equation}\label{zzzzz5}
\begin{split}
\haus^{N-1}(\partial X\cap B_1(p))  & \le \eps' + \haus^{N-1}(U)\le \eps' + (1+\eps')^{N-1}\Leb^{N-1}(u(U)) 
\\& \le \eps' + (1+\eps')^{2N-2}\omega_{N-1}\, .
\end{split}
\end{equation}
On the other hand since $u(B_1(p)\cap\partial X)\supset B_{1-2\eps'}^{\setR^{N-1}}(0)$, $u((\partial X\cap B_1(p))\setminus U)\le C(N)\eps'$ and $u$ is bi-Lipschitz on $U$ we infer that 
\begin{equation}\label{zzzzz6}
\haus^{N-1}(\partial X\cap B_1(p))\ge\haus^{N-1}(U)\ge \frac{1}{(1+\eps')^{N-1}}\left(\omega_{N-1}(1-2\eps')^{N-1}-C(N)\eps'\right)\, .
\end{equation}
The sought conclusion \eqref{eq:sharpahlfors} follows from \eqref{zzzzz5} and \eqref{zzzzz6} by choosing $\eps'$ small enough.
\medskip 

The second part of the statement follows from \eqref{eq:sharpahlfors} taking into account the following general property: given a noncollapsed $\RCD(-(N-1),N)$ space $(X,\dist,\haus^N)$ and a point $x\in\partial X$ such that
	\begin{equation}\label{eq:lowliminf}
	\liminf_{r\to 0}\frac{\haus^{N-1}(\partial X\cap B_r(x))}{r^{N-1}}>0\, ,
	\end{equation}
	then any tangent cone to $(X,\dist)$ at $x$ has boundary.\\	
	The verification of the claim above follows from \autoref{cor:stabwithlowerbound}, taking into account the scaling properties of $\haus^{N-1}$.

\end{proof}

\subsection{Convergence of boundary measures}
Let us recall that for measures defined on sequences of metric spaces converging in the pGH sense, weak convergence is understood in duality with continuous functions with bounded support once the pGH converging metric spaces are embedded in a common proper metric space (cf. \cite{AmbrosioHonda,AmbrosioHonda18}).

In this framework, two standard consequences of weak convergence are the lower semicontinuity of the evaluation on open sets and the upper semicontinuity of the evaluation on closed sets: if $\mu_n$ are locally finite measures on $Z$ weakly converging to $\mu$ in duality with continuous functions with bounded support as $n\to\infty$ and $A\subset Z$ and $C\subset Z$ are an open and a closed subset respectively, then
\begin{equation}
\mu(A)\le\liminf_{n\to\infty}\mu_n(A)\;\;\;\text{and}\;\;\;\mu(C)\ge\limsup_{n\to\infty}\mu_n(C)\, .
\end{equation}
From the two properties above one can easily infer the full convergence $\lim_{n\to\infty}\mu_n(O)=\mu(O)$, for any Borel set $O\subset Z$ such that $\mu(\partial O)=0$, where we denoted by $\partial O$ the topological boundary of $O$.

We now prove \autoref{intro:stabilityboundary} which is restated below for reader's convenience.

\begin{theorem}[Boundary Volume Convergence]
	Let $1\le N<\infty$ be a fixed natural number. Assume that $(X_n,\dist_n,\haus^N,p_n)$ are $\RCD(-(N-1),N)$ spaces converging in the pGH topology to $(X,\dist,\haus^N,p)$. 
	Then
	\begin{equation}
	\haus^{N-1}\res \partial X_n \to \haus^{N-1}\res \partial X\quad \text{ weakly}\, .
	\end{equation}
	
	In particular
	\[
	\lim_{n\to \infty} \haus^{N-1}(\partial X_n\cap B_r(x_n)) = \haus^{N-1}(\partial X\cap B_r(x))\, ,
	\]
	whenever $X_n\ni x_n\to x\in X$ and $\haus^{N-1}(\partial X\cap \partial B_r(x))=0$.	
\end{theorem}

\begin{proof}
Set $\nu_n:= \haus^{N-1}\res \partial X_n$. Up to extracting a subsequence one has that $\nu_n \to \nu$ weakly, where $\nu$ is a nonnegative measure on $X$ satisfying
\begin{itemize}
	\item[(i)] $\nu(B_r(x)) \le C(N) r^{N-1}$ for any $x\in X$ and a.e. $r\in (0,2)$;
	\item[(ii)] $\supp \nu \subset \mathcal{S}^{N-1}$.
\end{itemize}
Here we have used standard compactness theorem for measures along with \autoref{thm:boundatystructure} (i), the lower semicontinuity of the density $\Theta_{X_n}$ w.r.t. the GH convergence, and \autoref{rm:bdrytobdryhalf}.
We need to prove that  $\nu=\haus^{N-1}\res \partial X$.

Let us begin by observing that, as a consequence of (i) and (ii), it holds that $\nu \ll \haus^{N-1}\res \mathcal{S}^{N-1} = \haus^{N-1}\res \partial X $. In particular if $\mathcal{S}^{N-1}(X)\setminus \mathcal{S}^{N-2}(X) = \emptyset$ then $v = \haus^{N-1}\res \partial X = 0$. We can therefore assume that $\mathcal{S}^{N-1}(X)\setminus \mathcal{S}^{N-2}(X) \neq \emptyset$. To get the claimed conclusion  it is enough to verify that
\begin{equation}\label{zzzzz7}
	\lim_{r\to 0} \frac{\nu(B_r(x))}{\omega_{N-1}r^{N-1}} = 1\quad \text{for $\haus^{N-1}$-a.e. $x\in\partial X$}\, ,
\end{equation}
thanks to a classical result about differentiation of measures \cite{Kirchheim94}. Here we  used that $\partial X$ is $(N-1)$-rectifiable with locally finite $\haus^{N-1}$-measure.

To prove \eqref{zzzzz7} we rely on \autoref{cor:bdryvolconvergencereg}. Observe that we can check \eqref{zzzzz7} just considering regular boundary points $x\in \mathcal{S}^{N-1}\setminus \mathcal{S}^{N-2}$ for which the limit 
\begin{equation}
\lim_{r\to 0} \frac{\nu(B_r(x))}{\omega_{N-1}r^{N-1}} 
\end{equation}
exists.

Let us fix $\delta>0$. For any $x\in \partial X$ as above there exists $r_x<1$ such that $B_r(x)$ is a $\delta/2$-boundary ball for every $r\in (0,r_x)$, thanks to \autoref{thm:improvedneckregion} (i). In particular, if $\partial X_n\ni x_n\to x\in \partial X$ then $B_r(x_n)$ is a $\delta$-boundary ball for any $r\in (0,r_x]$, and $n$ big enough (here we rely again on \autoref{thm:improvedneckregion} (i) to handle radii in $(0,r_x)$). Notice that the existence of the sequence $(x_n)$ verifying the property above for $r=r_x$ is a consequence of the stability \autoref{thm:stabilityS}. Let us now fix $\eps>0$ and assume that $\delta\le \delta(N,\eps)$ so that \autoref{cor:bdryvolconvergencereg} holds true. We get
\[
\abs{\frac{\nu(B_r(x))}{\omega_{N-1} r^{N-1}} - 1} =\abs{ \lim_{n\to \infty}  \frac{\haus^{N-1}(\partial X\cap B_r(x_n))}{\omega_{N-1}r^{N-1}} - 1} \le \eps
\quad \text{for a.e. $r\in (0,r_x)$}\footnote{in particular for all those radii such that $\nu(\overline{B}_r(x)\setminus B_r(x))=0$, a property which fails in at most countable cases.},
\]
which yields \eqref{zzzzz7}, being $\eps$ arbitrary.

\end{proof}

\section{Topological regularity up to the boundary}\label{sec:topologicalstructureuptotheboundary}
In \cite[Corollary 3.2]{KapovitchMondino19}, following the arguments of \cite{CheegerColding97} (see also the previous \cite{Anderson90}) and relying on Reifenberg's theorem for metric spaces, it has been proved that on any noncollapsed $\RCD(K,N)$ space $(X,\dist,\haus^N)$ and for any $\alpha\in(0,1)$ there exists an open and dense subset $U\subset X$ such that:
\begin{itemize}
	\item $\dim_H(X\setminus (U\cup\partial X))\le N-2$;
	\item $U$ is an $N$-dimensional topological manifold with no boundary and $C^{\alpha}$-charts.
\end{itemize}
The aim of this section is to sharpen this result including the boundary in the topological regularity statement. We shall prove that any noncollapsed $\RCD(K,N)$ is space (with boundary) is homeomorphic, up to a set of codimension two, to a topological manifold (with boundary) with $\alpha$-biH\"older charts for any $0<\alpha<1$.

In view of \cite[Theorem 4.11]{KapovitchMondino19} and of \autoref{thm:neckdecompositiontheorem} the main tool needed to get topological regularity up to the boundary are biH\"older topological charts on $\delta$-boundary balls.

\begin{theorem}\label{thm:boundaryballshomeo}
		Let $N\in \setN$, $N \ge 1$ and $(X,\dist,\haus^N)$ be an $\RCD(-\delta(N-1),N)$ space, $p\in X$ such that $B_{16}(p)$ is a $\delta$-boundary ball.
		Then, for any $0<\eps<1/5$ if $\delta<\delta(N,\eps)$ there exists
		$F:B_1(p)\to \setR_+^N$ such that 
		\begin{itemize}
			\item[(i)] $(1-\eps) \dist(x,y)^{1+\eps} \le |F(x) - F(y)| \le C(N) \dist(x,y)$ for any $x,y\in B_1(p)$;
			\item[(ii)] $F(p)=0$ and $\partial \setR_+^N \cap B_{1-2\eps}(0) \subset F(\partial X \cap B_1(p)) = \partial \setR^N_+\cap F(B_1(p))$;
			\item[(iii)] $F$ is open and a homeomorphism onto its image;
			\item[(iv)] $ B_{1-2\eps}^{\setR_+^N}(0) \subset F(B_1(p))$.
		\end{itemize}
\end{theorem}

    By combining \cite[Theorem 4.11]{KapovitchMondino19} and \autoref{thm:boundaryballshomeo} we infer that on a noncollapsed $\RCD(-(N-1),N)$ m.m.s., if $\eps<\eps(N)$ then any point in $X\setminus\mathcal{S}^{N-2}_{\eps}$ has a neighbourhood which is $C^{\alpha}$-homeomorphic either to an open set in $\setR^N$ or to an open set in $\setR^N_+$ (see \eqref{eq:quantstrat} for the definition of the quantitative singular stratum $\mathcal{S}^{N-2}_{\eps}$). It is then easy to infer the following.

\begin{theorem}\label{cor:manifoldstructure}
	Let $N\in \setN$, $N\ge 1$ be fixed and $0<\alpha<1$. If $(X,\dist,\haus^N)$ is a $\RCD(-(N-1),N)$ metric measure space, then there exists a closed set of codimension two $C_\alpha \subset\mathcal{S}^{N-2}_{\eps}(X)$, for some $0<\eps<\eps(N,\alpha)$, such that $X\setminus C_\alpha$ is a topological manifold with boundary and $C^{\alpha}$-charts. Moreover the topological boundary of $X\setminus C_\alpha$ coincides with $\partial X\cap (X\setminus C_\alpha)$.
\end{theorem}

\begin{proof}
It is sufficient to prove that if $\eps<\eps(N,\alpha)$, then any point in $X\setminus\mathcal{S}^{N-2}_{\eps}$ admits a neighbourhood which is either $C^{\alpha}$-homeomorphic to an open subset of $\setR^N$, or $C^{\alpha}$-homeomorphic to an open subset of $\setR^N_+$.\\
In order to do so we just observe that, if $x\in X\setminus\mathcal{S}^{N-2}_{\eps}$, then there exists $0<r<1$ such that either $B_r(x)$ is $(N,\eps)$-symmetric, or $B_r(x)$ is an $\eps$-boundary ball. In the first case $x$ has a neighbourhood $C^{\alpha}$-homeomorphic to an open subset of $\setR^N$ by \cite[Theorem 4.11]{KapovitchMondino19} (see also \cite{CheegerColding97,Anderson90}), for $\eps<\eps(\alpha,N)$. In the second case, by \autoref{thm:boundaryballshomeo} $x$ has a neighbourhood $C^{\alpha}$-homeomorphic to an open subset of $\setR^N_{+}$, if $\eps<\eps(\alpha,N)$.

\end{proof}

	In the framework of limits of $N$-dimensional manifolds with convex boundary and Ricci tensor bounded below by $-(N-1)$ in the interior we can improve \autoref{cor:manifoldstructure} with the following. 
	
\begin{theorem}\label{thm:manifoldlimits}
	Let $(X,\dist,\haus^N)$ be an $\RCD$ m.m.s. arising as noncollapsed limit of a sequence of smooth Riemannian manifolds with convex and Ricci curvature bounded from below in the interior by $-(N-1)$. Then for any $0<\alpha<1$ there exists a constant $C=C(N,\alpha,\haus^N(B_1(p)))$ and a closed set of codimension two $C_\alpha \subset\mathcal{S}^{N-2}(X)$ such that
	\begin{equation}
	\haus^{N-2}(C_\alpha\cap B_1(p)) \le C(N,\alpha,\haus^N(B_1(p)))\footnote{Actually, even a stronger Minkowski and packing type estimate can be obtained with the same strategy.},\quad \text{for any $p\in X$}
	\end{equation}
	 and $X\setminus C_\alpha$ is a topological manifold with boundary and $C^{\alpha}$-charts. Moreover the topological boundary of $X\setminus C_\alpha$ coincides with $\partial X\cap (X\setminus C_\alpha)$.
\end{theorem}	

  The improvement will follow from the sharp measure estimates for the effective singular stratum $\mathcal{S}^{N-2}_{\eps}$ on noncollapsed Ricci limit spaces obtained in \cite{CheegerJiangNaber18}.  The conclusion is almost straightforward once we point out that the verbatim arguments of \cite{CheegerJiangNaber18} allow to treat also the case of noncollapsed limits of smooth Riemannian manifolds with convex boundary (and interior lower Ricci curvature bounds). Since that case was not considered therein, we also give a detailed proof relying on a gluing procedure (see \cite{Schlicting12}) in order to reduce the study of singularities in the boundary to that of interior singularities.

\begin{proof}[Proof of \autoref{thm:manifoldlimits}]

First let us point out that, the analogue of \cite[Theorem 1.9]{CheegerJiangNaber18} in the case of non collapsed limits of smooth manifolds with convex boundary and interior lower ricci curvature bounds yields that, under our assumptions,
\begin{equation}
\haus^{N-2}(\mathcal{S}^{N-2}_{\eps}\cap B_1(p))<C(N,\eps,\haus^N(B_1(p)))\, .
\end{equation}
The conclusion then follows from \autoref{cor:manifoldstructure}, where we proved that the topologically singular set is included in $\mathcal{S}^{N-2}_{\eps}$. The validity of \cite[Theorem 1.9]{CheegerJiangNaber18} in the case of manifolds with boundary can be checked with the verbatim arguments therein indicated.
\medskip

However, since the case of manifolds with boundary is not considered in \cite{CheegerJiangNaber18}, below we provide an alternative proof under the additional technical assumption that the smooth approximating manifolds have boundaries with uniformly bounded diameter.

	Let us start by pointing out that 
	\begin{itemize}
	\item[i)] if $x\in\partial X$ and $\Theta_X(x)\ge 1/2-\eta(N,\alpha)$, then $x$ has a neighbourhood $C^{\alpha}$-homeomorphic to an open set in $\setR^N_+$; 
	\item[ii)] if $x\in X$ verifies $\Theta_X(x)\ge 1-\eta(N,\alpha)$, then $x$ has an open neighbourhood $C^{\alpha}$-homeomorphic to an open subset of $\setR^N$.
	\end{itemize}
	We refer to \cite{KapovitchMondino19,CheegerColding97,Anderson90} for the proof of (ii), which is based on Reifenberg's theorem for metric spaces.\\
	Property (i) instead directly follows from the Boundary volume rigidity \autoref{cor:volrigidityb} and
	\autoref{thm:boundaryballshomeo}.
	
	It follows from the discussion above that, letting 
	\begin{equation}
	I:=\{x\in\partial X\,:\,\Theta_X(x)<\frac12-\eta(N,\alpha) \},\;\;\; P:=\{x\in X\setminus\partial X\,:\,\Theta_X(x)<1-\eta(N,\alpha) \}\, ,
	\end{equation}
	it suffices to prove 
	\begin{equation}\label{eq:bounbound}
	\haus^{N-2}\left((I\cup P)\cap B_1(p)\right)\le C(N,\alpha,\haus^N(B_1(p)))\, .
	\end{equation} 

Let now $(\hat X, \hat \dist, \haus^N)$ be the doubling of $(X,\dist,\haus^N)$ gluing along $\partial X$, see for instance \cite{ProfetaSturm18} for the precise definition. We claim that it is noncollapsed Ricci limit (of a sequence of smooth $N$-dimensional Riemannian manifolds with no boundary and Ricci curvature bounded from below by $-N$).
\smallskip

Before proving the claim let us see how it implies \eqref{eq:bounbound}. In order to do so we let $\iota:X\to\hat{X}$ be one of the canonical immersions of the starting space into its double. Since $\iota$ is isometric, in order to prove \eqref{eq:bounbound} it suffices to prove that 
\begin{equation}
	\haus^{N-2}\left(\iota(I\cup P)\cap B_1(\iota(p))\right)\le C(N,\alpha,\haus^N(B_1(\iota(p))))\, .
\end{equation}
It is easy to check that, for any $\hat{x}\in\iota(I\cup P)$ it holds
	\[
\Theta_{\hat X} (x) \le 1 - 2\eta(N,\alpha)\, .
\]
Hence, there exists $\eps=\eps(N,\alpha)$ such that
\begin{equation}
\iota(I\cup P)\subset \mathcal{S}^{N-2}_{\eps}(\hat{X})\, .
\end{equation}
Applying \cite[Theorem 1.9]{CheegerJiangNaber18} to $(\hat{X},\hat{\dist})$ we infer that
\begin{equation}
\haus^{N-2}(\iota(I\cup P)\cap B_{1}(\iota(p)))\le C(N,\alpha,\haus^N(B_1(p)))\, ,
\end{equation}
which yields the sought estimate.
\smallskip

Let us pass to the verification of the claim. In order to do so we let $(X_n,\dist_n)$ be the sequence of smooth manifolds with boundary converging to $(X,\dist)$. We then let $(\hat{X}_n,\hat{\dist}_n)$ be the doubling along the boundary of $(X_n,\dist_n)$. From \cite{Schlicting12} (see also the previous \cite{Kosolvsky02}) we deduce that $(\hat{X}_n,\hat{\dist}_n)$ is a noncollapsed limit of a sequence of smooth Riemannian manifolds with no boundary and Ricci curvature bounded below by $-N$, for any $n\in\setN$. Then it is easy to check that $(\hat{X}_n,\hat{\dist}_n)$ converge in the pGH topology to $(\hat{X},\hat{\dist})$ without collapse. To conclude we observe that a diagonal argument yields that $(\hat{X},\hat{\dist})$ is a noncollapsed Ricci limit space.

\end{proof}

The remaining part of this section is devoted to the proof of \autoref{thm:boundaryballshomeo}. Let us first introduce a regularization result for the distance to the boundary on $\delta$-boundary balls. 

\begin{lemma}\label{lemma:distanceapproximation}
	Let $N\in \setN$, $N\ge 1$ be fixed. For any $\eps>0$, if $\delta<\delta(N,\eps)$ then the following holds. Given an $\RCD(-\delta(N-1),N)$ m.m.s. $(X,\dist,\haus^N)$ and $p\in X$ such that $B_8(p)$ is a $\delta$-boundary ball, there exists a $(1+\eps)$-Lipschitz function $b: B_4(p) \to \setR_+$ satisfying:
	\begin{itemize}
		\item[(i)] $|b(x) - \dist_{\partial X}(x)| \le \eps \dist_{\partial X}(x)$ for any $x\in B_2(p)$;
		\item[(ii)] $b\in D_{\loc}(\Delta, B_4(p)\setminus \partial X )$ and 
		\begin{equation}\label{eq:almostsplit}
			\fint_{B_r(x)} |\nabla b - \nabla \dist_{\partial X}|^2 \di \haus^N + r^2 \fint_{B_r(x)} |\Delta b|^2 \di \haus^N \le \eps\, ,
		\end{equation}
		for any $x\in B_2(p)\setminus\partial X$, and $r=\dist_{\partial X}(x)/3$.
	\end{itemize}
\end{lemma}

\begin{proof} 
	We divide the proof into four steps. The first one aims at building a  partition of unity suitable for the geometry of our problem. In the second step we build harmonic approximations of $\dist_{\partial X}$ on balls with radius proportional to their distance from the boundary and prove good estimates as in the theory of Ricci limits (cf. \cite{CheegerColding96}). The sought function is obtained averaging the harmonic approximations of the distance obtained in Step 2 by the partition of unity built in Step 1. The third step is devoted to the proof of (i) while in the last step we obtain (ii).
	
	\smallskip
	
	{\bf Step 1.} There exist a family of functions $\set{\phi_k: B_8(p)\to \setR_+}_{k\in \setN}$ and a family of balls $\set{B_{r_k}(x_k)}_{k\in \setN} $ satisfying the following conditions:
	\begin{itemize}
		\item[(a)] $r_k:= \dist_{\partial X}(x_k)/8$, $B_4(p)\setminus \partial X \subset \bigcup_k B_{\frac{2}{3} r_k}(x_k)$;
		\item[(b)] if $B_{\frac{3}{2}r_{k_1}}(x_{k_1})\cap \ldots \cap B_{\frac{3}{2}r_{k_m}}(x_{k_m})\neq \emptyset$ then $m\le C(N)$ and $r_{k_i}\le C(N) r_{k_j}$ for any $i,j=1,\ldots,m$;
		\item[(c)] $\phi_k\in \Lip(X)\cap D(\Delta)$,  $\supp \phi_k \subset B_{r_k}(x_k)$ and
		\[
		\phi_k + r_k |\nabla \phi_k| + r_k^2 |\Delta \phi_k| \le C(N)\, ;
		\] 
		\item[(d)] $\sum_k \phi_k = 1 $ on $B_4(p)\setminus\partial X$.
	\end{itemize}
	Let us briefly explain how to build a family of balls satisfying (a) and (b). For any $\alpha\in \setN$ we cover $B_4(p)\cap \set{2^{-\alpha} \le \dist_{\partial X} \le 2^{-\alpha + 2}}$ using balls $\set{B_{2^{-\alpha-1}}(x_{\alpha,i}):\, i=1,\ldots ,m_{\alpha}}$ with $x_{\alpha,i}\in \set{2^{-\alpha} \le \dist_{\partial X} \le 2^{-\alpha + 2}}$ such that $\set{B_{2^{-\alpha - 3}}(x_{\alpha,i})\, i=1,\ldots ,m_{\alpha}}$ is a disjoint family. The verification of the fact that $\set{B_{2^{-\alpha-1}}(x_{\alpha,i}):\, \alpha=1,\ldots,m_\alpha,\, i\in \setN}$ satisfies (a) and (b) follows from the following simply verified observations: 
	\begin{itemize}
		\item $m_{\alpha} \le C(N)$ for any $\alpha\in\setN$;
		\item if $B_{\frac{3}{2}\cdot 2^{-\alpha-1}}(x_{\alpha,i})\cap B_{\frac{3}{2}\cdot 2^{-\beta-1}}(x_{\beta,j})\neq \emptyset$ then $|\alpha - \beta| \le 2$.
	\end{itemize}
	
	We build now the partition of unity $\set{\phi_k}$ satisfying (c) and (d) following a standard procedure. For any $k\in \setN$ we use \autoref{lem:good_cut-off} to get a nonnegative function $\eta_k$ satisfying $\eta_k=1$ on $B_{\frac{2}{3}r_k}(x_k)$ and $\eta_k=0$ on $X\setminus B_{ r_k}(x_k)$ along with the bound
	\[
	\eta_k + r_k |\nabla \eta_k| + r_k^2 |\Delta \eta_k| \le C(N)\, .
	\]
	Then we set 
	\[
	\phi_k:= \frac{\eta_k}{\sum_i \eta_i}\, .
	\]
	The verification of (c) and (d) is straightforward and builds upon the observation that $1\le \sum_i \eta_i \le C(N)$ on $B_4(p)\setminus\partial X$.
	
	\medskip
	
    {\bf Step 2.} If $\delta<\delta(N,\eps)$, $x\in B_4(p)$, $s=\dist_{\partial X}(x)/5$ then there exists a unique solution $b_{x,s}$ to the Dirichlet boundary value problem\footnote{The Dirichlet boundary condition below is understood as $b_{x,s}-\dist_{\partial X}\in H^{1,2}_0(B_s(x))$.}
    \begin{equation}\label{eq:poissonproblem}
    	\begin{cases}
    	\Delta b_{x,s} =0 & \text{on $B_s(x)$}\\
    	b_{x,s}=\dist_{\partial X} & \text{on $\partial B_s(x)$}
    	\end{cases}
    \end{equation}
	which satisfies moreover the estimates
	\begin{itemize}
		\item[(1)] $b_{x,s}>0$, $|\nabla b_{x,s}|\le C(N)$ and $|\nabla b_{x,s}|\le 1+\eps$ on $B_{s/2}(x)$;
		\item[(2)] $|b_{x,s} - \dist_{\partial X}| \le \eps s$ on $B_s(x)$;
		\item[(3)] $\fint_{B_s(x)} |\nabla b_{x,s} - \nabla \dist_{\partial X}|^2 \di \haus^N \le \eps$.
	\end{itemize}
    Existence and uniqueness of solutions to \eqref{eq:poissonproblem} follow from classical functional analytic arguments (cf. \cite[(4.5)]{Cheeger99} and \cite[(4.11)]{AmbrosioHonda18}) since $X\setminus B_{(1+\eps)s}(x)\neq\emptyset$. The positivity of $b_{x,s}$ in (1) is a consequence of the maximum principle, while the gradient bounds follow from \cite{Jiang14} (for the non sharp one) and \autoref{rm:sharpgradientbound}, for the sharp one given (iii).
    
    In order to verify (2) and (3) let us consider a point $q\in  B_4(p)\cap \partial X$ such that $\dist(x,q) \le 5s$ and notice that $B_{6s}(q)$ is a $\delta'$-boundary ball for $\delta\le \delta(N,\delta')$, thanks to \autoref{thm:improvedneckregion} (i). Since $B_s(x) \subset B_{6s}(q)$ we can scale the space of a factor $3/2 s$ and verify (2) and (3) in the special case $\dist_{\partial X}(x)/5 =s=2/3$.
    In order to do so we rely on the continuity of the harmonic replacement (see \cite{AmbrosioHonda18}) arguing by contradiction.
    
    	First we observe that $1/n$-boundary balls $B_2(q_n)$ converge to $B_2^{\setR^N_+}(0)$ as $n\to\infty$. Then we recall that \autoref{lemma:distancetodistance} yields uniform and $W^{1,2}$ convergence of the distance functions from the boundaries along the converging sequence, and on any converging sequence of balls $B_{2/3}(x_n)$.\footnote{This stronger statement can be checked arguing as in the case of balls centered at boundary points, given the uniform convergence of the distance functions.} To conclude we observe that on the half space the distance from the boundary is harmonic away from the boundary and local spectral convergence holds for any ball far away from the boundary. Therefore the harmonic replacements of the distance from the boundary verify:
    \begin{equation}
    \norm{b_{x_n,2/3}-\dist_{\partial X_n}}_{W^{1,2}(B_{2/3}(x_n))}\to 0\;\;\;\text{and}\;\;\; \norm{b_{x_n,2/3}-\dist_{\partial X_n}}_{L^\infty(B_{2/3}(x_n))}\to 0\, ,
    \end{equation}	
    as $n\to\infty$, yielding the sought estimates (2) and (3).
\smallskip

{\bf Step 3.} Let $\phi_k$ and $B_{r_k}(x_k)$ be as in Step 1. We set $b_k:= b_{x_k,2r_k}$, where $b_{x_k,2r_k}$ is obtained by Step 2, and we define
\[
b := \sum_k \phi_k b_k\, .
\]
Let us show that
$|b(x) - \dist_{\partial X}(x)| \le C(N) \eps \dist_{\partial X}(x)$ for any $x\in B_2(p)$.

First let us consider $x\in B_2(p)\setminus\partial X$. Using (1), (2), (b) and (d) we get
\begin{align*}
	|b(x) - \dist_{\partial X}(x)| & \le \sum_{k} \phi_k(x)|b_k(x) - \dist_{\partial X}(x)|
	\\& \le 2\eps \sum_{\{k:\, \phi_k(x)\neq 0\}} \phi_k(x) r_k
	\\& \le C(N) \eps \dist_{\partial X}(x)\, .
\end{align*}
Then we can estimate 
\begin{align*}
\abs{\nabla b}(x)\le& \sum_{k}\abs{\nabla\phi_k}(x)|b_k - \dist_{\partial X}|(x)+\sum_{k}\phi_k(x)\abs{\nabla b_k}(x)\\
\le& \eps C(N) \sum_{\{k:\, \phi_k(x)\neq 0\}}\abs{\nabla \phi_k}(x)\dist_{\partial X}(x)+ 1+\eps\\
\le& 1+C(N)\eps\, ,
\end{align*}
for $\haus^N$-a.e. $x\in B_4(p)$. Above we exploited the very definition of $b$, together with (a)--(d) and (1), (2).\\
This gradient estimate, together with the previous one, allows to infer that $b$ is Lipschitz once we set $b=0$ on $\partial X$.

\smallskip

{\bf Step 4.} We now verify that $b\in D_{\loc}(\Delta, B_4(p)\setminus \partial X )$ and 
\begin{equation}
\fint_{B_r(x)} |\nabla b - \nabla \dist_{\partial X}|^2 \di \haus + r^2 \fint_{B_r(x)} |\Delta b|^2 \di \haus^N \le \eps
\quad \forall\, x\in B_2(p),\, r=\dist_{\partial X}(x)/3\, .
\end{equation}

For $\haus^N$-a.e. $x\in B_2(p)$ one has
\begin{align}
\nonumber	|\nabla (b - \dist_{\partial X})|(x) &\le \sum_{\{k:\, \phi_k(x)\neq 0\}} |\nabla (\phi_k(b_k-\dist_{\partial X}))|(x) \\
	&\le C(N)\eps + \sum_{\{k:\, \phi_k(x)\neq 0\}} | \nabla b_k - \nabla \dist_{\partial X}|(x)\label{zzzz3}\, ,
\end{align}
where we have used (b), (c) and (2).

Let us now observe that on $B_4(p)\setminus\partial X$ it holds
\begin{equation}\label{eq:exprlapl}
\Delta b= \sum_k \Delta \phi_k b_k + 2\sum_k\nabla \phi_k\cdot \nabla b_k\, .
\end{equation}
The first sum in \eqref{eq:exprlapl} can be easily bounded by using (b), (c), (d) and (2)
\begin{equation}\label{zzzz4}
\begin{split}
	\abs{\sum_k \Delta \phi_k(x) b_k(x)} & = \abs{\sum_k \Delta \phi_k(x) (b_k(x)-\dist_{\partial X}(x))}
	\\& \le C(N)\sum_{\{k:\, \phi_k(x)\neq 0\}} r_k^{-2} \eps r_k
\le C(N) \eps \dist_{\partial X}^{-1}(x)\, .
\end{split}
\end{equation}
The estimate of the second sum in \eqref{eq:exprlapl} uses (b) and (d):
\begin{equation}\label{zzzz5}
    \begin{split}
	\abs{\sum_k \nabla \phi_k\cdot \nabla b_k(x) } &\le \sum_k |\nabla \phi_k|(x) |\nabla b_k - \nabla \dist_{\partial X}|(x)
	\\& \le C(N)\dist_{\partial X}^{-1}(x) \sum_{\{k:\, \phi_k(x)\neq 0\}} | \nabla b_k - \nabla \dist_{\partial X}|(x)\, .
	\end{split}
\end{equation}
By combining \eqref{zzzz3}, \eqref{zzzz4} and \eqref{zzzz5} we find out 
\begin{align*}
	\fint_{B_r(x)} |\nabla b - \nabla \dist_{\partial X}|^2& \di \haus^N + r^2  \fint_{B_r(x)} |\Delta b|^2 \di \haus^N
	\\& \le C(N) \eps + C(N) \fint_{B_s(x)}  \sum_{\{k: \phi_k(z)\neq 0\}} |\nabla b_k - \dist_{\partial X}|^2(z) \di \haus^{N}(z)\, ,
\end{align*}
which easily yields the sought conclusion as a consequence of (b) and (3).

\end{proof}

\begin{proof}[Proof of \autoref{thm:boundaryballshomeo}]
	For $\delta<\delta(N,\delta')$ we build a $\delta'$-splitting map $u: B_8(p) \to \setR^{N-1}$ with $u(p)=0$ and a function $b:B_8(p)\to \setR_+$ satisfying (i) and (ii) in \autoref{lemma:distanceapproximation} with $\delta'$ in place of $\eps$. We claim that $F := (u,b)$ verifies (i)--(iv).  
\smallskip	
	
	Let $x,y \in B_1(p)$ and set $r:=\dist(x,y)$. The inequality $|F(x) - F(y)|\le C(N)\dist(x,y)$ follows from the Lipschitz regularity of $u$ and $b$.
	
	Aiming at proving the inequality $|F(x) - F(y)| \ge (1-\eps) \dist(x,y)^{1+\eps}$
	we are going to argue as in the proof of \cite[Theorem 7.10]{CheegerJiangNaber18} (see also the proof of \eqref{eq:holderest}), relying on the transformation theorem. Since in this case the target is the half-space and not the Euclidean space, we need to study separately the two cases $r\le \dist_{\partial X}(x)/3$ and $r > \dist_{\partial X}(x)/3$.
	\smallskip
	
	Assume first $r\le \dist_{\partial X}(x)/3$. Let $q\in \partial X\cap B_1(p)$ such that $\dist(x,q)= \dist_{\partial X}(x)$. For $\delta<\delta(N,\delta')$ the ball $B_s(q)$ is a $\delta'$-boundary ball for any $s\in (0,8)$, by \autoref{thm:improvedneckregion} (i). The transformation theorem \autoref{prop:transformation} (see also the matrix growth estimate in \autoref{cor:growth}) applied to $u: B_{2\dist_{\partial X}(x)}(q) \to \setR^{N-1}$ (taking into account the fact that $u:B_2(q)\to\setR^{N-1}$ is a $\delta'$-splitting map) implies the existence of a matrix $T_x$ such that
	\begin{itemize}
		\item $T_x \circ u : B_{2\dist_{\partial X}(x)}(q) \to \setR^{N-1}$ is an $\eps'$-splitting map;
		\item $|T_x| \le (2\dist_{\partial X}(x))^{-\eps'}\le (6r)^{-\eps'}$,
	\end{itemize}
	whenever $\delta'<\delta'(N,\eps')$. Assume $\delta'\le \eps'$. Setting $v:=(T_x\circ u, b)$, thanks to \autoref{lemma:improvedepsio} and \eqref{eq:almostsplit}, we have
	\begin{equation}
		\sum_{\alpha, \beta=1}^N \fint_{B_s(x)} |\nabla v_\alpha \cdot \nabla v_\beta - \delta_{\alpha,\beta}| \di \haus^N + \sum_{\alpha=1}^N s^2 \fint_{B_s(x)} |\Delta v_{\alpha}|^2 \di \haus^N \le \eps''\, ,
	\end{equation}
	for $s:= \dist_{\partial X}(x)/3$ and $\eps'\le \eps'(N,\eps'')$.\\
	Applying again \autoref{prop:transformation} to $v: B_s(x)\to \setR^{N}$ taking into account that $B_t(x)$ is a $(N,\delta')$-symmetric ball for any $r\le t<3/2s$ (see \autoref{thm:improvedneckregion} (iii)) when $\delta \le \delta(N,\delta')$, we get the existence of a matrix $A_x$ such that
	\begin{itemize}
		\item $w:=A_x\circ v : B_{r}(x) \to \setR^N$ is a $\eps''$-almost splitting map;
		\item $|A_x|\le r^{-\eps''}$,
	\end{itemize}
	for $\eps' \le \eps'(N,\eps'')$. Hence, if $\eps''\le \eps''(N,\eps''')$, $w: B_r(x)\to \setR^N$ is a $\eps'''$-GH isometry thanks to \autoref{rm:epssplitimplepsisoimpr} and \autoref{rmk:almostdeltasplitting}, yielding that
	\[
	||w(x) - w(y)|-\dist(x,y)|\le \eps''' r = \eps''' \dist(x,y)\, .
	\]
	This implies in turn
	\[
	|F(x) - F(y)| \ge (1-\eps) \dist(x,y)^{1+\eps}\, ,
	\]
	being $w := B\circ F$, where $B$ is a matrix satisfying $|B|\le (1+\eps)r^{-\eps}$ for $\eps',\eps'',\eps'''$ small enough.
	\medskip

	Let us now deal with the case $r > \dist_{\partial X}(x)/3$. Let $q\in \partial X\cap B_1(p)$ such that $\dist(x,q)= \dist_{\partial X}(x)$. Notice that $B_{8r}(q)$ is a $\delta'$-boundary ball for $\delta\le \delta(N,\delta')$ as a consequence of \autoref{thm:improvedneckregion} (i). We apply the transformation theorem \autoref{prop:transformation} to get a matrix $A_x$, such that $|A_x|\le (8r)^{-\eps'}$ and $A_x\circ u:B_{8r}(q)\to\setR^{N-1}$ is an $\eps'$-splitting map. Relying now on \autoref{lemma:improvedepsio} and on \autoref{lemma:distanceapproximation} (i) we infer that, when $\delta\le\delta(N,\eps'')$, $v:=(A\circ u, b):B_{4r}(q)\to \setR_+^N$ is an $\eps''$-GH isometry. Notice that $y\in B_{4r}(q)$, hence
	\[
	||v(x) - v(y)|-\dist(x,y)|\le 4r \eps'' = 4\eps'' \dist(x,y)\, .
	\]
    Arguing as above we deduce $|F(x) - F(y)|\ge (1-\eps)\dist(x,y)^{1+\eps}$, for $\eps''$ small enough.
    
    \medskip
    
    The assertion (ii) follows from the fact that $\set{b=0}\cap B_4(p)=\partial X\cap B_4(p)$ and \autoref{thm:goodpar} (iv). 
    \smallskip
        	
    	Observe that $F:B_1(p)\setminus\partial X\to\setR^{N}_+\setminus \{x_N=0\}$ is an open mapping by invariance of the domain. Indeed, under our assumptions $B_1(p)\setminus\partial X$ is a homeomorphic to a topological manifold thanks to \autoref{thm:improvedneckregion} (iii) and Reifenberg's theorem. Being $F$ continuous and injective, to prove that $F:B_1(p)\to\setR^N_+$ is an homeomorphism with its image it is sufficient to prove (iv). If this is the case, to prove that the image of any open set in $B_1(p)$ is open in $\setR^N_+$ we just need to observe that on any boundary ball, up to an invertible matrix, the restriction of $F$ verifies the same properties that $F$ verifies on the ball of radius $1$.
    	\smallskip
    	
    	Let us therefore move to the verification of (iv). In order to do this it is sufficient to prove that 
    	\begin{equation}
    	F(\overline B_1(p)	\cap \set{\dist_{\partial X}\ge \delta})\supset B_{1-2\eps}(0)\cap \set{x_N > \delta(1 + \eps)}\, ,
    	\end{equation}
    	for any $\delta>0$. This claim can be verified arguing as we did in the proof of \autoref{thm:goodpar} (iv), relying once more on the invariance of the domain and on the fact that $F(x)\in\{x_N=0\}$ if and only if $x\in\partial X$.
 \end{proof}

\begin{remark}
Arguing as in the proof of \cite[Theorem 7.10]{CheegerJiangNaber18} (see also the proof of \autoref{thm:boundaryballshomeo} above) and relying on the transformation \autoref{prop:transformation} it is possible to obtain the following regularity result, which is worth pointing out.

If $(X,\dist,\haus^N)$ is a noncollapsed $\RCD(-(N-1),N)$ space and $x\in X$ is a regular point, then for any $0<\alpha<1$ there exists an open neighbourhood of $x$ which is $C^{\alpha}$-homeomorphic to an open subset of $\setR^N$, and the homeomorphism can be chosen with harmonic coordinate maps.

This observation gives in particular a positive answer to a question raised in \cite[Question 2.1]{Petrunin03} about existence of homeomorphisms from a neighbourhood of a regular point of an Alexandrov space with curvature bounded from below to an open set in $\setR^N$ with harmonic coordinates.  Notice that the regularity of the harmonic map cannot be improved to biLipschitz, due to the presence of singular points where harmonic maps do degenerate, see \cite[Example 2.14]{CheegerNaber15}.
\end{remark}


\begin{thebibliography}{GMS13}

\bibitem[AB03]{AlexanderBishop}
\textsc{S. Alexander, R. Bishop:}
\textit{$\mathcal{F}K$-convex functions on metric spaces,} 
Manuscripta Math. 110 (2003), no. 1, 115–133.

\bibitem[A19]{Ambrosio19}
\textsc{L. Ambrosio:}
\textit{Calculus, heat flow and curvature-dimension bounds in metric measure spaces,} Proceedings of the International Congress of Mathematicians—Rio de Janeiro 2018. Vol. I. Plenary lectures, 301–340, World Sci. Publ., Hackensack, NJ, 2018. 

\bibitem[AB18]{AmbrosioBertrand18}
\textsc{L. Ambrosio, J. Bertrand:}
\textit{DC calculus,} 
Math. Z. 288 (2018), no. 3-4, 1037–1080. 

\bibitem[ABS19]{AmbrosioBrueSemola19}
\textsc{L. Ambrosio, E. Bru\'e, D. Semola:}
\textit{Rigidity of the 1-Bakry-\'Emery inequality and sets of finite perimeter in RCD spaces}.
Geom. Funct. Anal., {\bf 19} (2019), n.4, 949-1001
%

%
\bibitem[AGMR15]{AmbrosioGigliMondinoRajala15}
\textsc{L. Ambrosio, N. Gigli, A. Mondino, T. Rajala:}
\textit{Riemannian Ricci curvature lower bounds in metric measure spaces with $\sigma$-finite measure.}
Trans. Amer. Math. Soc., {\bf 367} (2015), 4661–4701. 
%

%						
\bibitem[AGS14]{AmbrosioGigliSavare14}
\textsc{L. Ambrosio, N. Gigli, G. Savar\'e:}
\textit{ Metric measure spaces with Riemannian Ricci curvature bounded from below}.
Duke Math. J., {\bf 163} (2014), 1405--1490.   
%

%				
\bibitem[AH17]{AmbrosioHonda}
\textsc{L. Ambrosio, S. Honda}:
\textit{New stability results for sequences of metric measure spaces with uniform Ricci bounds from below}.
Measure Theory in Non-Smooth Spaces, De Gruyter Open, Warsaw, (2017), 1--51.		
%

\bibitem[AH18]{AmbrosioHonda18}
\textsc{L. Ambrosio, S. Honda:}
\textit{Local spectral convergence in $\RCD^*(K,N)$ spaces,} 
Nonlinear Anal. 177 (2018), part A, 1–23.



\bibitem[AHT18]{AmbrosioHondaTewodrose}
\textsc{L. Ambrosio, S. Honda, D. Tewodrose:}
\textit{Short-time behavior of the heat kernel and Weyl's law on $\RCD^*(K,N)$-spaces}.
Ann. Global Anal. Geom., {\bf 53} (2018), 97--119.
%		

%		
\bibitem[AMS14]{AmbrosioMondinoSavare14}
\textsc{L. Ambrosio, A. Mondino, G. Savar\'e}:
\textit{On the Bakry-\'Emery condition, the gradient estimates and
the local-to-global property of $\RCD^*(K,N)$ metric measure spaces}.
J.  Geom. Anal., \textbf{26} (2014), 1-33.
	
		
\bibitem[AMS15]{AmbrosioMondinoSavare15}
\textsc{L. Ambrosio, A. Mondino, G. Savar\'e}:
\textit{Nonlinear diffusion equations and curvature conditions in metric measure spaces}.
Mem. Amer. Math. Soc. 262 (2019), no. 1270, v+121 pp.
			
\bibitem[AT14]{AmbrosioTrevisan14}
\textsc{L. Ambrosio, D. Trevisan:}
\textit{Well posedness of Lagrangian flows and continuity equations in metric measure spaces.}
Anal. PDE, {\bf 7} (2014), 1179--1234.
					
					
\bibitem[A90]{Anderson90}
\textsc{M. Anderson:}
\textit{Convergence and rigidity of manifolds under Ricci curvature bounds.} 
Invent. Math. {\bf 102} (1990), no. 2, 429–445. 					
			



\bibitem[ABS19]{AntonelliBrueSemola19}
\textsc{G. Antonelli, E. Bruè, D. Semola:}	
\textit{Volume bounds for the quantitative singular strata of noncollapsed $\RCD$ metric measure spaces,}
Anal. Geom. Metr. Spaces, {\bf 7} 2019, no. 1.

\bibitem[BS10]{BacherSturm10}
\textsc{K. Bacher, K.-T. Sturm:}
\textit{Localization and tensorization properties of the curvature-dimension condition for metric measure spaces,}
J. Funct. Anal. 259 (2010), no. 1, 28–56. 


\bibitem[B85]{Beer}
\textsc{G. Beer:}
\textit{On convergence of closed sets in a metric space and distance functions,}
Bull. Austral. Math. Soc. 31 (1985), no. 3, 421–432. 

					
\bibitem[BPS19]{BruePasqualettoSemola19}		
\textsc{E. Bruè, E. Pasqualetto, D. Semola:}	
\textit{Rectifiability of the reduced boundary for sets of finite perimeter over $\RCD(K,N)$ spaces,}
Preprint arXiv:1909.00381.
		
\bibitem[BPS20]{BruePasqualettoSemola20}		
\textsc{E. Bruè, E. Pasqualetto, D. Semola:}	
\textit{Rectifiability of $\RCD(K,N)$ spaces via $\delta$-splitting maps,}
to appear on Ann. Sci. Acc. Fenn., preprint arXiv:2001.07911.					
								
\bibitem[BS19]{BrueSemola18}
\textsc{E. Bru\'e, D. Semola:}
\textit{Constancy of the dimension for $RCD(K,N)$ spaces via regularity of Lagrangian flows.}
Comm. Pure Appl. Math., 73 (2020), no. 6, 1141–1204. .


\bibitem[BS18a]{BrueSemola18b}
\textsc{E. Bru\'e, D. Semola:}
\textit{Regularity of Lagrangian flows over $\RCD^*(K,N)$ spaces,}
J. Reine Angew. Math. 765 (2020), 171–203.

\bibitem[BCM19]{BuffaComiMiranda}
\textsc{V. Buffa, G. Comi, M. Miranda jr:}
\textit{On BV functions and essentially bounded divergence-measure fields in metric spaces,}
Preprint arXiv:1906.07432. 



\bibitem[BBI01]{BuragoBuragoIvanov01}
\textsc{D. Burago, Y. Burago, S. Ivanov:}
\textit{A course in metric geometry.} 
Graduate Studies in Mathematics, 33. American Mathematical Society, Providence, RI, 2001. xiv+415 pp. 


\bibitem[CM16]{CavallettiMilman16}
\textsc{F. Cavalletti, E. Milman:}
\textit{The Globalization theorem for the Curvature-Dimension condition,}
Preprint arXiv:1612.07623. 


\bibitem[CM18]{CavallettiMondino18}
\textsc{F. Cavalletti, A. Mondino:}
\textit{New formulas for the Laplacian of distance functions and applications,}
to appear on Analysis \& PDE, preprint arXiv:1803.09687. 
				
\bibitem[C99]{Cheeger99}
\textsc{J. Cheeger:}
\textit{Differentiability of Lipschitz functions on metric measure spaces.}
Geom. Funct. Anal., {\bf 9} (1999), 428–517.			
%			
\bibitem[CC96]{CheegerColding96}
\textsc{J. Cheeger, T.-H. Colding:}
\textit{Lower bounds on Ricci curvature and the almost rigidity of warped products,}
Ann. of Math. (2), {\bf 144} (1996), 189--237.			
%			
\bibitem[CC97]{CheegerColding97}
\textsc{J. Cheeger, T.-H. Colding:}
\textit{On the structure of spaces with Ricci curvature bounded below. I,}
J. Differential Geom., {\bf 46} (1997), 406--480.
%		
					
\bibitem[CJN18]{CheegerJiangNaber18}
\textsc{J. Cheeger, W. Jiang, A. Naber:}
\textit{Rectifiability of singular sets in noncollapsed spaces with Ricci curvature bounded below,}
Preprint, arXiv:1805.07988 (2018).


\bibitem[CN13]{CheegerNaber13}
\textsc{J. Cheeger, A. Naber:}
\textit{Lower bounds on Ricci curvature and quantitative behavior of singular sets,}
Invent. Math. 191 (2013), no. 2, 321–339.


\bibitem[CN15]{CheegerNaber15}
\textsc{J. Cheeger, A. Naber:}
\textit{Regularity of Einstein manifolds and the codimension 4 conjecture,} 
Ann. of Math. (2) \textbf{182} (2015), no. 3, 1093–1165.


\bibitem[C97]{Colding97}
\textsc{T.-H. Colding:}
\textit{Ricci curvature and volume convergence,}
Ann. of Math. (2) \textbf{145} (1997), no. 3, 477–501. 



\bibitem[CoN12]{ColdingNaber13}
\textsc{T.-H. Colding, A. Naber:}
\textit{Sharp Hölder continuity of tangent cones for spaces with a lower Ricci curvature bound and applications,} 
Ann. of Math. (2) 176 (2012), no. 2, 1173–1229. 


\bibitem[CN13]{ColdingNaber13b}
\textsc{T.-H. Colding, A. Naber:}
\textit{Characterization of tangent cones of noncollapsed limits with lower Ricci bounds and applications,} 
Geom. Funct. Anal. \textbf{23} (2013), no. 1, 134–148. 


\bibitem[D20]{Deng20}
\textsc{Q. Deng:}
\textit{H\"older continuity of tangent cones on $\RCD(K,N)$ spaces and applications to non branching}
Preprint arXiv:2009.07956.

%			
\bibitem[DPG16]{DePhilippisGigli16}
\textsc{G. De Philippis, N. Gigli:}
\textit{From volume cone to metric cone in the nonsmooth setting.}
Geom. Funct. Anal., {\bf 26} (2016), 1526--1587.			

			
\bibitem[DPG18]{DePhilippisGigli18}
\textsc{G. De Philippis, N. Gigli:}
\textit{ Non-collapsed spaces with Ricci curvature bounded from below}.
J. Éc. polytech. Math., {\bf 5} (2018), 613–650.			
%
\bibitem[DPMR17]{DePhlippisMarcheseRindler17}
\textsc{G. De Philippis, A. Marchese, F. Rindler:}
\textit{On a conjecture of Cheeger,} 
Measure theory in non-smooth spaces, 145–155,
Partial Differ. Equ. Meas. Theory, De Gruyter Open, Warsaw, 2017. 
%
%			
\bibitem[EKS15]{ErbarKuwadaSturm15}
\textsc{M. Erbar, K. Kuwada, K.-T. Sturm:}
\textit{On the equivalence of the entropic curvature-dimension condition and Bochner’s inequality on metric measure spaces.}
Invent. Math., {\bf 201} (2015), 993--1071.		
%
%


%			
\bibitem[G13]{Gigli13}
\textsc{N. Gigli:}
\textit{The splitting theorem in non-smooth context}. Preprint arXiv:1302.5555.
		
\bibitem[G14]{Gigli14}
\textsc{N. Gigli:}
\textit{An overview of the proof of the splitting theorem in spaces with non-negative Ricci curvature.}
Anal. Geom. Metr. Spaces, {\bf 2} (2014), 169–213.



%
\bibitem[G15]{Gigli15}
\textsc{N. Gigli:}
\textit{On the differential structure of metric measure spaces and applications.}
Mem. Amer. Math. Soc., {\bf 236} (2015), vi--91.
%
\bibitem[G18]{Gigli18}
\textsc{N. Gigli:}
\textit{Nonsmooth differential geometry: an approach tailored for spaces with Ricci curvature bounded from below.}
Mem. Amer. Math. Soc., {\bf 251} (2018), v--161.



\bibitem[GH15]{GigliHan15a}
\textsc{N. Gigli, B.-X. Han:}
\textit{The continuity equation on metric measure spaces},
Calc. Var. Partial Differential Equations, {\bf 53}, (2015), 149--177. 


%	   	   
\bibitem[GMS15]{GigliMondinoSavare15}	   
\textsc{N. Gigli, A. Mondino, G. Savar\'e:}
\textit{Convergence of pointed non-compact metric measure spaces and stability of Ricci curvature bounds and heat flows.}	   
Proc. Lond. Math. Soc. (3), {\bf 111} (2015), 1071–1129.
%	   
\bibitem[GP16a]{GigliPasqualetto16a}
\textsc{N. Gigli, E. Pasqualetto:}
\textit{Behaviour of the reference measure on $\RCD$ spaces under charts,} 
to appear on Communications in Analysis and Geometry, preprint arXiv:1607.05188.	   
%	   	    
%	   	    
	    

\bibitem[GRS16]{GigliRajalaSturm16}
\textsc{N. Gigli, T. Rajala, K.-T. Sturm:}
\textit{Optimal maps and exponentiation on finite-dimensional spaces with Ricci curvature bounded from below,}
J. Geom. Anal. 26 (2016), no. 4, 2914–2929. 

\bibitem[GT19]{GigliTamanini19}
\textsc{N. Gigli, L. Tamanini:}
\textit{Second order differentiation formula on $\RCD^*(K,N)$ spaces,}
To appear on J. Eur. Math. Soc. (JEMS), preprint arXiv:1802.02463.  	   	    

\bibitem[GV13]{GolubVanLoan13}
\textsc{G.-H. Golub, C.-F. Van Loan:}
\textit{Matrix computations,} 
Fourth edition. Johns Hopkins Studies in the Mathematical Sciences. Johns Hopkins University Press, Baltimore, MD, 2013. xiv+756 pp.

\bibitem[H17]{Han17}
\textsc{B.-X. Han:}
\textit{Measure rigidity of synthetic lower Ricci curvature bound on Riemannian manifolds,}
Adv. Math. 373 (2020), 107327. 

\bibitem[H19]{Honda19}
\textsc{S. Honda:}
\textit{New differential operator and non-collapsed $\RCD$ spaces,}
Preprint arXiv:1905.00123, to appear on Geometry \& Topology. 

\bibitem[JLZ14]{JangLiZhang}
\textsc{R. Jiang, H. Li, H. Zhang:}
\textit{Heat Kernel Bounds on Metric Measure Spaces and Some Applications.}
Potential Anal., {\bf 44} (2016), 601--627.		    

\bibitem[J14]{Jiang14}
\textsc{R. Jiang:}
\textit{Cheeger-harmonic functions in metric measure spaces revisited,}
J. Funct. Anal. 266 (2014), no. 3, 1373–1394. 



\bibitem[JN16]{JiangNaber16}
\textsc{W. Jiang, A. Naber:}
\textit{$L^2$ curvature bounds on manifolds with bounded Ricci curvature,}
Preprint (2016), arXiv:1605.05583.


\bibitem[KM19]{KapovitchMondino19}
\textsc{V. Kapovitch, A. Mondino:}
\textit{On the topology and the boundary of $N$-dimensional $\RCD(K,N)$ spaces,}
To appear on Geometry \& Topology, Preprint arXiv:1907.02614.



\bibitem[K15]{Ketterer15}
\textsc{C. Ketterer:}
\textit{Cones over metric measure spaces and the maximal diameter theorem.}
J. Math. Pures Appl. (9) 103 (2015), no. 5, 1228–1275. 


%
\bibitem[KM18]{KellMondino18}
\textsc{M. Kell, A. Mondino:}
\textit{On the volume measure of non-smooth spaces with Ricci curvature bounded below,}
Ann. Sc. Norm. Super. Pisa Cl. Sci. (5) {\bf 18} (2018), no. 2, 593–610. 
%	    


\bibitem[K94]{Kirchheim94}
\textsc{B. Kirchheim:}
\textit{Rectifiable metric spaces: local structure and regularity of the Hausdorff measure,}
Proc. Amer. Math. Soc. 121 (1994), no. 1, 113–123. 


\bibitem[KL16]{KitabeppuLakzian16}
\textsc{Y. Kitabeppu, S. Lakzian:}
\textit{Characterization of low dimensional $\RCD^*(K,N)$ spaces,}
Anal. Geom. Metr. Spaces \textbf{4} (2016), no. 1, 187–215. 

\bibitem[K19]{Kitabeppu19}
\textsc{Y. Kitabeppu:}
\textit{A sufficient condition to a regular set being of positive measure on $\RCD$ spaces,}
Potential Anal. \textbf{51} (2019), no. 2, 179–196.


\bibitem[K02]{Kosolvsky02}
\textsc{N.-N. Kosovski:}
\textit{Gluing of Riemannian manifolds of curvature $\ge k$.}
Algebra i Analiz 14 (2002), no. 3, 140–157; translation in 
St. Petersburg Math. J. {\bf 14} (2003), no. 3, 467–478.


\bibitem[LN19]{LiNaber19}
\textsc{N. Li, A. Naber:}
\textit{Quantitative Estimates on the Singular Sets of Alexandrov Spaces,}
Peking Math. J. 3 (2020), no. 2, 203–234.

\bibitem[LV09]{LottVillani}
\textsc{J. Lott, C. Villani:}
\textit{Ricci curvature for metric-measure spaces via optimal transport.}
Ann. of Math. (2), {\bf 169} (2009), 903--991.


\bibitem[MN19]{MondinoNaber19}
\textsc{A. Mondino, A. Naber:}
\textit{Structure theory of metric measure spaces with lower Ricci curvature bounds,}
J. Eur. Math. Soc. (JEMS) 21 (2019), no. 6, 1809–1854. 




%
\bibitem[Mi03]{MirandaJr}
\textsc{M. Miranda Jr.:}
\textit {Functions of bounded variation on ``good'' metric spaces.} 
J. Math. Pures Appl., {\bf 82} (2003), 975--1004.




\bibitem[NV17]{NaberValtorta17}
\textsc{A. Naber, D. Valtorta:}
\textit{Rectifiable-Reifenberg and the regularity of stationary and minimizing harmonic maps,}
Ann. of Math. (2) 185 (2017), no. 1, 131–227. 

\bibitem[NV19]{NaberValtorta19}
\textsc{A. Naber, D. Valtorta:}
\textit{Energy identity for stationary Yang Mills,}
Invent. Math. 216 (2019), no. 3, 847–925. 

\bibitem[NT08]{NguyenTorres}
\textsc{N.-C. Phuc, M. Torres:}
\textit{Characterizations of the existence and removable singularities of divergence-measure vector fields,}
Indiana Univ. Math. J. 57 (2008), no. 4, 1573–1597. 




\bibitem[P16]{Perales16}
\textsc{R. Perales,}
\textit{Volumes and limits of manifolds with Ricci curvature and mean curvature bounds,}
Differential Geom. Appl. 48 (2016), 23–37. 


\bibitem[P91]{Perelman91}
\textsc{G. Perelman:}
\textit{A.D. Alexandrov’s spaces with curvatures bounded from below, II,} available at http://www.math.psu.edu/petrunin/papers/alexandrov/perelmanASWCBFB2+.pdf, 1991.

\bibitem[P03]{Petrunin03}
\textsc{A. Petrunin:}
\textit{Harmonic functions on Alexandrov spaces and their applications,}
Electron. Res. Announc. Amer. Math. Soc. 9 (2003), 135–141. 

\bibitem[P11]{Petrunin}
\textsc{A. Petrunin:}
\textit{Alexandrov meets Lott-Villani-Sturm.} 
M\"unster J. Math., {\bf 4} (2011), 53--64.


\bibitem[PS18]{ProfetaSturm18}
\textsc{A. Profeta, K.-T. Sturm:}
\textit{Heat flow with Dirichlet boundary conditions via optimal transport and gluing of metric measure spaces,}
Calc. Var. Partial Differential Equations {\bf 59} (2020), no. 4, 117, 34 pp. 



\bibitem[R60]{Reifenberg}
\textsc{E. R. Reifenberg:}
\textit{Solution of the Plateau Problem for m-dimensional surfaces of varying topological type,} 
Acta Math. 104 (1960), 1–92.


\bibitem[S14]{Savare14}
\textsc{G. Savar\'e:}
\textit{Self-improvement of the Bakry-\'Emery condition and Wasserstein contraction of the heat flow in $\RCD(K,\infty)$ metric measure spaces,} 
Discrete Contin. Dyn. Syst. 34 (2014), no. 4, 1641–1661. 


\bibitem[S12]{Schlicting12}
\textsc{A. Schlichting:} 
\textit{Gluing Riemannian manifolds with curvature operators at least k,} 
ArXiv e-prints, available at arXiv:1210.2957.


\bibitem[S19]{Schultz19}
\textsc{T. Schultz:}
\textit{On one-dimensionality of metric measure spaces,}
Proc. Amer. Math. Soc. {\bf 149} (2021), 383-396 



%
\bibitem[S06a]{Sturm06a}
\textsc{K.-T. Sturm:}
\textit{On the geometry of metric measure spaces I.}
Acta Math., {\bf 196} (2006), 65--131.
%		
\bibitem[S06b]{Sturm06b}
\textsc{K.-T. Sturm:}
\textit{On the geometry of metric measure spaces II.}
Acta Math., {\bf 196} (2006), 133--177.
%

\bibitem[V09]{Villani09}
\textsc{C. Villani:}
\textit{Optimal transport. Old and New.}
Grundlehren der Mathematischen Wissenschaften, {\bf 338}. Springer-Verlag Berlin, 2009.
%
\bibitem[VR08]{VonRenesse08}
\textsc{M.-K. Von Renesse:}
\textit{On local Poincaré via transportation.}
Math. Z., {\bf 259} (2008), 21--31.	



\bibitem[ZZ10]{ZhangZhu10}
\textsc{H.-C. Zhang, X.-P. Zhu:}
\textit{Ricci curvature on Alexandrov spaces and rigidity theorems,}
Comm. Anal. Geom. 18 (2010), no. 3, 503–553. 

			
\end{thebibliography}
\end{document}